\documentclass[leqno,12pt]{article} 
\usepackage{color}
\usepackage{amsmath,amsfonts,amssymb,amsthm,amscd}
\usepackage{hyperref}
\usepackage{wrapfig}
\font\cmssl=cmss10 at 12 pt

\usepackage{amsmath}
\usepackage{amssymb}
\usepackage{amsthm}
\usepackage{pb-diagram}
\usepackage{ascmac}
\usepackage{ulem}
\usepackage{here}
\allowdisplaybreaks

\newcommand{\n}{\nabla}

\newtheorem{theorem}{Theorem}[section]
\newtheorem{proposition}[theorem]{Proposition}
\newtheorem{lemma}[theorem]{Lemma}
\newtheorem{corollary}[theorem]{Corollary}
\newtheorem{remark}[theorem]{Remark}
\newtheorem{example}[theorem]{Example}
\newtheorem{definition}[theorem]{Definition}

\begin{document}

\title{Intrinsic characterization of projective special complex manifolds}
\author{%
Vicente Cort{\' e}s
and
Kazuyuki Hasegawa
}

\maketitle
\begin{abstract}
We define the notion of an $S^1$-bundle of projective special complex base type
and construct a conical special complex manifold from it. 
Consequently the base space of such an $S^{1}$-bundle can be realized as $\mathbb{C}^{\ast}$-quotient of 
a conical special complex manifold. 
As a corollary, we give an intrinsic characterization 
of a projective special complex manifold generalizing Mantegazza's characterization of 
a projective special K\"ahler manifold. 
Our characterization is in the language of c-projective structures. 
As an application, a non-trivial  
$S^1$-family of Obata-Ricci-flat hypercomplex structures (given by a generalization of the rigid c-map) on the tangent bundle of the total space of a 
$\mathbb{C}^*$-bundle over a complex manifold with certain kind of 
c-projective structure is constructed. Finally, we show that the quaternionic structure underlying any of these hypercomplex structures is in general not flat and that its 
flatness implies the vanishing of the c-projective Weyl tensor of  
the base of the $\mathbb{C}^*$-bundle. Conversely, any c-projectively flat complex manifold satisfying a cohomological integrality condition gives rise to a flat quaternionic structure. \\

\end{abstract}
\tableofcontents

\section{Introduction}
\setcounter{equation}{0}

In \cite{ACD}, (conical) special complex and 
projective special complex manifolds are introduced as generalizations of  
(conical) special K{\"a}hler and  projective special K{\"a}hler manifolds. 
A projective special complex manifold is a $\mathbb{C}^{\ast}$-quotient of a conical special complex manifold. 
The rigid c-map associates with a special K{\"a}hler manifold a  
hyperK{\"a}hler structure on its tangent bundle \cite{CFG}. Also the supergravity c-map, which was introduced in  \cite{FS}, is an assignment from a projective special 
K{\"a}hler manifold to a quaternionic K{\"a}hler manifold. This is a special case of the HK/QK-correspondence, see \cite{Haydys,ACM,Hitchin,ACDM}. 
In \cite{CH}, the rigid c-map, the supergravity c-map and the HK/QK-correspondence 
are generalized in the absence of a metric.  
In these generalizations, a (conical) special complex and 
a projective special complex manifold play an important role, see the diagram below. 

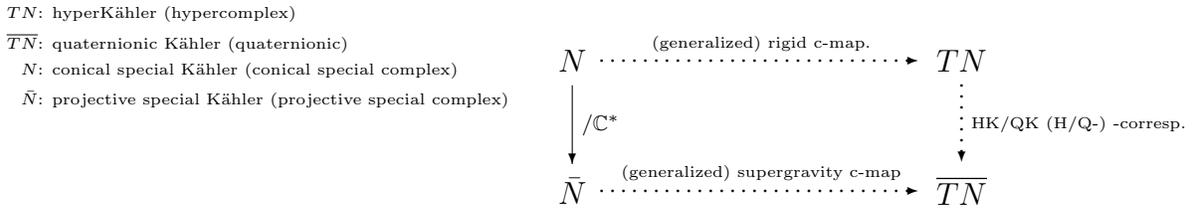
{
\begin{wrapfigure}{l}{0.5cm}
\vspace{-0.7cm}
\tiny
\begin{align*}
TN&\mbox{: hyperK{\"a}hler (hypercomplex)} \\
\overline{TN}&\mbox{: quaternionic K{\"a}hler (quaternionic) } \\
N&\mbox{: conical special K{\"a}hler (conical special complex)} \\ 
\bar{N}&\mbox{: projective special K{\"a}hler (projective special complex)}
\end{align*}
\end{wrapfigure}

\begin{eqnarray*}
\hspace{-3cm}
  \begin{diagram}
    \node[2]{N}     
    \arrow[3]{e,t,..}{\mbox{\tiny (generalized) rigid c-map.}}
    \arrow[1]{s,r}{/\mathbb{C}^{\ast}}
    \node[3]{TN}
    \arrow[1]{s,r,..}{\tiny \mbox{HK/QK (H/Q-) -corresp. }}\\
    \node[2]{\bar{N}}
    \arrow[3]{e,t,..}{\mbox{\tiny (generalized) supergravity c-map}}
    \node[3]{\overline{TN}}
  \end{diagram}
\end{eqnarray*}
}

A c-projective structure on a projective special complex manifold is canonically induced from 
the special complex manifold (\cite{CH}). 
Therefore the generalized supergravity c-map associates with this type of 
complex manifold of real dimension $2n$ with c-projective structure a quaternionic manifold of dimension $4n+4$ (\cite{CH}). 
In \cite{BC}, a construction of a quaternionic manifold  
from a complex manifold with a certain type of c-projective structure is given. 
This construction yields a quaternionic manifold from a complex manifold of half the dimension with 
c-projective structure and is known as the quaternionic Feix-Kaledin construction, which is a generalization 
of \cite{K, F}.

Therefore it is natural to consider the problem of characterizing those complex manifolds with c-projective structure that can be 
realized as $\mathbb{C}^{\ast}$-quotients of conical special complex manifolds. 
In \cite{CH} it is shown that the Weyl curvature of 
the canonically induced c-projective structure on a projective special 
complex manifold is of type $(1,1)$. Therefore 
any complex manifold with a c-projective structure whose Weyl curvature is {\it not} type $(1,1)$ 
cannot be realized as $\mathbb{C}^{\ast}$-quotient of a conical special complex manifold 
such that its canonical c-projective structure coincides with the given one. 
Based on these considerations, we introduce in Definition \ref{intrinsic_3} the notion of an $S^{1}$-bundle of 
projective special complex base type (abbreviated as PSCB type for simplicity)  
and show in Theorem \ref{char:thm} that any base manifold of an $S^{1}$-bundle of PSCB type 
can be realized as a 
$\mathbb{C}^{\ast}$-quotient of a conical special complex manifold, 
i.e., its base manifold is a projective special complex manifold.
The proof of Theorem \ref{char:thm} is performed by 
considering the product  of a $S^{1}$-bundle of PSCB type with $\mathbb{R}^{>0}$ and 
patching several geometric objects associated to an open covering with local triviality. 
Then a special connection can be constructed on the total space of the resulting $\mathbb{C}^{\ast}$-bundle 
over the given complex manifold. 

One of the simplest examples of Theorem \ref{char:thm} is to construct $\mathbb{C}^{n+1} \backslash \{ 0 \}$ with 
the usual $\mathbb{C}^{\ast}$-action from the the Hopf fibration $S^{2n+1} \to \mathbb{C}P^{n}$. 
More generally, a c-projectively flat manifold can be realized as a projective special 
complex manifold under certain assumptions, each of which is expressed 
in terms of c-projective geometric objects  (Corollary \ref{cflat_ex}). 
In particular, for the case of $\mathbb{C}P^{n}$ ($n \geq 2$) with the standard c-projective structure, 
this result recovers the Hopf fibration 
(Example \ref{ex_cp_n}). 
If there is a trivial $S^1$-bundle of type PSCB over the given complex manifold with a c-projective structure, 
then it is a projective special complex manifold, that is,   
Corollary \ref{app_1} gives an intrinsic characterization of projective special complex manifolds  
in terms of conditions 
all of which are described by objects on the complex manifold with c-projective structure. 
Since the c-projective structure is unique for the lowest dimensional case, 
we give another definition of a projective special complex base type for that case 
 (Definition \ref{intrinsic_2}) and obtain Theorem \ref{char:thm_two_dim} 
 in a similar way. As applications of Theorems \ref{char:thm} and \ref{char:thm_two_dim}, 
Section \ref{low_ex} provides examples of four- and two-dimensional projective special complex manifolds with a parallel Ricci tensor. 
In particular, we present an example that does not occur in projective special K{\"a}hler manifolds.

As we stated in the above paragraph, 
there is an assignment from a projective special K{\"a}hler manifold to 
a quaternionic K{\"a}hler manifold. Projective special K{\"a}hler manifolds are a special class of special complex manifolds. In \cite{M}, an intrinsic characterization of projective special K{\"a}hler manifolds is given, which we recover as Corollary \ref{proj_kaehler_char}.  
We also refer to \cite{MS1}, which provides an intrinsic characterization of a simply connected K{\"a}hler group with an exact Kähler form that admits a compatible projective special K{\"a}hler structure. 
In this case, since the Kähler group is parallelizable and the associated 
$S^{1}$-bundle is trivial, the characterization is described in terms of matrix-valued one-forms. This characterization can be recovered within our framework (Remark \ref{gr_ch_ms}).

Applying Theorem \ref{char:thm} to the generalized rigid c-map yields 
a non-trivial $(\mathbb{R}/\pi \mathbb{Z})$-family of hypercomplex structures on the space of the c-map, 
originating from a certain kind of complex manifold with c-projective structure (Theorem \ref{app_cmap} and Corollary \ref{app3}). 
Additionally, a local example of such a manifold is given by using the local expression of a conical special complex manifold in terms of a holomorphic $1$-form.
The Ricci curvature of the Obata connection of these hypercomplex structures vanishes (\cite{CH}). 
Since the Obata connection is unique for a hypercomplex manifold (\cite{O}), 
all of its invariants, such as the curvature tensor, Ricci curvature, geodesics, and holonomy group, are also invariants of the hypercomplex structure. 
It is known that the Ricci curvature of the Obata connection on 
a hypercomplex manifold of dimension $4n$ ($n \geq 2$)
vanishes 
if and only if its restricted holonomy group is contained in $\mathrm{SL}(n,\mathbb{H})$. 
Therefore, a hypercomplex manifold with a Ricci-flat Obata connection is of particular interest in hypercomplex geometry.

In the last section of the paper we study the quaternionic structure underlying a Ricci-flat hypercomplex structure constructed by 
the generalized rigid c-map. We show that the quaternionic structure can only be flat if the c-projective Weyl tensor of the initial complex manifold is zero, see 
Theorem \ref{flatness_q_c}. In the setting of projective special K\"ahler manifolds 
this reduces to the case of complex hyperbolic space.\\

\section{Preliminaries}
\setcounter{equation}{0}

Throughout this paper, 
all manifolds are assumed to be smooth and without boundary 
and maps are assumed to be smooth 
unless otherwise mentioned. 
The space of sections of a vector bundle $E\rightarrow M$ 
is denoted by $\Gamma(E)$. 

\subsection{Projective special complex manifolds}

In this subsection, we recall the definition of a (conical/projective) special complex manifold 
(see \cite{ACD, CH}). See also \cite{Freed, MS} for the more specific case of a special K{\"a}hler manifold.

\begin{definition}\label{scm:def}
A {\cmssl special complex manifold} $(M,J,\nabla)$ is a complex manifold $(M,J)$ endowed with a 
torsion-free flat connection $\nabla$ such that the $(1,1)$-tensor field $\nabla J$ is symmetric.
Such a connection $\nabla$ is called {\cmssl special}.
A {\cmssl conical special complex manifold}  $(M,J,\nabla, \xi)$ is a special complex manifold $(M,J,\nabla)$
endowed with a nowhere-vanishing vector field $\xi$ such that 
\begin{itemize}
\item $\nabla \xi=\mathrm{Id}$ and 
\item $L_{\xi} J=0$ or, equivalently, $\nabla_\xi J =0$.
\end{itemize}
\end{definition}

We recall \cite{CH} that this definition implies $L_{J\xi }J=0$ and we set $A:=\nabla J$. 
If $A=0$, then $(M,J,\nabla)$ is called {\cmssl trivial}. 
We know that the connection 
\begin{align}\label{connection}
D := \nabla -\frac12 J A 
\end{align}
is torsion-free and complex. Its curvature is given by 
\begin{equation} 
\label{RD:eq}R^D= -\frac14 [A,A]. 
\end{equation}

\begin{lemma}\label{closed_A}
$d^{D}A=0$.
\end{lemma}

\begin{proof}
We have   
\[ (d^DA)(X,Y) = (d^\nabla A)(X,Y) -\frac12 [J A_X,A_Y] +\frac12 [JA_Y,A_X].\]
The first term vanishes, since  $(d^\nabla A)(X,Y) = R^\nabla (X,Y)J=0$ and the other two terms cancel, since 
\[ [J A_X,A_Y] = J \{ A_X, A_Y\} \]
is symmetric. 
\end{proof}

Let $(M,J)$ be a complex manifold and consider  the following spaces of connections  
on $(M,J)$: 
\begin{align*}
{\cal S}_{(M,J)} :=&\{ \nabla \mid \nabla \,\, \mbox{is  a special connection} \} \\
{\cal C}_{(M,J)} :=&\{ (D,S) \mid D \,\, \mbox{is a torsion-free complex connection,} \\  
&\hspace{1.7cm}S \,\, \mbox{is a symmetric $(1,2)$-tensor and anti-commutes with}\,\, J, \,\, \\
&\hspace{1.7cm}R^{D}=-[S,S], \,\, d^{D}S=0 \}.
\end{align*}

\begin{lemma}\label{char_sc} There is a natural bijection 
${\cal S}_{(M,J)} \cong {\cal C}_{(M,J)}$.  
\end{lemma}
 
\begin{proof}
If $\nabla \in  {\cal S}_{(M,J)}$, put $S=-\frac{1}{2}JA$. Then we see that  
$D=\nabla +S$ is a torsion-free complex connection and 
$S$ is symmetric and anti-commute with $J$.
Moreover we have 
\[ d^{D}S=-\frac{1}{2} \left( (DJ) \wedge A+J(d^{D}A) \right)=0 \]
by Lemma \ref{closed_A}. 
Finally we see from  (\ref{RD:eq}) that $R^{D}=-[S,S]$. 
Conversely, take $(D,S) \in {\cal C}_{(M,J)}$ and set $\nabla=D-S$. 
Because $S$ is symmetric, $\nabla$ is torsion-free. 
Since $\nabla J=-2JS$, 
we see that $\nabla J$ is symmetric. 
From $R^{D}=R^{\nabla}+d^{D}S-[S,S]$, $\nabla$ is flat. 
Therefore $\nabla$ is a special connection. 
\end{proof}

Recall that $A_\xi = A_{J\xi}=0$. 
We also have that $L_\xi A =0$, since $L_\xi \nabla=0$ and $L_\xi J =0$.

\begin{lemma}\label{Lie_A}
$L_{J\xi}\nabla = A$, $L_{J\xi}A=-2JA$ and $L_{J\xi}(JA) = 2 A$. 
\end{lemma}
\begin{proof}
We compute $L_{J\xi}\nabla$ on a pair of $\nabla$-parallel vector fields $X, Y$:
\begin{align*}
(L_{J\xi}\nabla)_XY &= [J\xi ,\nabla_XY]-\nabla_{[J\xi , X]}Y -\nabla_X[J\xi ,Y]\\
&=  0-0 + \nabla_X\nabla_Y(J\xi) = \nabla_X(\nabla_YJ)\xi +\nabla_X(J\nabla_Y\xi ) = 
(\nabla_XJ)Y = A_XY,
\end{align*}
where we have used that $A_Y\xi = A_\xi Y=0$. This implies
\[ L_{J\xi}A =   L_{J\xi}(\nabla J) = [L_{J\xi}\nabla ,J] = [A,J] =2AJ = -2JA\]
and $L_{J\xi}(JA) = J (-2JA)= 2A$. 
\end{proof}

\begin{definition}
If the holomorphic vector field 
$\xi-\sqrt{-1} J \xi$ generates a principal $\mathbb{C}^{\ast}$-action with corresponding 
principal $\mathbb{C}^{\ast}$-bundle 
$\pi:(M,J,\nabla,\xi) \to (\bar{M},\bar{J})$, then we call 
the complex base manifold  $(\bar{M},\bar{J})$
a {\cmssl projective special complex manifold}. \end{definition}

\subsection{Realization of projective special complex manifolds as submanifolds of 
projective space}\label{realizations}

In this subsection, we give an alternative proof of an extrinsic realization theorem of 
special complex and projective special complex manifolds, originally proved in \cite{ACD}. 
The alternative proof for a special complex manifold 
is carried out from the viewpoint of affine immersions. 
See \cite{NS} for affine immersions.
This leads to a useful expression for a conical special complex manifold in terms of a holomorphic one-form,  
which will be used in Section \ref{loc_ex_TU}. 
Let $(M,J)$ be a simply connected complex manifold of real dimension $\dim M=2n+2$.

\begin{proposition}\label{realization:prop}
For all $(D,S) \in {\cal C}_{(M,J)}$, there exists a holomorphic affine immersion 
$f:(M,J,D) \to (\mathbb{C}^{2n+2}, \tilde{J},\tilde{D})$ with the transversal bundle $\varepsilon( f_*TM)$ 
and the affine fundamental form $\varepsilon f_*S$, where 
$\tilde{D}$ is the standard flat torsion-free connection on $\mathbb{C}^{2n+2}$, 
$\tilde{J}$ is the standard complex structure on $\mathbb{C}^{2n+2}$ and 
$\varepsilon$ is a field of endomorphisms such that $\tilde{D} \varepsilon=0$, $\varepsilon^{2}=\mathrm{Id}$ and 
$\varepsilon \tilde{J}=-\tilde{J} \varepsilon$. 
\end{proposition}
\begin{proof}
We consider the vector bundle 
$E=TM \oplus TM$ over $M$ and define a connection $D^{E}$ and endomorphisms $J^{E}$, 
$\varepsilon^{E}$ by 
\begin{align*}
 D^{E}_{X}(Y,Z) &=(D_{X}Y+S_{X}Z,S_{X}Y+D_{X}Z), \\
 J^{E}(X,Y)&=(JX,-JY), \,\,\, \varepsilon^{E}(X,Y)=(Y,X),  
\end{align*}
for $X$, $Y$, $Z \in \Gamma(TM)$. From the assumption $(D,S) \in {\cal C}_{(M,J)}$ one can easily check that $D^{E}$ is flat
and that the equations $D^{E}J^{E}=0$ and $D^{E} \varepsilon^{E}=0$ hold. 
There exists a holomorphic affine immersion 
$f:(M,J,D) \to (\mathbb{C}^{2n+2},\tilde{J},\tilde{D})$ 
and an endomorphism $\varepsilon$ on $\mathbb{C}^{2n+2}$ 
such that $\tilde{D} \varepsilon=0$, $\varepsilon^{2}=\mathrm{Id}$, 
$\varepsilon \tilde{J}=-\tilde{J} \varepsilon$ whose transversal bundle is $\varepsilon(f_*TM)$ 
by the existence theorem for holomorphic affine immersions. In fact, $R^{D}=-[S,S]$ is the Gauss (and Ricci) equation and $d^{D}S=0$ is the Codazzi equation. It is easy to see that the affine fundamental form is 
$\varepsilon f_{\ast} S$ by the definition of $D^{E}$. Note that the 
tangent bundle corresponds to first summand of $E=TM \oplus TM$ and the transversal bundle corresponds to the second one. The holomorphicity for the immersion $f$ follows from the $J^{E}$-invariance of 
the first summand of $E$. 
\end{proof}

By Proposition \ref{realization:prop} and Lemma \ref{char_sc}, we recover the result that any simply connected special complex manifold $(M,J,\nabla)$ 
admits a holomorphic totally complex immersion into $\mathbb{C}^{2n+2}$ 
(see Theorem 4 in \cite{ACD}). Here totally complex means that $f_{\ast}(TM)$ is transversal to the fixed-point  
set of $\varepsilon$.  We note that the fixed-point set of $\varepsilon$ is a (real) transversal bundle different from the (complex) transversal bundle $\varepsilon(f_*TM)$ of the affine immersion. 
Any holomorphic one-form $f(=\sum F_{i} d z_{i})$
on a domain $U \subset \mathbb{C}^{n+1}$ can be considered as a holomorphic immersion $f:U \to T^{\ast} \mathbb{C}^{n+1} \cong \mathbb{C}^{2n+2}$. 
It is known that the real matrix 
$\mathrm{Im} \frac{\partial F_{i}}{\partial z_{j}}$ is invertible, which is referred to as 
{\cmssl regular} if and only if  the corresponding immersion $f$ is totally complex 
(\cite[Lemma 3]{ACD}). 
Therefore any special complex manifold can be locally obtained by a regular holomorphic one-form 
(\cite[Corollary 3]{ACD}).

Next, we consider a projective special complex manifold $(\bar{M},\bar{J})$, with 
$M$ simply connected and $\pi: (M,J,\nabla,\xi) \to (\bar{M},\bar{J})$
the corresponding principal $\mathbb{C}^*$-bundle. 
As discussed above, we have a holomorphic affine immersion 
$f:(M,J,D) \to (\mathbb{C}^{2n+2},\tilde{J},\tilde{D})$. In the next proposition 
we recover of \cite[Remark 3]{ACD}.  

\begin{proposition}\label{realization:projective}
Up to a translation in $\mathbb{C}^{2n+2}$, the immersion $f$ is conical, that is 
$f: M \to \mathbb{C}^{2n+2} \backslash \{ 0 \}$ and 
$f(\alpha \cdot x)=\alpha f(x)$ for all $\alpha \in \mathbb{C}^{\ast}$ and $x \in M$. As a consequence,
it induces a holomorphic immersion of $(\bar{M},\bar{J})$  into $\mathbb{C}P^{2n+1}$.
\end{proposition}
\begin{proof}
Since $M$ is a conical special complex manifold, 
the Gauss formula for $f$ gives 
\[ \tilde{D}_{X} f_{\ast} \xi = f_{\ast}(D_{X}\xi)+\varepsilon f_{\ast} S_{X}\xi=f_{\ast}X \quad \mbox{and } \quad
    \tilde{D}_{X} (\tilde{J} f_{\ast} \xi )= \tilde{J}  \tilde{D}_{X}f_{\ast} \xi=\tilde{J} f_{\ast}X. 
\]  
The $\mathbb{C}^{\ast}$-multiplication on $\mathbb{C}^{2n+2} \backslash \{ 0 \}$ is induced by 
$\tilde{\xi}-\sqrt{-1}\tilde{J} \tilde{\xi}$, where $\tilde \xi$ is the Euler field (position vector field), which satisfies
$\tilde{D}_X\tilde \xi = X$ and $\tilde{D}_X(\tilde J \tilde \xi )= \tilde J X$ for all vector fields $X$ on $\mathbb{C}^{2n+2}$. Comparing the above 
formulas, we see that 
\[ f_{\ast} \xi_{x}=\tilde{\xi}_{f(x)}\quad and \quad
f_{\ast} J \xi_{x}=(\tilde{J}\tilde{\xi})_{f(x)}\] 
for all $x\in M$, up to changing the immersion $f$ by a translation. 
Now it follows that the immersion does not go through the origin (since $\xi$ does not vanish), i.e.\ $f: M \to \mathbb{C}^{2n+2} \backslash \{ 0 \}$. 
These observations  mean  that $f(\alpha \cdot x)=\alpha f(x)$ 
for all $\alpha \in \mathbb{C}^{\ast}$, $x \in M$, and that $f$ induces an immersion into a 
projective space. 
Here $(\alpha ,x)\mapsto \alpha \cdot x$ denotes the principal $\mathbb{C}^{\ast}$-action on $M$. 
\end{proof}

In Proposition \ref{realization:projective}, the condition of the immersion $f$ being conical corresponds to 
the one-form $\sum F_{i} d z_{i}$ having homogeneity of degree one, i.e. 
$F_{i}(\lambda z)=\lambda F_{i}( z)$ for all $z \in U$ and all $\lambda$ near $1$ in $\mathbb{C}^{\ast}$. 
Thus any conical special complex manifold can be locally obtained by a regular holomorphic one-form $\sum F_{i} d z_{i}$ with homogeneity of degree one \cite[Corollary 5]{ACD}.

\subsection{The canonical c-projective structure}\label{subsec23}

In this subsection, we analyze the canonical c-projective structure of a projective special complex manifold introduced in \cite{CH}.  
A smooth curve $t\mapsto c(t)$ on a complex manifold $(M,J)$ 
is called {\cmssl $J$-planar with
respect to a connection $D$ if $D_{c'}c' \in \langle  c', Jc' \rangle$.} 
We say that torsion-free complex connections $D^{1}$ and $D^{2}$ on a 
complex manifold $(M,J)$ are {\cmssl c-projectively related} (\cite{Ishi, CEMN})   
if they have the same $J$-planar curves. It is known that 
$D^{1}$ and $D^{2}$ are c-projectively related
if and only if there exists a one-form $\theta$ on $M$ such that
\begin{align}\label{def_cpro}
D^{1}_{X}Y
=&D^{2}_{X} Y
+\theta(X)Y+\theta(Y)X
-\theta(JX)JY-\theta(JY)JX
\end{align}
for $X$, $Y \in \Gamma(TM)$.
This defines an equivalence relation on the space of torsion-free complex connections on $M$.
The equivalence classes are called {\cmssl c-projective structures}.
For a complex connection $D$, 
we define a $(0,2)$-tensor $P^{D}$ on $M$ by 
\begin{align}\label{P-tensor:eq}
P^{D} =& \frac{1}{n+1} \left(
Ric^{D}+\frac{1}{n-1} 
\left( (Ric^{D})^{s} -J^{*}(Ric^{D})^{s} \right) \right), 
\end{align}
which is called the {\cmssl Rho tensor}, where $2n=\dim M \geq 4$, 
$Ric^{D}$ is the Ricci tensor of $D$ and 
$(Ric^{D})^{s}$ is its symmetrization. 
For a $(0,2)$-tensor $l$ and a $(1,1)$-tensor $K$, we define $l \wedge K$ by 
\[  (l \wedge K)_{X,Y}Z=l(X,Z)KY-l(Y,Z)KX \]
for tangent vectors $X$, $Y$ and $Z$. 
For any $(0,2)$-tensor $k$ on a complex manifold with a complex structure $J$, 
define the $(0,2)$-tensor $k_{J}$ by 
\[ k_{J}(X,Y):=k(X,JY) \]
for tangent vectors $X$ and $Y$. 
The c-projective Weyl curvature $W^{[D]}$ of $[D]$ is given by the following coordinate-free version of the formula 
(24) in \cite{CEMN}: 
\begin{align}\label{W_cur}
W^{[D]} =& R^{D} +(P^{D})^{a} \otimes Id 
       -(P^{D}_{J})^{a} \otimes J
+ \frac{1}{2}P^{D} \wedge Id  - \frac{1}{2}P^{D}_{J} \wedge J, 
\end{align}
where $(\, \cdot \,)^{a}$ denotes anti-symmetrization. 
This expression immediately implies the following.
\begin{proposition} \label{comm_W_J:prop}The c-projective Weyl curvature preserves the complex structure: $[W^{[D]} , J]=0$. 
\end{proposition}

\begin{lemma}\label{pho_ricci}
$(P^{D}_{J})^{a} =\frac{1}{n+1} (Ric^{D}_{J})^{a}$. 
\end{lemma}
\begin{proof}
It follows from a straightforward calculation.  
\end{proof}

\begin{lemma}\label{rho_inv}
For a  complex manifold $(M,J,[D])$ 
with a c-projective structure $[D]$, 
if the $2$-form $(P^{D}_{J})^{a}$ is closed, then 
the de Rham cohomology class $[(P^{D}_{J})^{a} ]=[\frac{1}{n+1} (Ric^{D}_{J})^{a}]$ is independent of the choice 
of a connection in $[D]$. 
\end{lemma}
\begin{proof}
If two complex connections $D^{1}$ and 
$D^{2}$ are c-projective related as in 
\eqref{def_cpro}, then 
\begin{align*}
P^{D^{1}} = & P^{D^{2}} - 2(D^{2} \theta) 
+\theta \otimes \theta-J^{\ast}\theta \otimes J^{\ast}\theta. 
\end{align*}
Here we refer to \cite[Proposition 2.11]{CEMN}. 
Using this equation, we have 
\begin{align*}
P^{D^{1}}_{J} = & P^{D^{2}}_{J} - 2 J^{\ast}(D^{2} \theta) 
+\theta \otimes J^{\ast}\theta+J^{\ast}\theta \otimes \theta. 
\end{align*}
Thus it follows that 
\begin{align*}
(P^{D^{1}}_{J})^{a} =& (P^{D^{2}}_{J})^{a} - 2 ((D^{2} J^{\ast}\theta)^{a} 
                              +\frac{1}{2} ( \theta \otimes J^{\ast}\theta - J^{\ast}\theta \otimes \theta) 
                              +\frac{1}{2} ( J^{\ast}\theta \otimes \theta -  \theta \otimes J^{\ast}\theta ) \\
                            =&  (P^{D^{2}}_{J})^{a} - d(J^{\ast} \theta)
\end{align*}
since $D^{2}J=0$. 
\end{proof}

From now on and for the rest of this subsection, let $(M,J,\n , \xi )$ be a conical special complex manifold of $\dim M=2n+2$ 
which is the total space of a 
(holomorphic) principal $\mathbb{C}^*$-bundle $\pi : M \rightarrow \bar M$, the base of which is a projective special complex manifold $\bar{M}$. 
The fundamental vector field generated by $1$ is denoted by $\xi$. 
Note that $J \xi$ is the fundamental vector field generated by $\sqrt{-1}$. 
Let $\omega$ be a connection form on $M$. 
We assume that the horizontal distribution $\mathcal H := \mathrm{Ker}\, \omega $ is 
$J$-invariant or, equivalently, that $\omega$ is of type $(1,0)$
 (but not necessarily holomorphic). 
Using $\omega$ we can project the connection 
$D$ on $M$ to a connection $D^{(\omega)}$ on 
$\bar{M}$, which is complex with respect to $\bar{J}$. 

\begin{theorem}[\cite{CH}]\label{can_c_pro}
Let $\omega$ be a principal connection in the smooth principal bundle $\pi:M\rightarrow \bar M$
such that $\mathrm{Ker}\, \omega$ is $J$-invariant. 
The connection $D^{(\omega)}$ defines a 
c-projective structure $[D^{(\omega)}]$ on $\bar M$. 
The class $[D^{(\omega)}]$ is independent of $\omega$. 
\end{theorem}

The c-projective structure $[D^{(\omega)}]$ in Theorem \ref{can_c_pro} is  called the {\cmssl canonical} 
c-projective structure of the projective special complex manifold $\bar M$ and is denoted 
by $\bar{\mathcal{P}}(\pi)$.  
Note that $\pi:(M,D) \to (\bar{M}, D^{(\omega)})$
is an affine submersion with the horizontal distribution $\mathcal{H}=\mathrm{Ker}\, \omega$ in the sense of \cite{AH} .  
The horizontal lift of $X \in T\bar{M}$ is denoted by $\tilde{X}$. 
Let $h:TM \to \mathcal{H}$ and $v:TM \to \mathcal{V}$ be the projections with respect to 
the decomposition $TM= \mathcal{H} \oplus \mathcal{V}$, where $\mathcal{V}=\mathrm{Ker}\, \pi_{\ast}$. 
We define the fundamental tensors $\mathcal{A}^{D}$ and $\mathcal{T}^{D}$ by 
\[ \mathcal{A}^{D}_{E} F=v (D_{hE} hF) + h(D_{hE} vF) \]
and
\[ \mathcal{T}^{D}_{E} F=h (D_{vE} vF) + v(D_{vE} hF) \]
for $E$, $F \in \Gamma(T M)$. 
We have $ \mathcal{T}^{D}=0$ and $\mathcal{A}^{D}_{X}{\xi} = X, \,\,\,  \mathcal{A}^{D}_{X} (J\xi ) =JX$
for any horizontal vector $X$. 
Since $D$ and the projections $v$, $h$ are $\mathbb{C}^{\ast}$-invariant, 
$\mathcal{A}^{D}$ is as well, and hence, 
\begin{align}\label{fund_A}
\mathcal{A}^{D}_{ \tilde{X}} \tilde{Y} 
&=\bar{a}(X,Y)\xi - \bar{a}_{\bar{J}}(X,Y) J\xi 
\end{align}
for a $(0,2)$-tensor $\bar{a}$ on $\bar{M}$. 
See \cite[Lemmas 7.1, 7.7]{CH}.

\begin{definition}\label{fund_aa}
The tensor $\bar{a}$ defined in \eqref{fund_A} 
is called a {\cmssl fundamental $(0,2)$-tensor} on a projective special complex manifold $\bar{M}$. 
\end{definition}

The fundamental $(0,2)$-tensor $\bar{a}$ will be given as \eqref{bar_a} if $\dim \bar{M} \geq 4$. 
If $\bar{M}$ is a projective special K{\"a}hler manifold, then the fundamental $(0,2)$-tensor coincides 
with the K{\"a}hler metric (see Section \ref{SpKah}).

\begin{lemma}\label{A_eq}
The fundamental tensor $\mathcal{A}^{D}$ satisfies 
\begin{align*}
v((D_{\tilde{X}} \mathcal{A}^{D})(\tilde{Y},\tilde{Z}))
&=(D^{(\omega)}_{X} \bar{a})(Y,Z) \xi
-(D^{(\omega)}_{X} \bar{a}_{\bar{J}})(Y,Z) J \xi 
\end{align*}
for tangent vectors $X$, $Y$ and $Z$ on $\bar{M}$. 
\end{lemma}

\begin{proof}
By $D^{(\omega)} \bar{J}=0$, we have the conclusion. 
\end{proof}

Let $(r,\theta)$ be the polar coordinates of $\mathbb{C}^{\ast}$ and consider 
a (smooth) local trivialization 
$\pi^{-1}(\bar{U}) \cong \bar{U} \times \mathbb{C}^*$ of the 
$\mathbb{C}^*$-principal bundle $\pi : M\rightarrow \bar M$, such that $\xi=r \partial_{r}$ and 
$J\xi = \partial_\theta$. 
For each local trivialization 
$\pi^{-1}(\bar{U}) \cong \bar{U} \times \mathbb{C}^*$,  
we set 
\[ B:=e^{2 \theta J} A \,\,\,\,\, (e^{\theta J}:=(\cos \theta)\mathrm{Id}+(\sin \theta)J). \]

\begin{lemma}\label{localB}
The tensor $B$ is projectable, that is, $B_{\xi}=B_{J\xi}=0$ and 
$L_{\xi}B=L_{J\xi}B=0$.
\end{lemma}
\begin{proof}
It follows from Lemma \ref{Lie_A}. 
\end{proof}

We note that the condition $B_{\xi}=B_{J\xi}=0$ is called horizontality of $B$ while 
the condition $L_{\xi}B=L_{J\xi}B=0$ is equivalent to the $\mathbb{C}^*$-invariance of $B$. 
Since the symmetric tensor field $B$ is defined {\it locally} on the $\mathbb{C}^*$-invariant 
open subset $\pi^{-1}(\bar U)\subset M$ and is projectable,  
it projects to a {\it locally} defined symmetric tensor field $\bar{B}$ on $\bar{U}$.

\begin{lemma}
The tensor $B^{2}:(X,Y) \mapsto B_{X} \circ B_{Y}$ is a {\it globally} defined tensor field on $M$, in particular it holds $[B,B]=[A,A]$. 
As a consequence, we have the {\it globally} defined tensor field $[\bar{B},\bar{B}]$ on $\bar{M}$. 
\end{lemma}

By Lemma \ref{localB}, we can decompose 
\begin{equation} \label{BBbar:eq}B_{\tilde{X}}\tilde{Y}=(\bar{B}_{X}Y)^{\widetilde{}}+(\pi^{\ast}\bar{c})(X,Y)\xi 
+(\pi^{\ast}\bar{c}_{\bar{J}})(X,Y)J\xi \end{equation}
for $X$, $Y \in T\bar{U}$, where $\bar{c}$ is a symmetric tensor on $\bar{U}$. 
Since $B$ anti-commutes with ${\bar{J}}$,   
we get $\bar{c}(\bar{J}X,\bar{J}Y)=-\bar{c}(X,Y)$ 
for $X$, $Y \in T\bar{M}$.

Using (\ref{RD:eq}), we see that $R^{D}_{E,F}G=0$ if one of the vectors $E$, $F$, $G$ is vertical.  
By a fundamental equation of an affine submersion (\cite{AH}), 
the curvature $R^{D^{(\omega)}}$ is given in \cite{CH}:
\begin{align}\label{curv} 
R^{D^{(\omega)}} 
=& 
-\frac{1}{4}[\bar{B},\bar{B}]
 +2 \bar{a}^{a} \otimes \mathrm{Id} - 2 (\bar{a}_{\bar{J}})^{a} \otimes \bar{J} 
 + \bar{a} \wedge \mathrm{Id} - \bar{a}_{\bar{J}} \wedge \bar{J}.  
\end{align}
Note that the sign convention of \cite{AH} for the curvature tensor is opposite to ours.  
The equation \eqref{curv} corresponds to the horizontal component of $R^{D}_{\tilde{X},\tilde{Y}} \tilde{Z}$. 
By computing the vertical component of 
$R^{D}_{\tilde{X},\tilde{Y}} \tilde{Z}$, we have the following. 

\begin{lemma}\label{bar_c}
$(d^{D^{(\omega)}} \bar{a})(X,Y,Z)=-\frac{1}{4}(\bar{c}(X,\bar{B}_{Y}Z) -\bar{c}(Y,\bar{B}_{X}Z))$ 
for all $X$, $Y$, $Z \in T \bar{M}$. 
\end{lemma}

\begin{proof}
The vertical component of $R^{D}_{\tilde{X},\tilde{Y}} \tilde{Z}$ is 
\[ v(R^{D}_{\tilde{X},\tilde{Y}} \tilde{Z})
=(d^{D^{(\omega)}} \bar{a})(X,Y,Z) \xi - (d^{D^{(\omega)}} \bar{a}_{\bar{J}})(X,Y,Z)J\xi \]
by \cite[Theorem 3.5]{AH} and Lemma \ref{A_eq}. The equation (\ref{RD:eq}) leads to the conclusion. 
\end{proof}

Since $d^{D^{(\omega)}} \bar{a}_{\bar J}=(d^{D^{(\omega)}} \bar{a})
(\, \cdot \,  ,  \, \cdot \, ,  \bar{J} \, \cdot \, )$, Lemma \ref{bar_c} 
is equivalent to 
\[ 
(d^{D^{(\omega)}} \bar{a}_{\bar{J}})(X,Y,Z)
=\frac{1}{4}(\bar{c}_{\bar{J}}(X,\bar{B}_{Y}Z) -\bar{c}_{\bar{J}}(Y,\bar{B}_{X}Z)),
\]
which corresponds to the $J \xi$-component.

Defining $\bar{\cal B}$ by $\bar{\cal B}(Y,Z)=\mathrm{Tr} (\bar{B}_{Y}\bar{B}_{Z})$ for $Y$, $Z \in T\bar{M}$, 
which is a symmetric, $\bar{J}$-Hermitian globally defined 
$(0,2)$-tensor, 
we have shown in \cite{CH} that 
\begin{align}\label{bar_a}
\bar{a}=\frac{1}{8(n+1)} \bar{\cal B} 
                 -\frac{1}{2} P^{D^{(\omega)}}. 
\end{align}
Note that
\begin{align}
\bar{a}^{a} &=-\frac{1}{2} (P^{D^{(\omega)}})^{a} \left(=-\frac{1}{2(n+1)} (Ric^{D^{(\omega)}})^{a} \right), \label{a_ricci} \\ 
(\bar{a}_{\bar{J}})^{a} & = \frac{1}{8(n+1)} \bar{\cal B}_{\bar{J}}
-\frac{1}{2} (P^{D^{(\omega)}}_{\bar{J}})^{a}. \label{eq_11}
\end{align}

\begin{remark}\label{a_bar_skew}
{\rm 
Let $V:=(\Lambda^{2n}(T^{\ast}\bar{M}) \backslash \{ 0 \}) 
/\{ \pm 1\}$ be the (trivial) principal $\mathbb{R}^{>0}$-bundle over 
$\bar{M}$. 
Let $\Omega$ be a global frame of $\Lambda^{2n}(T^{\ast}\bar{M})$. We can define 
a connection form on $V$ as follows. The equation $D^{(\omega)} \Omega=l \otimes \Omega$ defines a 
one-form $l$ on $\bar M$. 
Since $-2(Ric^{D^{(\omega)}})^{a}=d l $ and (\ref{a_ricci}), 
the two-form $p^{\ast} \bar{a}^{a}$ is the curvature form of the connection 
\[ p^{\ast} \left( \frac{1}{4(n+1)} l \right) + \frac{dr}{r}\]
on $V$, where $r$ is the standard coordinate on $\mathbb{R}^{>0}$ and 
$p:V(\cong \bar{M} \times \mathbb{R}^{>0}) \to \bar{M}$ 
is the bundle projection. 
}
\end{remark}

The c-projective Weyl curvature of the canonical c-projective structure 
$\bar{\mathcal{P}}(\pi)$ on $\bar{M}$ is given as follows \cite[Theorem 7.10]{CH}: 
\begin{align}\label{Weyl_bar}
W^{\bar{\mathcal{P}}(\pi)}
=& -\frac{1}{4} [\bar{B},\bar{B}] 
-\frac{1}{4(n+1)} \bar{\cal B}_{\bar{J}} \otimes \bar{J} 
+\frac{1}{8(n+1)} \bar{\cal B} \wedge \mathrm{Id}
    -\frac{1}{8(n+1)} \bar{\cal B}_{\bar{J}} \wedge \bar{J}. 
\end{align}
In particular, $W^{\bar{\mathcal{P}}(\pi) }_{\bar{J}(\, \cdot \, ),\bar{J}(\, \cdot \, )}=W^{\bar{\mathcal{P}}(\pi) }$, that is, 
$W^{\bar{\mathcal{P}}(\pi)}$ is of type $(1,1)$ 
as an $\mathrm{End}(T\bar{M})$-valued two-form. 
Therefore we observe that 
\begin{itemize}
\item any complex manifold $(\bar{M},\bar{J},[\bar{D}])$ 
with a c-projective structure $[\bar{D}]$ whose c\--pro\-jec\-tive Weyl curvature 
is {\it not} of type $(1,1)$ {\it cannot} be realized as a projective special complex manifold 
such that $[\bar{D}]$ coincides with the canonical c-projective structure 
$\bar{\mathcal{P}}(\pi)$. 
\end{itemize}
The realization problem for $(\bar{M},\bar{J},[\bar{D}])$ as a projective special complex manifold, 
which is the purpose of this paper, will be addressed in Section \ref{sec3}.

Since the two-form $2(\bar{a}_{\bar{J}})^{a}$ is the curvature form of an $S^{1}$-bundle $S \to \bar{M}$ 
(\cite[Proposition 7.9 and Remark 8.12]{CH}), we see that  
\begin{align}\label{int_1}
\left[ \frac{1}{\pi } (\bar{a}_{\bar{J}})^{a} \right] \in H^{2}(\bar{M},\mathbb{Z}). 
\end{align}
The cohomology class $\left[ \frac{1}{\pi } (\bar{a}_{\bar{J}})^{a} \right]$ 
is called the {\cmssl characteristic class} of $S \to \bar{M}$.  
By \eqref{curv}, we have
\begin{align}\label{tr_R}
\mathrm{Tr}_{\mathbb{R}} R^{D^{(\omega)}}_{X,Y} \circ \bar{J}
=-\frac{1}{2} \bar{\cal B}_{\bar{J}} (X,Y)
+4(n+1)(\bar{a}_{\bar{J}})^{a}(X,Y)
\end{align}
for $X$, $Y \in T\bar{M}$. 
Since we know that the first Chern class $c_{1}(\bar{M})$ of $(\bar{M},\bar{J})$ is given by  
$c_{1}(\bar{M})= [ \frac{1}{4 \pi} (\mathrm{Tr}_{\mathbb{R}} R^{D^{(\omega)}} \circ \bar{J}) ]$, 
\eqref{int_1} and \eqref{tr_R} imply  
\begin{align}\label{trB_coh}
\left[ \frac{1}{8 \pi} \, \bar{\cal B}_{\bar{J}} \right]=-c_{1}(\bar{M})
+(n+1)\left[ \frac{1}{\pi} (\bar{a}_{\bar{J}})^{a} \right] \in H^{2}(\bar{M},\mathbb{Z}). 
\end{align} 
In particular, if $c_{1}(\bar{M}) \neq (n+1) \left[ \frac{1}{\pi} (\bar{a}_{\bar{J}})^{a} \right]$, then the conical special complex manifold $M$
is not 
trivial, that is, $A \neq 0$. 

\begin{proposition}\label{rep_fist}
The first Chern class $c_{1}(\bar{M})$
of a projective special complex manifold $(\bar{M},\bar{J})$ with $\dim \bar{M} \geq 4$ is given by 
\begin{align*}
c_{1}(\bar{M})=\left[ -\frac{n+1}{2 \pi} (P^{D^{(\omega)}}_{\bar{J}})^{a} \right] 
= \left[ -\frac{1}{2\pi} (Ric^{D^{(\omega)}}_{\bar{J}})^{a} \right]. 
\end{align*}
\end{proposition}
\begin{proof}
This follows from equations \eqref{eq_11}, \eqref{trB_coh} and Lemma \ref{pho_ricci}. 
\end{proof}

This also indicates that a c-projective geometric invariant 
$[ -\frac{n+1}{2 \pi} (P^{D^{(\omega)}}_{\bar{J}})^{a}]$ (see Lemma \ref{rho_inv}) on a projective special complex manifold 
$(\bar{M},\bar{J})$ 
is determined solely by the complex manifold structure of  
$(\bar{M},\bar{J})$.

\subsection{Local expressions}

In this subsection, we give local expressions for the fundamental objects of a 
special complex manifold and a projective special complex manifold.
Let $(r,\theta)$ be the polar coordinates of $\mathbb{C}^{\ast}$ and consider 
a (smooth) local trivialization 
$\pi^{-1}(\bar{U}) \cong \bar{U} \times \mathbb{C}^*$ of the 
$\mathbb{C}^*$-principal bundle $\pi : M\rightarrow \bar M$, such that $\xi=r \partial_{r}$ and 
$J\xi = \partial_\theta$. 
A principal connection $\omega$ is locally given by 
\[ \omega := \omega_{1} \otimes 1 +\omega_{2} \otimes \sqrt{-1} 
= \pi^{\ast}(\gamma_{1} \otimes 1+ \gamma_{2} \otimes \sqrt{-1}) 
+ \left(\frac{dr}{r} \otimes 1 +d \theta \otimes \sqrt{-1} \right) \]
for a $\mathbb{C}$-valued one-form $\gamma_{1} \otimes 1+ \gamma_{2} \otimes \sqrt{-1}$
on $\bar{U} \subset \bar{M}$. We recall  that $B=e^{2 \theta J} A.$
By Lemma \ref{closed_A} and using that 
$d^{D}e^{-2 \theta J}=-2d \theta \wedge J e^{-2 \theta J}$, 
we have
\begin{align*}
0 &=d^{D}A=d^{D}(e^{-2 \theta J} B)
=-2d \theta \wedge J e^{-2 \theta J} B + e^{-2 \theta J} d^{D}B.
\end{align*}
Therefore it holds
\begin{align}
  d^{D}B = 2d \theta \wedge JB=2(\omega_{2}-\pi^{\ast}\gamma_{2}) \wedge JB.
\end{align}

\begin{lemma}\label{Bbar:lem}
The equations
\begin{align*}
& d ^{D^{(\omega)}}\bar{B}  =\bar c \wedge \mathrm{Id} + \bar c_{\bar J}\wedge \bar J - 2\gamma_{2} \wedge \bar{J} \bar{B}, \\
& \bar{a}(X,\bar{B}_{Y}Z)-\bar{a}(Y,\bar{B}_{X}Z)+(d^{D^{(\omega)}} \bar{c})(X,Y,Z)
=2(\gamma_{2} \wedge \bar{c}_{\bar{J}}) (X,Y,Z)
\end{align*}
hold for all $X$, $Y$, $Z \in T\bar{M}$.  
\end{lemma}

\begin{proof}
It suffices to calculate 
$(d^{D}B)(\tilde{X}, \tilde{Y}, \tilde{Z})$ and 
$(2(\omega_{2}-\pi^{\ast}\gamma_{2}) \wedge JB)
(\tilde{X}, \tilde{Y},\tilde{Z})$, 
which lead to the conclusion. We find 
\begin{align*}
(d^{D} B)(\tilde{X}, \tilde{Y}, \tilde{Z})
=&(d^{D^{(\omega)}} \bar{B})(X,Y, Z){\,}^{\widetilde{}}+ \bar c (Y,Z)\tilde{X} + \bar c_{\bar J} (Y,Z)J\tilde{X} \\
&-\bar c (X,Z)\tilde{Y} - \bar c_{\bar J} (X,Z)J\tilde{Y} \\
&+\left( \bar{a}(X,\bar{B}_{Y}Z)-\bar{a}(Y,\bar{B}_{X}Z)+
(d^{D^{(\omega)}} \bar{c})(X,Y,Z) \right)\xi \\
&+\left( -\bar{a}_{\bar{J}}(X,\bar{B}_{Y}Z)+\bar{a}_{\bar{J}}(Y,\bar{B}_{X}Z)+
(d^{D^{(\omega)}} \bar{c}_{\bar{J}})(X,Y,Z) \right) J \xi
\end{align*}
and 
\begin{align*}
(2(\omega_{2}-\pi^{\ast}\gamma_{2}) \wedge JB)
(\tilde{X}, \tilde{Y},\tilde{Z})
=& -2 ((\gamma_{2} \wedge \bar{J} \bar{B})(X,Y,Z)){\,}^{\widetilde{}}
+2 (\gamma_{2} \wedge \bar{c}_{\bar{J}})(X,Y,Z) \xi \\
&-2 (\gamma_{2} \wedge \bar{c})(X,Y,Z) J\xi. \qedhere
\end{align*}
\end{proof}

The second equation in the previous lemma is equivalent to 
\[ 
 -\bar{a}_{\bar{J}}(X,\bar{B}_{Y}Z)+\bar{a}_{\bar{J}}(Y,\bar{B}_{X}Z)+
(d^{D^{(\omega )}} \bar{c}_{\bar{J}})(X,Y,Z) 
=-2 (\gamma_{2} \wedge \bar{c})(X,Y,Z)
\]
for all $X$, $Y$ and $Z \in T\bar{M}$. This can be seen by replacing $Z$ by $\bar J Z$ in the second 
equation of Lemma \ref{Bbar:lem} and doing a short calculation.

Let $\{ \bar{U}_{\alpha} \}_{\alpha \in \Lambda}$ be an open covering of 
$\bar{M}$ 
with local trivializations  
$U_{\alpha}:=\pi^{-1}(\bar{U}_{\alpha}) \cong \bar{U}_{\alpha} \times \mathbb{C}^{\ast}$ 
and $g_{\alpha \beta}:\bar{U}_{\alpha} \cap \bar{U}_{\beta} \to \mathbb{C}^{\ast}$ 
be the corresponding transition functions. If we write 
$g_{\alpha \beta}=e^{f^{1}_{\alpha \beta} + f^{2}_{\alpha \beta} \sqrt{-1}}$, then 
\begin{align}
&f^{1}_{\alpha \beta} + f^{1}_{\beta \gamma} - f^{1}_{\alpha \gamma} =0, \label{cocycle_1} \\ 
&f^{2}_{\alpha \beta} + f^{2}_{\beta \gamma} - f^{2}_{\alpha \gamma} \in 2 \pi \mathbb{Z}, \label{cocycle_2}\\
&\gamma_{1}^{\beta}-\gamma_{1}^{\alpha} =d f^{1}_{\alpha \beta}, \label{cocycle_3}\\
&\gamma_{2}^{\beta}-\gamma_{2}^{\alpha} =d f^{2}_{\alpha \beta} \label{cocycle_4},
\end{align}
where $\gamma_1^\alpha + \sqrt{-1}\gamma_2^\alpha$ denotes the connection form 
(previously denoted by $\gamma_{1}+\sqrt{-1}\gamma_{2}$)
associated with the trivialization labeled by $\alpha$. 
We write $\bar{B}^{\alpha}$ and $\bar{c}^{\alpha}$ for $\bar{B}$ and $\bar{c}$  
on each $\bar{U}_{\alpha}$ ($\alpha \in \Lambda$). 
Then, by the cocycle conditions, we see that   
\begin{align}
\bar{B}^{\beta} &=e^{-2 f^{2}_{\alpha \beta} \bar{J}} \bar{B}^{\alpha}, \label{cocy_B} \\
\bar{c}^{\beta}+\sqrt{-1}\, \bar{c}^{\beta}_{\bar{J}} 
&=e^{-2 f^{2}_{\alpha \beta} \sqrt{-1}}(\bar{c}^{\alpha}+\sqrt{-1} \, \bar{c}^{\alpha}_{\bar{J}} ). \label{cocy_c}
\end{align}
In fact, since $B^{\beta}=e^{2 \theta_{\beta} J}A
=e^{2 (\theta_{\alpha}-f^{2}_{\alpha \beta}) J}A
=e^{-2 f^{2}_{\alpha \beta} J}B^{\alpha}$, 
we have 
\begin{align*}
B^{\beta}_{\tilde{X}} \tilde{Y}
&=(\bar{B}^{\beta}_{X}Y){}^{\widetilde{}}+\bar{c}^{\beta}(X,Y)\xi
+\bar{c}^{\beta}_{\bar{J}}(X,Y)J\xi \\
&=(\cos(2 f^{2}_{\alpha \beta}) \mathrm{Id} - \sin(2 f^{2}_{\alpha \beta}) J)
((\bar{B}^{\alpha}_{X}Y){}^{\widetilde{}}+\bar{c}^{\alpha}(X,Y)\xi
+\bar{c}^{\alpha}_{\bar{J}}(X,Y)J\xi ).
\end{align*}
Therefore it holds that
\begin{align*}
\bar{B}^{\beta}
&=e^{-2 f^{2}_{\alpha \beta} \bar{J}}\bar{B}^{\alpha}, \\
\bar{c}^{\beta}
&=\cos(2 f^{2}_{\alpha \beta}) \bar{c}^{\alpha} 
+ \sin (2 f^{2}_{\alpha \beta}) \bar{c}^{\alpha}_{\bar{J}}, \,\, 
\bar{c}^{\beta}_{\bar{J}}
=\cos(2 f^{2}_{\alpha \beta}) \bar{c}^{\alpha}_{\bar{J}} 
- \sin (2 f^{2}_{\alpha \beta}) \bar{c}^{\alpha}. 
\end{align*}
The last two equations are equivalent to (\ref{cocy_c}).

\begin{definition} \label{hortype:def} A conical special complex manifold $M$ 
is said to be of {\cmssl horizontal type} with respect to a principal connection of 
$\pi:M \to \bar{M}$ 
if $v \circ A=0$, that is 
if the (1,2)-tensor  $\nabla J$ takes only horizontal values. 
\end{definition}
Note that the special complex  manifold is of horizontal type if and only if $\bar{c}^{\alpha}=0$ for all $\alpha \in \Lambda$.

\section{Characterization of projective special complex manifolds in terms of c-projective structures}
\label{sec3}
\setcounter{equation}{0}

The calculations in the previous section, \eqref{Weyl_bar} and 
Remark \ref{a_bar_skew} motivate the following definition. 
In the case of $\dim \bar{M}=2$, any two complex connections are 
c-projectively related. So we divide the definition into $\dim \bar{M} \geq 4$ and $\dim \bar{M}=2$. 

\subsection{$\dim \bar{M} \geq 4$ case}
\label{dim_ge4:sec}

\begin{definition}\label{intrinsic_3}
Let $(\bar M , \bar J, \mathcal{P})$ be a complex manifold of real dimension $\dim \bar{M}=2n \geq 4$ and  
a c-projective structure $\mathcal{P}$, and 
$\pi_{S}:S \to \bar{M}$ be a principal $S^{1}$-bundle with a connection 
$\eta$. 
Let $\{ \bar{U}_{\alpha} \}_{\alpha \in \Lambda}$ be an open covering of 
$\bar{M}$ 
with local trivializations  
$\pi_{S}^{-1}(\bar{U}_{\alpha}) \cong \bar{U}_{\alpha} \times S^{1}$ for 
each $\alpha \in \Lambda$ and $e^{f^{2}_{\alpha \beta} \sqrt{-1}}$ be the transition functions associated 
with $\{ \bar{U}_{\alpha} \}_{\alpha \in \Lambda}$. 
Let $\{ \gamma_{2}^{\alpha} \}_{\alpha \in \Lambda}$ be the collection of a local $1$-forms on $\bar{U}_{\alpha}$ 
associated to the connection $1$-form $\eta$.   
We say that $\pi_{S}:S \to \bar{M}$ (endowed with the connection $\eta$) is of {\cmssl projective special complex base (PSCB)  type} with
collections 
$\{\bar{B}^{\alpha} \}_{\alpha \in \Lambda}$,  
$\{\bar{c}^{\alpha} \}_{\alpha \in \Lambda}$ 
if the following conditions hold: \\
(1) \, $\bar{B}^{\alpha}$ is symmetric $(1,2)$-tensor on $\bar{U}_{\alpha}$ 
and anti-commutes with $\bar{J}$ for all $\alpha \in \Lambda$, \\
(2) \, 
$\bar{B}^{\beta} =e^{-2 f^{2}_{\alpha \beta} \bar{J}} \bar{B}^{\alpha}$
on $\bar{U}_{\alpha} \cap \bar{U}_{\beta}$, \\
(3) \, $\bar{c}^{\alpha}$ is a symmetric $(0,2)$-tensor on $\bar{U}_{\alpha}$ with 
$\bar{J}^{\ast} \bar{c}^{\alpha}=-\bar{c}^{\alpha}$
for all $\alpha \in \Lambda$, \\
(4) \, $\bar{c}^{\beta}+\sqrt{-1} \, \bar{c}^{\beta}_{\bar{J}} 
=e^{-2 f^{2}_{\alpha \beta} \sqrt{-1}} (\bar{c}^{\alpha}+\sqrt{-1} \, \bar{c}^{\alpha}_{\bar{J}} )$ on $\bar{U}_{\alpha} \cap \bar{U}_{\beta}$, \\ 
(5) \, the c-projective Weyl curvature $W^{\mathcal{P} }$ of $\mathcal{P}$ satisfies 
\begin{align*}
W^{\mathcal{P} }=& -\frac{1}{4} [\bar{B}^{\alpha},\bar{B}^{\alpha}] 
-\frac{1}{4(n+1)} \bar{\cal B}_{\bar{J}} \otimes \bar{J} 
+\frac{1}{8(n+1)} \bar{\cal B} \wedge \mathrm{Id}
-\frac{1}{8(n+1)} \bar{\cal B}_{\bar{J}} \wedge \bar{J},  
\end{align*}
where $\bar{\cal B}(Y,Z) :=\mathrm{Tr} (\bar{B}^{\alpha}_{Y}\bar{B}^{\alpha}_{Z})$ for 
$Y$, $Z \in T \bar{M}$, \\
(6) \, there exists a torsion-free complex connection $\bar{D} \in \mathcal{P}$ such that \\
(6-1) \, the curvature of $\eta$ satisfies $d \eta=2 \pi_{S}^{\ast}((\bar{a}_{\bar{J}})^{a})$, where 
\[ 
\bar{a}=\frac{1}{8(n+1)} \bar{\cal B}
                 -\frac{1}{2} P^{\bar{D}}, \]
(6-2) \, $(d^{\bar{D}} \bar{a})(X,Y,Z)=-(1/4) (\bar{c}^{\alpha}(X,\bar{B}^{\alpha}_{Y}Z) 
-\bar{c}^{\alpha}(Y,\bar{B}^{\alpha}_{X}Z))$ 
for all $\alpha \in \Lambda$ and $X$, $Y$, $Z \in T\bar{U}_{\alpha}$, \\
(6-3) \, $d ^{\bar{D}} \bar{B}^{\alpha}  =- 2\gamma_{2}^{\alpha} 
\wedge \bar{J} \bar{B}^{\alpha} +\bar c^\alpha \wedge \mathrm{Id} + \bar c^\alpha_{\bar J}\wedge \bar J$ for all $\alpha \in \Lambda$, \\
(6-4) \, $\bar{a}(X,\bar{B}^{\alpha}_{Y}Z)-\bar{a}(Y,\bar{B}^{\alpha}_{X}Z)
+(d^{\bar{D}} \bar{c}^{\alpha})(X,Y,Z)
=2(\gamma_{2}^{\alpha} \wedge \bar{c}^{\alpha}_{\bar{J}}) (X,Y,Z)$ for all $\alpha \in \Lambda$ and $X$, $Y$, $Z \in T\bar{U}_{\alpha}$. 
\end{definition}

\begin{remark}
{\rm 
The locally defined $\bar{B}^{\alpha}$ give a global $(1,3)$-tensor $(\bar{B}^{\alpha})^{2}$. 
Hence $\bar{\cal B}$ and $[\bar{B}^{\alpha},\bar{B}^{\alpha}]$ are globally defined on $\bar{M}$. 
}
\end{remark}

\begin{remark}
{\rm 
The local conditions (6-2), (6-3) and (6-4) are in fact global conditions on $\bar{M}$. 
For example, we have 
\begin{align*}
d^{\bar{D}} \bar{B}^{\beta} 
+ 2 \gamma^{\beta}_{2} \wedge \bar{J} \bar{B}^{\beta} - \bar c^\beta \wedge \mathrm{Id} 
& - \bar c^\beta_{\bar J}\wedge \bar J \\
&=e^{-2 f^{2}_{\alpha \beta} \bar{J}}
(d^{\bar{D}} \bar{B}^{\alpha} 
+ 2 \gamma^{\alpha}_{2} \wedge \bar{J} \bar{B}^{\alpha} -\bar c^\alpha \wedge \mathrm{Id} - \bar c^\alpha_{\bar J}\wedge \bar J). 
\end{align*}
The others are checked similarly using part (2) and (4) in Definition 
\ref{intrinsic_3} and (\ref{cocycle_4}). 
}
\end{remark}

\begin{example}\label{hopf_bdl}
Let $(\mathbb{C}P^{n}, \bar{J}, [ \bar{D}^{\bar g} ])$ 
be the complex projective space with the standard complex structure $\bar{J}$ and 
the c-projective structure of the Levi-Civita connection $\bar{D}^{\bar g}$ 
of the Fubini-Study metric $\bar{g}$. 
One of the simplest examples of a principal $S^1$-bundle of PSCB type 
is the Hopf fibration $S^{2n+1} \to \mathbb{C}P^{n}$. 
In fact, in this case $\eta$ is the canonical contact form of $S^{2n+1}$, 
 $\bar{a}=-\bar{g}$, 
$\bar B^{\alpha}=0$, $\bar c^{\alpha}=0$ for all $\alpha$. 
\end{example}

Next we will show that given an $S^1$-bundle $S \to (\bar M,\bar J, \mathcal{P})$ of PSCB type, 
the base manifold is really a projective special complex manifold such that 
its canonical c-projective structure coincides with $\mathcal{P}$. 
To prove this we construct a conical special complex manifold $M$ with principal $\mathbb{C}^*$-action 
and check that  $M/\mathbb{C}^* =\bar M$ with the canonical c-projective structure on the quotient 
coinciding with the initial c-projective structure $\mathcal{P}$. 

Let $(\bar M , \bar J, \mathcal{P})$ be a complex manifold with 
a c-projective structure $\mathcal{P}$ and 
let $\pi_{S}:S \to \bar M$ be a principal $S^1$-bundle of PSCB type with  connection 
$\eta$ as in Definition \ref{intrinsic_3}. 
Let $\bar{D} \in \cal{P}$ be a torsion-free complex connection 
satisfying (6) of Definition \ref{intrinsic_3}. 
Consider $p: V = (\Lambda^{2n}T^*\bar M\setminus \{ 0\})/\{ \pm 1\} \cong\bar{M} \times \mathbb{R}^{>0}\to \bar M$ the trivial $\mathbb{R}^{>0}$-bundle 
with the connection 
\begin{align}\label{loc_con}
\omega_1 := -p^{\ast} \left( \frac{1}{2(n+1)} l \right) + \frac{dr}{r}, 
\end{align}
where $l$ is defined by the covariant derivative $\bar D\Omega = l\otimes \Omega$ of the 
trivializing section $\Omega$. 
We set the $\mathbb{C}^{\ast}$-bundle 
$M:=S \times \mathbb{R}^{>0}$ and consider the connection $\omega = \omega_1 + \sqrt{-1}\omega_2$ 
on $M$ induced from (\ref{loc_con}) and $\omega_2 := \eta$ (considered as one-form on $M$). 
We denote by $\pi$ the projection from $M$ onto $\bar{M}$, which is defined 
as the composition of the canonical projection $\pi_{1}:M =S \times \mathbb{R}^{>0}\to S$ 
with $\pi_{S} : S \to \bar M$. 
The fundamental vector fields generated by 
$1 \in T_{1} \mathbb{R}^{>0} \cong \mathbb{R}$ and 
$\sqrt{-1} \in T_{1}S^{1} \cong \sqrt{-1} \mathbb{R}$ are denoted by 
$\xi$ and $Z_M$, respectively. 
Let $(r_{\alpha}=r,\theta_{\alpha})$ be the polar coordinates with respect to 
a local trivialization $\pi^{-1}(\bar{U}_{\alpha}) \cong \bar{U}_{\alpha} \times \mathbb{C}^{\ast}$. 

The connection form $\omega$ is locally expressed on 
$\pi^{-1}(\bar{U}_{\alpha}) =\pi_{S}^{-1}(\bar{U}_{\alpha}) \times \mathbb{R}^{>0} \cong 
\bar{U}_{\alpha} \times S^{1} \times \mathbb{R}^{>0}$ 
as 
\[ \omega = \omega_{1} \otimes 1 +\omega_{2} \otimes \sqrt{-1} = \pi^{\ast}(\gamma^{\alpha}_{1} \otimes 1
+ \gamma^{\alpha}_{2} \otimes \sqrt{-1}) 
+ \left(\frac{dr_{\alpha}}{r_{\alpha}} \otimes 1 
+d \theta_{\alpha} \otimes \sqrt{-1} \right). \]
Note that $\{ \gamma^{\alpha}_{1}=-\frac{1}{2(n+1)} l  \}_{\alpha \in \Lambda}$ 
and 
$\{ \gamma^{\alpha}_{2}=\sigma_{\alpha}^{\ast} \eta \}_{\alpha \in \Lambda}$ satisfy 
(\ref{cocycle_3}) and (\ref{cocycle_4}), respectively, 
where $\sigma_{\alpha}:\bar{U}_{\alpha} \to S$ is a local section of 
$\pi_{S}$ given by the local trivialization of 
$\pi^{-1}(\bar{U}_{\alpha}) \cong \bar{U}_{\alpha} \times \mathbb{C}^{\ast}$. 
Its curvature is expressed by 
\begin{align}\label{d_eta_1}
 d \omega =-2 (\pi^{\ast} \bar{a}^{a}) \otimes 1
+ 2 (\pi^{\ast} (\bar{a}_{\bar{J}})^{a}) \otimes \sqrt{-1}, 
\end{align}
where $\bar{a}$ is the $(0,2)$-tensor on $\bar{M}$ defined as 
\[ \bar{a}=\frac{1}{8(n+1)} \bar{\cal B}
                 -\frac{1}{2} P^{\bar{D}} \]
by (\ref{loc_con}) and (6-1) of Definition \ref{intrinsic_3}. 
We define an almost complex structure $J$ on $M$ by 
\[ J \xi=Z_M,\,\, JZ_M=-\xi, \,\, J(\tilde{X})=(\bar{J}X){}^{\widetilde{}} \]
for $X \in T\bar{M}$. 
Note that $\mathrm{Ker}\,  \omega$ is $J$-invariant.  
Next consider the connection $D$ on $M$ defined by 
\begin{align}\label{def_c_con}
D_{\tilde{X}} \tilde{Y}
&:=(\bar{D}_{X} Y) {}^{\widetilde{}} + (\pi^{\ast} \bar{a})(\tilde{X},\tilde{Y})\xi
- (\pi^{\ast} \bar{a}_{\bar J})(\tilde{X},\tilde{Y})Z_M
\\
D_{\tilde{X}} \xi &=D_{\xi} \tilde{X}:=\tilde{X} \nonumber\\
D_{\tilde{X}} Z_M &=D_{Z_M} \tilde{X}:=J \tilde{X} \nonumber\\
D_{\xi} \xi &:=\xi, \,\, D_{Z_M}Z_M:=-\xi, \,\, D_{Z_M} \xi=D_{\xi}Z_M:=Z_M \nonumber
\end{align}
for $X$ and $Y \in \Gamma(T\bar{M})$.

\begin{lemma}\label{cpx}
We have \\
(1) $D$ is torsion-free, \\
(2) $D J=0$, which implies the integrability of $J$, \\
(3) $D \xi = \mathrm{Id}$, $DZ_M=J$,\\
(4) $L_{\xi} J=0$ and $L_{Z_M} J=0$.
\end{lemma}

\begin{proof}
Let $T^{D}$ be the torsion tensor of $D$. It is easy to see that 
$T^{D}(\xi, \, \cdot \,)=0$ and $T^{D}(Z_M, \, \cdot \,)=0$. 
For the horizontal lifts $\tilde{X}$ and $\tilde{Y}$ of $X$ and $Y$, we have  
\[ hT^D(\tilde{X},\tilde{Y})=\widetilde{[X,Y]}- h[\tilde{X},\tilde{Y}]=0\]  and 
\begin{align*}
v T^{D}(\tilde{X},\tilde{Y})
=2 \bar{a}^{a}(X,Y)\xi -2 \bar{a}_{\bar{J}}^{a}(X,Y)Z_M
+(d \omega_{1})(\tilde{X},\tilde{Y})\xi+(d \omega_{2})(\tilde{X},\tilde{Y})Z_M. 
\end{align*}
Therefore (\ref{d_eta_1}) implies that $T^{D}(\tilde{X},\tilde{Y})=0$. 
For (2), it is easy to see that $D_{\xi} J=0$, $D_{Z_M} J=0$, $(D_{\tilde{X}} J) \xi =0$ and 
$(D_{\tilde{X}} J) Z_M =0$. For the horizontal lifts $\tilde{X}$ and $\tilde{Y}$, we have 
\begin{align*}
(D_{\tilde{X}} J) \tilde{Y} 
=&((\bar{D}_{X} \bar{J}) Y{})^{\widetilde{}} + \bar{a}(X,\bar{J}Y)\xi - \bar{a}_{\bar{J}}(X,\bar{J}Y)Z_M
-\bar{a}(X,Y)Z_M - \bar{a}_{\bar{J}}(X,Y)\xi =0. 
\end{align*}
The definition of $D$ and $J$ immediately implies (3) and (4) follows from the 
previous formulas, since $L_\xi = D_\xi -D\xi = D_\xi -\mathrm{Id}$ and 
$L_{Z_M}=D_{Z_M}-DZ_M = D_{Z_M} -J$.   
\end{proof}

We define a symmetric tensor 
$B^{\alpha}$ on $\pi^{-1}(\bar{U}_{\alpha})$ for each 
$\alpha \in \Lambda$ by 
\begin{align}\label{test:def}
B^{\alpha}_{\tilde{X}} \tilde{Y}
&:=(\bar{B}^{\alpha}_{X}Y) {}^{\widetilde{}} +(\pi^{\ast}\bar{c}^{\alpha})(X,Y)\xi 
+(\pi^{\ast}\bar{c}_{\bar J}^{\alpha})(X,Y)Z_M, \quad 
B^{\alpha}_{\xi}  :=0, \,\,\, B^{\alpha}_{J \xi } :=0.   
\end{align}
Note that $B^{\alpha}$ and $J$  anti-commute.  
By (5) of Definition \ref{intrinsic_3}, 
the curvature $R^{\bar{D}}$ of $\bar{D}$ 
is of the form  
\begin{align}\label{RD_eq}
R^{\bar{D}} 
=& -\frac{1}{4}[\bar{B}^{\alpha},\bar{B}^{\alpha}]
 +2 \bar{a}^{a} \otimes Id - 2 (\bar{a}_{\bar{J}})^{a} \otimes \bar{J} 
+ \bar{a} \wedge Id - \bar{a}_{\bar{J}} \wedge \bar{J}. 
\end{align}

\begin{lemma}
The curvature $R^{D}$ of $D$ is given by 
\[ R^{D}=-\frac{1}{4} [B^{\alpha},B^{\alpha}]. \] 
\end{lemma}
\begin{proof}
It is easy to see that $R^{D}_{E,F}G=0$ if one of the vectors $E$, $F$, $G$ is vertical. 
Thus it is sufficient to calculate $R^{D}$ for horizontal vectors. 
Using (6-2) of Definition \ref{intrinsic_3} and (\ref{RD_eq}), we will do that. 
By the definition of $D$ in \eqref{def_c_con}, the projection $\pi:(M,D) \to (\bar{M},\bar{D})$ is 
an affine submersion and its fundamental tensors are given by  
$\mathcal{A}^{D}_{ \tilde{X}} \tilde{Y} =\bar{a}(X,Y)\xi - \bar{a}_{\bar{J}}(X,Y) Z_{M}$, 
$\mathcal{A}^{D}_{ \tilde{X}} \xi= \tilde{X}$, 
$\mathcal{A}^{D}_{ \tilde{X}} Z_{M}=({\bar{J} X})^{\widetilde{}} $, 
$\mathcal{T}^{D}=0$. By the fundamental equation of an affine submersion 
\cite[Theorem 3.5]{AH}, we have 
\begin{align*}
R^{D}_{\tilde{X},\tilde{Y}}\tilde{Z}
=&(R^{\bar{D}}_{X,Y}Z)^{\widetilde{}}-h (D_{v[\tilde{X},\tilde{Y}]} \tilde{Z})
-\mathcal{A}^{D}_{ \tilde{Y}} \mathcal{A}^{D}_{ \tilde{X}} \tilde{Z} 
+\mathcal{A}^{D}_{ \tilde{X}} \mathcal{A}^{D}_{ \tilde{Y}} \tilde{Z} \\
 &+v(D_{\tilde{X}} \mathcal{A}^{D} )_{\tilde{Y}} \tilde{Z}
   -v(D_{\tilde{Y}} \mathcal{A}^{D} )_{\tilde{X}} \tilde{Z} \\
=&(R^{\bar{D}}_{X,Y}Z)^{\widetilde{}}-h (D_{v[\tilde{X},\tilde{Y}]} \tilde{Z})\\
&-\bar{a}(X,Z)\tilde{Y}+ \bar{a}_{\bar{J}}(X,Z)(\bar{J}Y)^{\widetilde{}}
+\bar{a}(Y,Z)\tilde{X}- \bar{a}_{\bar{J}}(Y,Z)(\bar{J}X)^{\widetilde{}} \\
&+(d^{\bar{D}} \bar{a})(X,Y,Z) \xi - (d^{\bar{D}} \bar{a}_{\bar{J}})(X,Y,Z) Z_{M}. 
\end{align*}
With the help of \eqref{d_eta_1}, we calculate the term  
\begin{align*}
-h (D_{v[\tilde{X},\tilde{Y}]} \tilde{Z})
&=-h(D_{\tilde{Z}} v[\tilde{X},\tilde{Y}])
=(d \gamma^{\alpha}_{1})(X,Y)\tilde{Z}+(d \gamma^{\alpha}_{2})(X,Y)(\bar{J}Z)^{\widetilde{}}\\
&=-2\bar{a}^{a}(X,Y)\tilde{Z}+2\bar{a}_{\bar{J}}^{a}(X,Y) (\bar{J}Z)^{\widetilde{}}.
\end{align*}
Therefore, using (6-2) of Definition \ref{intrinsic_3} and (\ref{RD_eq}), we have 
\begin{align*}
R^{D}_{\tilde{X},\tilde{Y}}\tilde{Z}
=&\left( -\frac{1}{4} [\bar{B}^{\alpha}_{X}, \bar{B}^{\alpha}_{Y}]Z \right)^{\widetilde{}} \\
&-\frac{1}{4} \left( \bar{c}^{\alpha}(X,\bar{B}^{\alpha}_{Y}Z) 
-\bar{c}^{\alpha} (Y,\bar{B}^{\alpha}_{X}Z) \right) \xi
-\frac{1}{4} \left( \bar{c}^{\alpha}_{\bar{J}}(X,\bar{B}^{\alpha}_{Y}Z) 
-\bar{c}^{\alpha}_{\bar{J}} (Y,\bar{B}^{\alpha}_{X}Z) \right) Z_{M}.
\end{align*}
Moreover, using \eqref{test:def}, we see that 
\begin{align*}
[B^{\alpha}_{\tilde{X}}, B^{\alpha}_{\tilde{Y}} ] \tilde{Z}
=&B^{\alpha}_{\tilde{X}} (\bar{B}^{\alpha}_{Y}Z)^{\widetilde{}}
-B^{\alpha}_{\tilde{Y}} (\bar{B}^{\alpha}_{X}Z)^{\widetilde{}} \\
=&(\bar{B}^{\alpha}_{X} \bar{B}^{\alpha}_{Y}Z)^{\widetilde{}}
+\bar{c}^{\alpha}(X,\bar{B}^{\alpha}_{Y}Z)\xi
+\bar{c}^{\alpha}_{\bar{J}}(X,\bar{B}^{\alpha}_{Y}Z) Z_{M} \\
&-(\bar{B}^{\alpha}_{Y} \bar{B}^{\alpha}_{X}Z)^{\widetilde{}}
-\bar{c}^{\alpha}(Y,\bar{B}^{\alpha}_{X}Z)\xi
-\bar{c}^{\alpha}_{\bar{J}}(Y,\bar{B}^{\alpha}_{X}Z) Z_{M}.
\end{align*}
We conclude that 
\[ R^{D}_{\tilde{X},\tilde{Y}} \tilde{Z}
=-\frac{1}{4} [ B^{\alpha}_{\tilde{X}}, B^{\alpha}_{\tilde{Y}} ] \tilde{Z} 
\]
for all $X$, $Y$, $Z \in T \bar{M}$.
\end{proof}

We set 
\begin{equation}\label{AB:eq}
A^{\alpha}:=e^{-2 \theta_{\alpha} J} B^{\alpha} 
\end{equation}
on each $\pi^{-1}(\bar{U}_{\alpha})$

\begin{lemma}\label{def_A}
It holds that $A^{\alpha}=A^{\beta}$ on 
$\pi^{-1}(\bar{U}_{\alpha}) \cap \pi^{-1}(\bar{U}_{\beta})$.
\end{lemma}

\begin{proof}
This follows from (2) and (4) of Definition \ref{intrinsic_3}.
\end{proof}

Therefore, by Lemma \ref{def_A}, the collection $\{ A^{\alpha} \}_{\alpha \in \Lambda}$ defines a globally defined symmetric $(1,2)$-tensor $A$ on $M$ and  we see that 
\[ R^{D}=-\frac{1}{4} [B^{\alpha},B^{\alpha}]=-\frac{1}{4} [A^{\alpha},A^{\alpha}]=-\frac{1}{4} [A,A]. \]

\begin{lemma}\label{closedA}
$d^{D}A=0$. 
\end{lemma}
\begin{proof}
The conditions (6-3) and (6-4) of Definition \ref{intrinsic_3} imply that 
$d^{D} B^{\alpha}=2 d \theta_{\alpha} \wedge JB{\alpha}$ for each $\alpha \in \Lambda$. 
In fact, one can easily check that this equation holds when evaluated on three vectors 
such that at least one of them is vertical.  To check it on horizontal vectors we compute 
for all $X$, $Y$, $Z \in T\bar{M}$: 
\begin{align*}
(D_{\tilde{X}}B^{\alpha})_{\tilde{Y}} \tilde{Z}
=&((\bar{D}_{X} \bar{B}^{\alpha})_{Y}Z)^{\widetilde{}}
+((\bar{D}_{X} \bar{c}^{\alpha})(Y,Z))\xi
+((\bar{D}_{X} \bar{c}^{\alpha}_{\bar{J}})(Y,Z))Z_{M}\\
&+\bar{a}(X,\bar{B}^{\alpha}_{Y}Z)\xi-\bar{a}_{\bar{J}}(X,\bar{B}^{\alpha}_{Y}Z) Z_{M}
+\bar{c}^{\alpha}(Y,Z)\tilde{X}+\bar{c}^{\alpha}_{\bar{J}}(Y,Z)(\bar{J}X)^{\widetilde{}}.
\end{align*}
It follows that
\begin{align*}
 &(d^{D} B^{\alpha})(\tilde{X},\tilde{Y},\tilde{Z})\\
=&(d^{\bar{D}} \bar{B}^{\alpha})(X,Y,Z))^{\widetilde{}}
+(d^{\bar{D}} \bar{c}^{\alpha})(X,Y,Z))\xi
+(d^{\bar{D}} \bar{c}^{\alpha}_{\bar{J}})(X,Y,Z))Z_{M}\\
&+(\bar{a}(X,\bar{B}^{\alpha}_{Y}Z)-\bar{a}(Y,\bar{B}^{\alpha}_{X}Z))\xi
-(\bar{a}_{\bar{J}}(X,\bar{B}^{\alpha}_{Y}Z)-\bar{a}_{\bar{J}}(Y,\bar{B}^{\alpha}_{X}Z) )Z_{M}\\
&+\bar{c}^{\alpha}(Y,Z)\tilde{X}-\bar{c}^{\alpha}(X,Z)\tilde{Y}
+\bar{c}^{\alpha}_{\bar{J}}(Y,Z)(\bar{J}X)^{\widetilde{}}
-\bar{c}^{\alpha}_{\bar{J}}(X,Z)(\bar{J}Y)^{\widetilde{}}.
\end{align*}
The conditions (6-3) and (6-4) 
of Definition \ref{intrinsic_3} give that 
\begin{align*}
 &(d^{D} B^{\alpha})(\tilde{X},\tilde{Y},\tilde{Z})\\
=&(-2 \gamma^{\alpha}_{2} \wedge \bar{J} \bar{B}^{\alpha} 
+\bar{c}^{\alpha} \wedge \mathrm{Id} +\bar{c}^{\alpha}_{\bar{J}} \wedge \bar{J})(X,Y,Z))^{\widetilde{}} \\
&+2 (\gamma^{\alpha}_{2} \wedge \bar{c}^{\alpha}_{\bar{J}})(X,Y,Z)\xi
-2 (\gamma^{\alpha}_{2} \wedge \bar{c}^{\alpha})(X,Y,Z)Z_{M}\\
&-(\bar{c}^{\alpha} \wedge \mathrm{Id} +\bar{c}^{\alpha}_{\bar{J}} \wedge \bar{J})(X,Y,Z))^{\widetilde{}}\\
=&(-2 \gamma^{\alpha}_{2} \wedge \bar{J} \bar{B}^{\alpha}(X,Y,Z))^{\widetilde{}}
+2 (\gamma^{\alpha}_{2} \wedge \bar{c}^{\alpha}_{\bar{J}})(X,Y,Z)\xi
-2 (\gamma^{\alpha}_{2} \wedge \bar{c}^{\alpha})(X,Y,Z)Z_{M}. 
\end{align*}
On the other hand, we have
\begin{align*}
 &2(d \theta_{\alpha} \wedge J B^{\alpha}) (\tilde{X},\tilde{Y},\tilde{Z})\\
=& 2((\eta-\pi^{\ast} \gamma^{\alpha}_{2})
\wedge J B^{\alpha})(\tilde{X},\tilde{Y},\tilde{Z}) \\
=&-2\gamma^{\alpha}_{2}(X)J B^{\alpha}_{\tilde{Y}} \tilde{Z}
+2\gamma^{\alpha}_{2}(Y)J B^{\alpha}_{\tilde{X}} \tilde{Z}\\
=&-2\gamma^{\alpha}_{2}(X)J ((\bar{B}^{\alpha}_{Y} Z)^{\widetilde{}}
+\bar{c}^{\alpha}(Y,Z)\xi+\bar{c}^{\alpha}_{\bar{J}}(Y,Z)Z_{M})\\
&+2\gamma^{\alpha}_{2}(Y)J ((\bar{B}^{\alpha}_{X} Z)^{\widetilde{}}
+\bar{c}^{\alpha}(X,Z)\xi+\bar{c}^{\alpha}_{\bar{J}}(X,Z)Z_{M})\\
=&((-2\gamma^{\alpha}_{2} \wedge \bar{J}\bar{B}^{\alpha}) (X,Y,Z))^{\widetilde{}}
-2 \gamma^{\alpha}_{2}(X) \bar{c}^{\alpha}(Y,Z)Z_{M}
+2 \gamma^{\alpha}_{2}(X) \bar{c}^{\alpha}_{\bar{J}}(Y,Z)\xi \\
&+2 \gamma^{\alpha}_{2}(Y) \bar{c}^{\alpha}(X,Z)Z_{M}
-2 \gamma^{\alpha}_{2}(Y) \bar{c}^{\alpha}_{\bar{J}}(X,Z)\xi  \\
=&((-2 \gamma^{\alpha}_{2} \wedge \bar{J} \bar{B}^{\alpha})(X,Y,Z))^{\tilde{}} \\
 &+2 (\gamma^{\alpha}_{2} \wedge \bar{c}^{\alpha}_{\bar{J}})(X,Y,Z)\xi
-2 (\gamma^{\alpha}_{2} \wedge \bar{c}^{\alpha})(X,Y,Z)Z_{M}.
\end{align*}
Therefore we proved that $d^{D} B^{\alpha}=2 d \theta_{\alpha} \wedge JB_{\alpha}$. 
Then we have 
\begin{align*}
d^{D}A
&=d^{D}A^{\alpha}=(De^{-2 \theta_{\alpha}J})\wedge B^{\alpha}+e^{-2 \theta_{\alpha}J}(d^{D}B^{\alpha}) \\
&=-2 d \theta_{\alpha} J e^{-2 \theta_{\alpha}J} \wedge B^{\alpha}
+ 2 e^{-2 \theta_{\alpha}J} d \theta_{\alpha} \wedge JB_{\alpha}=0
\end{align*}
on $\pi^{-1}(\bar{U}_{\alpha})$ for all $\alpha \in \Lambda$.
\end{proof}

Then it holds that $(D,-\frac{1}{2}JA) \in {\cal C}_{(M,J)}$. 
By Lemma \ref{char_sc}, we obtain a special connection $\nabla= D+(1/2)JA \in {\cal S}_{(M,J)}$ on $(M,J)$.

\begin{lemma}
$\nabla \xi = \mathrm{Id}$.
\end{lemma}

\begin{proof}
This follows from $A_{\xi}=0$ and $D \xi=\mathrm{Id}$.
\end{proof}

We can conclude the following theorem. 

\begin{theorem}\label{char:thm}
Let $(\bar M , \bar J, \mathcal{P})$ be a complex manifold with $\dim \bar{M} \geq 4$ and 
a c-projective structure $\mathcal{P}$, and 
$\pi_{S}:S \to \bar{M}$ be a principal $S^{1}$-bundle with a connection 
$\eta$ of PSCB type in the sense of Definition \ref{intrinsic_3}. 
Then $M:=S \times \mathbb{R}^{>0}$ admits a conical special complex structure $(J,\nabla,\xi)$ such that 
$\bar{M}$ is a projective special complex manifold with the corresponding principal $\mathbb{C}^{\ast}$-bundle 
$\pi:=\pi_{S} \circ \pi_{1}:M \to \bar{M}$ 
and the canonical c-projective structure $\bar{\mathcal{P}}(\pi)$ on $\bar{M}$ 
coincides with $\mathcal{P}$, where $\pi_{1}:S \times \mathbb{R}^{>0} \to S$ is the projection onto 
the first factor.  
\end{theorem}

By \eqref{def_c_con}, the fundamental $(0,2)$-tensor of the constructed 
principal $\mathbb{C}^{\ast}$-bundle 
$\pi : M \to \bar{M}$ is given by (6-1) in Definition \ref{intrinsic_3}. 
We turn to the simplest example described in Example \ref{hopf_bdl}. 
The Hopf fibration $S^{2n+1} \to \mathbb{C}P^{n}$ (with its canonical connection) is of PSCB type  
and the constructed conical trivial special complex manifold is $\mathbb{C}^{n+1} \backslash \{ 0 \}$ with 
the usual $\mathbb{C}^{\ast}$-action. 
The canonical c-projective structure coincides with the c-projective structure $[\bar{D}^{\bar g}]$, where 
$\bar{D}^{\bar g}$ is the Levi-Civita connection of the 
Fubini-Study metric $\bar g$. See also Corollary \ref{cflat_ex} 
and Example \ref{ex_cp_n}. 
The following corollary characterizes the Hopf fibration from our viewpoint.

\begin{corollary}\label{hopf_uni}
Any principal $S^{1}$-bundle over 
$(\mathbb{C}P^{n}, \bar{J}, [\bar{D}^{\bar g}])$ of PSCB type 
with $[ \bar{\cal B}_{\bar{J}} ]=0$ is isomorphic to 
the Hopf fibration $S^{2n+1} \to \mathbb{C}P^{n}$ as a principal $S^{1}$-bundle if $n \geq 2$. 
\end{corollary}

\begin{proof}
The set $P(\bar{M},S^{1})$ of all principal $S^{1}$-bundles over $\bar{M}$ forms  
an additive group. The map from $P(\bar{M},S^{1})$ to $H^{2}(\bar{M}, \mathbb{Z})$, given by  
$S \mapsto \mbox{(its characteristic class)}$, is homomorphism (\cite[Theorems 1 and 2]{Kob}). 
In particular, 
if $\bar{M}=\mathbb{C}P^{n}$, 
then the homomorphism is actually an isomorphism and 
$H^{2}(\mathbb{C}P^{n}, \mathbb{Z}) \cong \mathbb{Z}$, and hence 
\[ P(\mathbb{C}P^{n},S^{1}) \cong H^{2}(\mathbb{C}P^{n}, \mathbb{Z}) \cong \mathbb{Z}.  \]
The Hopf fibration $S^{2n+1} \to \mathbb{C}P^{n}$ corresponds to $1 \in \mathbb{Z}$ under this identification. See \cite[Theorem 8]{Kob}. 
Let $\alpha \in H^{2}(\bar{M},\mathbb{Z)}$ denote the characteristic class of the  Hopf fibration, which corresponds to $1 \in \mathbb{Z}$. 
Note that $\alpha$ is represented as $[ -\frac{1}{2 \pi} (P^{\bar{D}^{\bar{g}}}_{\bar{J}})^{a} ]$ and 
the first Chern class of $\mathbb{C}P^{n}$ is expressed by $c_{1}(\mathbb{C}P^{n})=(n+1)\alpha$.  
Consider an $S^{1}$-bundle $S \to \mathbb{C}P^{n}$ of PSCB type with $[ \bar{\cal B}_{\bar{J}} ]=0$, 
which corresponds to $k \in \mathbb{Z}$. 
Therefore we have $S=k \cdot S^{2n+1}=S^{2n+1}+\cdots+S^{2n+1}$($k$ times) 
with the characteristic class $k \alpha$. 
Because $S=k \cdot S^{2n+1} \to \mathbb{C}P^{n}$ is PSCB type, 
Theorem \ref{char:thm} shows that $(\mathbb{C}P^{n}, \bar{J}, [\bar{D}^{\bar g}])$ is a 
projective special complex manifold associated with $S=k \cdot S^{2n+1} \to \mathbb{C}P^{n}$ 
and allows us to apply results in Section \ref{subsec23}. 
The equation \eqref{trB_coh} gives   
\[ \left[ \frac{1}{8 \pi} \, \bar{\cal B}_{\bar{J}} \right]
=-(n+1)\alpha+(n+1)(k \alpha) 
=(n+1) (k-1)\alpha \in H^{2}(\mathbb{C}P^{n}, \mathbb{Z}) \cong \mathbb{Z}, \]
which implies $k=1$. This concludes the proof.
\end{proof}

Next corollary shows the realization of $(\bar M , \bar{J}, [\bar{D}])$ with $W^{[\bar{D}]}=0$ 
as a projective special complex manifold. 
Another non-trivial realization of a (four-dimensional) c-projectively flat manifold is given in Proposition \ref{4_dim_ex_0}.

\begin{corollary}[C-projectively flat manifold]\label{cflat_ex}
Let $(\bar M , \bar{J}, [\bar{D}])$ 
be a complex manifold endowed with the c-projective structure induced by a torsion-free complex connection $\bar{D}$ such that $\dim \bar{M}=2n \geq 4$ and 
$W^{[\bar{D}]}=0$. Additionally, we assume that the Cotton-York tensor 
$C^{\bar{D}}:=d^{\bar{D}}P^{\bar{D}}$ vanishes when $2n=4$. 
If $[\frac{1}{2 \pi}(P^{\bar{D}}_{\bar{J}})^{a}] 
\in H^{2}(\bar{M},\mathbb{Z})$, 
then $\bar{M}$ can be realized as a projective 
special complex manifold by a $\mathbb{C}^{\ast}$-quotient 
of a trivial conical special complex manifold $(M,J,\nabla,\xi)$. 
\end{corollary}

\begin{proof}
Since $[-\frac{1}{2 \pi}(P^{\bar{D}}_{\bar{J}})^{a}] 
\in H^{2}(\bar{M},\mathbb{Z})$, 
there exists an $S^{1}$-bundle $\pi_{S}:S \to \bar{M}$ with a connection $\eta$ 
whose curvature form $d \eta=\pi_{S}^{\ast}(-(P^{\bar{D}}_{\bar{J}})^{a})$. 
We can verify that 
$\pi_{S}:S \to \bar{M}$ is of PSCB type with $\bar{B}^{\alpha}=0$ and $\bar{c}^{\alpha}=0$ for all 
$\alpha \in \Lambda$. 
If $\dim \bar{M} \geq 6$, $W^{ [\bar{D}]}=0$ implies vanishing of 
the Cotton-York tensor 
$C^{\bar{D}}$ with the help of the second Bianchi identity. 
Moreover, we have $d \eta=\pi_{S}^{\ast}(-(P^{\bar{D}}_{\bar{J}})^{a})=2\pi_{S}^{\ast}((\bar{a}_{\bar{J}})^{a})$ 
and $2d^{\bar{D}} \bar{a}=-C^{\bar{D}}=0$. Here we note that $\bar{a}=-\frac{1}{2}P^{\bar{D}}$.   
Therefore Theorem \ref{char:thm} shows that $\bar{M}$ is realized as a projective 
special complex manifold by a $\mathbb{C}^{\ast}$-quotient 
of a conical special complex manifold $(M,J,\nabla,\xi)$ with $A=0$. 
\end{proof}

\begin{remark}
{\rm 
If $W^{[D]}=0$, 
the first Chern class $c_{1}(\bar{M})$ is given by
$c_{1}(\bar{M})=[-\frac{n+1}{2 \pi}(P^{\bar{D}}_{\bar{J}})^{a}] $ due to \eqref{W_cur}. 
Compare with Proposition \ref{rep_fist}. All the assumptions in Corollary \ref{cflat_ex} 
are expressed in terms of  c-projective geometric objects by Lemma \ref{rho_inv}. 
}
\end{remark}

\begin{example}\label{ex_cp_n}
When $(\bar M , \bar{J}, [\bar{D}])=(\mathbb{C}P^{n}, \bar{J}, [\bar{D}^{\bar g}])$ with 
$n \geq 2$, then we have $W^{[\bar{D}^{\bar g}]}=0$ 
and 
$[\frac{1}{2 \pi}(P^{\bar{D}}_{\bar{J}})^{a}] 
\in H^{2}(\bar{M},\mathbb{Z})$. 
The $S^{1}$-bundle $\pi_{S}:S \to \mathbb{C}P^{n}$ in the proof of Corollary \ref{cflat_ex}
is isomorphic to the Hopf fibration due to $[ \bar{\cal B}_{\bar{J}} ]=0$, as shown in Corollary \ref{hopf_uni},
and the constructed $M$ is $\mathbb{C}^{n+1} \backslash \{ 0 \}$ with 
the usual $\mathbb{C}^{\ast}$-action.
Therefore the construction in Corollary \ref{cflat_ex}
for $(\mathbb{C}P^{n}, \bar{J}, [\bar{D}^{\bar g}])$  
results the standard projection 
$\mathbb{C}^{n+1} \backslash \{ 0 \} \to \mathbb{C}P^{n}$. 
\end{example}

As an application, we have the following corollary.  
In this corollary, we describe sufficient conditions  (which locally are also necessary) 
for a complex manifold with c-projective structure to 
admit a realization as a projective special complex manifold.  The conditions  
are given purely in terms of a complex manifold 
$\bar M$ without reference to a bundle over $\bar M$.

\begin{corollary}\label{app_1}
Let $(\bar{M},\bar{J},\bar{D})$ be a complex manifold with $\dim \bar{M} \geq  4$ and 
a complex connection $\bar{D}$. 
Let $\bar{B}$ be a symmetric 
$(1,2)$-tensor which anti-commutes 
with $\bar{J}$
and $\bar{c}$ be a symmetric $(0,2)$-tensor on $\bar{M}$ such that  
$\bar{J}^{\ast}\bar{c}=-\bar{c}$. Set 
\[ 
\bar{a}=\frac{1}{8(n+1)} \bar{\cal B}
                 -\frac{1}{2} P^{\bar{D}},  \]
                 where the tensor $P^{\bar{D}}$ is the Rho tensor defined in \eqref{P-tensor:eq} and 
                 $\bar{\mathcal B}(X,Y)= \mathrm{Tr} (\bar B_X\bar B_Y)$, $X,Y\in T\bar M$. 
We assume that the c-projective Weyl curvature $W^{[\bar{D}]}$ of $[\bar{D}]$ satisfies  
\begin{align}\label{c1}
W^{ [\bar{D}]}=& -\frac{1}{4} [\bar{B},\bar{B}] 
-\frac{1}{4(n+1)} \bar{\cal B}_{\bar{J}} \otimes \bar{J} 
+\frac{1}{8(n+1)} \bar{\cal B} \wedge \mathrm{Id}
-\frac{1}{8(n+1)} \bar{\cal B}_{\bar{J}} \wedge \bar{J}.  
\end{align}
If 
\begin{align}\label{c2}
(d^{\bar{D}} \bar{a})(X,Y,Z)
=-\frac{1}{4}(\bar{c}(X,\bar{B}_{Y}Z) 
-\bar{c}(Y,\bar{B}_{X}Z)) 
\end{align}
for all $X$, $Y$, $Z \in T\bar{M}$
and there exists a one-form $\gamma$ on $\bar{M}$
such that 
\begin{align}
&d \gamma = 2 (\bar{a}_{\bar{J}})^{a}, \label{c3} \\ 
&d ^{\bar{D}} \bar{B}  
=- 2\gamma
\wedge (\bar{J} \bar{B})  +\bar c \wedge \mathrm{Id} + \bar c_{\bar J}\wedge \bar J, \label{c4}\\
&\bar{a}(X,\bar{B}_{Y}Z)-\bar{a}(Y,\bar{B}_{X}Z)
+(d^{\bar{D}} \bar{c})(X,Y,Z)
=2(\gamma \wedge \bar{c}_{\bar{J}}) (X,Y,Z), \label{c5}
\end{align}
then $\bar{M}$ can be realized as a projective 
special complex manifold by the $\mathbb{C}^{\ast}$-quotient 
of a conical special complex manifold $(M=\bar{M} \times S^{1} \times \mathbb{R}^{>0},J,\nabla,\xi)$. Moreover 
the c-projective structure $[\bar{D}]$ on $\bar{M}$ coincides 
with the one canonically induced from 
the conical special complex structure $(J,\nabla,\xi)$. 
\end{corollary}

\begin{proof}
Consider the trivial $S^{1}$-bundle 
$S:=\bar{M} \times S^{1}$ over $\bar{M}$ with the connection 
\begin{equation}\label{eta_gamma:eq} \eta = \pi_{S}^{\ast}( \gamma \otimes \sqrt{-1})  
+ d \theta \otimes \sqrt{-1}. \end{equation}
Then $\pi_{S} : S \to \bar{M}$ 
is a principal $S^{1}$-bundle with connection 
$\eta$ of PSCB type.
\end{proof}

\begin{example}[(Affine) special complex manifolds as projective special complex manifolds]\label{asc_ex}

Let $(\bar{M},\bar{J},\bar{\nabla})$ be a special complex manifold of $\dim \bar{M}=2n$. 
Taking 
$\bar{B} :=\bar{A} =\bar{\nabla} \bar{J}$, we have 
\[ R^{\bar{D}}=-\frac{1}{4}[\bar{B},\bar{B}], \,\,\, Ric^{\bar{D}} =\frac{1}{4}  \bar{\cal B}, \]
where $\bar{D}=\bar{\nabla}-\frac12 \bar{J}\bar{A}$. Since $Ric^{\bar{D}}$ is symmetric and $\bar{J}$-Hermitian,
$P^{\bar{D}}=\frac{1}{n+1}Ric^{\bar{D}}=\frac{1}{4(n+1)}\bar{\cal B}$. Then we obtain 
\begin{align*}
W^{[\bar{D}]}=& -\frac{1}{4} [\bar{B},\bar{B}] 
-\frac{1}{4(n+1)} \bar{\cal B}_{\bar{J}} \otimes \bar{J} 
+\frac{1}{8(n+1)} \bar{\cal B} \wedge \mathrm{Id}
-\frac{1}{8(n+1)} \bar{\cal B}_{\bar{J}} \wedge \bar{J}
\end{align*}
and $\bar{a}=0$. Taking the trivial the $S^{1}$-bundle with the trivial connection $\gamma=0$ 
and $\bar{c}=0$, 
we can obtain 
a conical special complex manifold $(M,J,\nabla)$ whose $\mathbb{C}^{\ast}$-quotient is $\bar{M}$ 
by Corollary \ref{app_1}. 
Therefore any special complex manifold can be realized as a projective special complex manifold.    
Note that the first Chern class of a special complex manifold $(\bar{M},\bar{J},\bar{\nabla})$ is given by 
\[ c_{1}(\bar{M})=\left[ \frac{1}{4 \pi} \mathrm{Tr} R^{\bar{D}} \circ \bar{J} \right]
=-\left[ \frac{1}{8 \pi}\bar{\cal B}_{\bar{J}} \right]. \]
If a special complex manifold $\bar{M}$ is realized as a projective special complex manifold, then 
$(n+1)\left[ \frac{1}{\pi } (\bar{a}_{\bar{J}})^{a} \right]=0$ by \eqref{trB_coh}. 
\end{example}

Other examples will be given in Section \ref{low_ex}.

\subsection{$\dim \bar{M}=2$ case}\label{dim_2:sec}

In this subsection we define $S^1$-bundles $S\to \bar M$ of PSCB type  for 
$\dim \bar{M}=2$.

\begin{definition}\label{intrinsic_2}
Let $(\bar M , \bar J, \bar{D})$ be a complex manifold with real dimension 
$\dim \bar{M}=2$ and a torsion-free complex connection $\bar{D}$, 
and 
$\pi_{S}:S \to \bar{M}$ be a principal $S^{1}$-bundle with a connection 
$\eta$. 
Let $\{ \bar{U}_{\alpha} \}_{\alpha \in \Lambda}$ be an open covering of 
$\bar{M}$ 
with local trivializations  
$\pi_{S}^{-1}(\bar{U}_{\alpha}) \cong \bar{U}_{\alpha} \times S^{1}$ for 
each $\alpha \in \Lambda$ and $e^{f^{2}_{\alpha \beta} \sqrt{-1}}$ be the transition functions associated 
with $\{ \bar{U}_{\alpha} \}_{\alpha \in \Lambda}$. 
Let $\{ \gamma_{2}^{\alpha} \}_{\alpha \in \Lambda}$ be the collection of local $1$-forms on the open sets $\bar{U}_{\alpha}$ 
associated to the connection $1$-form $\eta$. 
We say that $\pi_{S}:S\to \bar{M}$ is of 
{\cmssl PSCB type} with 
a $(0,2)$-tensor $\bar{a}$ and collections 
$\{\bar{B}^{\alpha} \}_{\alpha \in \Lambda}$,  
$\{\bar{c}^{\alpha} \}_{\alpha \in \Lambda}$ 
if the following conditions hold: \\
(1) \, $\bar{B}^{\alpha}$ is a symmetric $(1,2)$-tensor on $\bar{U}_{\alpha}$ 
and anti-commutes with $\bar{J}$ for all $\alpha \in \Lambda$, \\
(2) \, 
$\bar{B}^{\beta} =e^{-2 f^{2}_{\alpha \beta} \bar{J}} \bar{B}^{\alpha}$
on $\bar{U}_{\alpha} \cap \bar{U}_{\beta}$, \\
(3) \, $\bar{c}^{\alpha}$ is a symmetric $(0,2)$-tensor on $\bar{U}_{\alpha}$ with 
$\bar{J}^{\ast} \bar{c}^{\alpha}=-\bar{c}^{\alpha}$
for all $\alpha \in \Lambda$, \\
(4) \, $\bar{c}^{\beta}+\sqrt{-1} \, \bar{c}^{\beta}_{\bar{J}} 
=e^{-2 f^{2}_{\alpha \beta} \sqrt{-1}} (\bar{c}^{\alpha}+\sqrt{-1} \, \bar{c}^{\alpha}_{\bar{J}} )$ on $\bar{U}_{\alpha} \cap \bar{U}_{\beta}$, \\ 
(5) \, the curvature of $\eta$ satisfies $d \eta=2 \pi_{S}^{\ast}(\bar{a}_{\bar{J}})^{a}$,\\ 
(6) \, the torsion free complex connection $\bar{D}$ satisfies that \\
(6-1) its curvature tensor 
$R^{\bar{D}}$ is of the form  
\begin{align*}
R^{\bar{D}} 
=& -\frac{1}{4}[\bar{B}^{\alpha},\bar{B}^{\alpha}]
 +2 \bar{a}^{a} \otimes \mathrm{Id} - 2 (\bar{a}_{\bar{J}})^{a} \otimes \bar{J} 
+ \bar{a} \wedge \mathrm{Id} - \bar{a}_{\bar{J}} \wedge \bar{J},  
\end{align*}
(6-2) \, $(d^{\bar{D}} \bar{a})(X,Y,Z)=-(1/4) (\bar{c}^{\alpha}(X,\bar{B}^{\alpha}_{Y}Z) 
-\bar{c}^{\alpha}(Y,\bar{B}^{\alpha}_{X}Z))$ 
for all $\alpha \in \Lambda$ and $X$, $Y$, $Z \in T\bar{U}_{\alpha}$, \\
(6-3) \, $d ^{\bar{D}} \bar{B}^{\alpha}  =- 2 \gamma_{2}^{\alpha} 
\wedge \bar{J} \bar{B}^{\alpha}$ 
$+\bar c^\alpha \wedge \mathrm{Id} + \bar c^\alpha_{\bar J}\wedge \bar J$
all $\alpha \in \Lambda$ \\
and \\
(6-4) \, $\bar{a}(X,\bar{B}^{\alpha}_{Y}Z)-\bar{a}(Y,\bar{B}^{\alpha}_{X}Z)
+(d^{\bar{D}} \bar{c}^{\alpha})(X,Y,Z)
=2(\gamma_{2}^{\alpha} \wedge \bar{c}^{\alpha}_{\bar{J}}) (X,Y,Z)$ for all $\alpha \in \Lambda$ and $X$, $Y$, $Z \in T\bar{U}_{\alpha}$. \\
\end{definition}

We note that Definition \ref{intrinsic_2} is valid for the case of $\dim \bar{M} \geq 4$.  
In a similar way as in the proof of Theorem \ref{char:thm}, we obtain the following theorem.  

\begin{theorem}\label{char:thm_two_dim}
Let $(\bar M , \bar J, \bar{D})$ be a complex manifold of real dimension $\dim \bar{M}=2$
with a complex connection $\bar{D}$ and 
$\pi_{S}:S \to \bar{M}$ be a principal $S^{1}$-bundle with a connection $\eta$. 
If $\pi_{S}:S\to \bar{M}$ is of PSCB type  
in the sense of Definition \ref{intrinsic_2}, 
then $M:=S \times \mathbb{R}^{>0}$ admits a conical special complex structure $(J,\nabla,\xi)$ such that 
$\bar{M}$ is a projective special complex manifold with the corresponding principal $\mathbb{C}^{\ast}$-bundle 
$\pi:=\pi_{S} \circ \pi_{1}:M \to \bar{M}$, where $\pi_{1}:S \times \mathbb{R}^{>0} \to S$ is the projection onto 
the first factor. (Note that the 
c-projective structure on $\bar M$ is unique for dimensional reasons.)
\end{theorem}

\begin{remark}
{\rm 
We can give an intrinsic characterization of a projective special complex manifold $(\bar{M},\bar{J},\bar{D})$ with $\dim \bar{M}=2$ purely in terms of objects on $\bar{M}$ as in Corollary~\ref{app_1}. 
}
\end{remark}

\begin{remark}\label{rem_cp}
{\rm
Note that the equation \eqref{trB_coh} holds for $\dim \bar{M}=2$.  
Therefore, similarly to Corollary \ref{hopf_uni}, we see that
any principal $S^{1}$-bundle over 
$(\mathbb{C}P^{1}, \bar{J}, \bar{D}^{\bar g})$ of PSCB type in the sense of Definition \ref{intrinsic_2}
with $[ \bar{\cal B}_{\bar{J}} ]=0$ is isomorphic to 
the Hopf fibration $S^{3} \to \mathbb{C}P^{1}$ as a principal $S^{1}$-bundle, 
where $\bar{g}$ is the Fubini-Study metric on $\mathbb{C}P^{1}$. 
}
\end{remark}

\subsection{The special case of projective special K\"{a}hler manifolds}\label{SpKah}

In this subsection we apply Theorems \ref{char:thm} and \ref{char:thm_two_dim}
in the case when the base is equipped with a K\"ahler metric recovering the characterization of projective 
special K{\"a}hler manifolds in terms of $S^1$-bundles in \cite{M}. Moreover, we give an intrinsic characterization of 
projective special K\"ahler manifolds in the spirit of our Corollary~\ref{app_1}, which does not involve any bundle over $\bar M$. 
At first, we recall the definition of conical special K{\"a}hler and projective special K{\"a}hler manifolds.

\begin{definition}
We call $(M,g,J,\omega,\nabla,\xi)$ a {\cmssl conical special K{\"a}hler manifold} if 
\begin{itemize}
\item $(M,J,\nabla,\xi)$ is a conical special complex manifold such that $D^{g} \xi=\mathrm{Id}$, where 
$D^{g}$ is the Levi-Civita connection of $g$,  
\item $(M,J,g)$ is an indefinite K{\"a}hler manifold and $\omega$ is the K{\"a}hler form with $\nabla \omega=0$, 
\item $g(\xi,\xi)$ is nowhere-vanishing and $\langle \xi , J \xi \rangle$ is negative definite and its orthogonal complement is positive
definite in $TM$. 
\end{itemize}
\end{definition}

We set $r:=\sqrt{-g(\xi,\xi)}$ and $S:=r^{-1}(1) \subset M$. 
Then it holds that $M=S \times \mathbb{R}^{>0}$. 
A K{\"a}hler manifold $(\bar{M},\bar{J},\bar{g})$ is said to be 
a {\cmssl projective special K{\"a}hler} if  $\pi:M \to \bar{M}$ is a 
$\mathbb{C}^{\ast}$-bundle such that 
the principal $\mathbb{C}^{\ast}$-action is generated by the holomorphic vector field 
$\xi-\sqrt{-1} J \xi$ and $\bar{M}$ is the K{\"a}hler quotient by the induced $\mathrm{U}(1)$-action, that is, 
$\bar{M}=S/\langle J\xi \rangle(=M/\mathbb{C}^{\ast})$. 
There exists a unique principal connection $\varphi$ on $M$ whose horizontal subbundle 
is orthogonal to the vertical subbundle with respect to $g$. 
Note that any conical special K\"ahler manifold is of horizontal type 
with respect to $\varphi$, see Definition \ref{hortype:def}. 

We have 
\[ \xi=r \partial_{r}, \,\,\,  
g=r^{2} \pi^{\ast} \bar{g}-r^{2} \eta \otimes \eta -dr \otimes dr,  \]
where $\eta$ is the $\mathfrak{u}(1)$-component of $\varphi$. See \cite{M} for example.  
Using the Koszul formula for the Levi-Civita connection $D^{g}$ of $g$, we obtain 
$g(D^{g}_{\tilde{X}} \tilde{Y},\xi)=-r^{2} \bar{g}(X,Y)$ for 
$X$, $Y \in \Gamma(T \bar{M})$. 
Therefore the fundamental $(0,2)$-tensor $\bar{a}$ coincides with the Riemannian metric $\bar{g}$;
\begin{align}\label{metric_fund}
\bar{a}=\bar{g} 
\end{align}
and it holds that 
\[ \left[ \, \frac{\bar{\omega}}{\pi} \, \right] \in H^{1,1}(\bar{M},\mathbb{Z}) \]
from \eqref{int_1}. 
Note that  
$\pi:(M,g) \to (\bar{M},\bar{g})$ is not a pseudo-Riemannian submersion, however it is 
an affine submersion with respect to the Levi-Civita connections.   
By \eqref{metric_fund} and  \eqref{bar_a},
\begin{align}\label{ineq_scal}
\bar{g}=\frac{1}{8(n+1)} \bar{\cal B} -\frac{1}{2(n+1)} Ric^{\bar{D}^{\bar{g}}}, 
\end{align}
where $\bar{D}^{\bar{g}}$ is the Levi-Civita connection of $\bar{g}$. 
Consequently the equation \eqref{ineq_scal} gives an inequality
\begin{align}\label{ineq_321}
4n(n+1)+\mbox{(the scalar curvature of $\bar{g}$)}\left(=\frac{1}{4}  \mathrm{Tr}_{\bar g}\,   \bar{\cal B}\right) \ge  0, 
\end{align}
where $\mathrm{Tr}_{\bar g}\,   \bar{\cal B} \ge 0$ follows from 
\[ \bar{\cal B}(X,X) = \mathrm{Tr}\, \bar B_X^2 =  \mathrm{Tr}\, B_{\tilde{X}}^2 
= \mathrm{Tr}\, A_{\tilde{X}}^2 = \mathrm{Tr}\, (JA_{\tilde{X}})^2\ge 0,\] since $JA_{\tilde{X}}=2(\nabla -D)_{\tilde{X}}$ is symmetric with respect to $g$ (both connections 
preserve the K\"ahler form and $JA_{\tilde{X}}$ anticommutes with $J$). Here $\tilde{X}$ denotes the horizontal lift of $X$. 
Equality holds in \eqref{ineq_321} if and only if $B^{\alpha}=0$ for all $\alpha \in \Lambda$.
This occurs only when $\bar{g}$ has constant holomorphic curvature $-4$ by \eqref{curv}.
See \cite[Corollary 7.4 and Proposition 9.5]{M}.  

Conversely, if we assume that conditions
(1), (2), (6-1), (6-3) hold and replace a curvature condition 
for the Levi-Civita connection in (5)
and symmetric property for $\{\bar{B}^{\alpha} \}_{\alpha \in \Lambda}$ in (6-4) for 
a K{\"a}hler manifold $(\bar{M},\bar{g},\bar{J})$, then 
we can construct a conical special K{\"a}hler manifold $(M,g,J,\omega,\nabla,\xi)$ 
and its $\mathbb{C}^{\ast}$-quotient is $(\bar{M},\bar{g},\bar{J})$. Precisely we can obtain 
the following corollary. 

\begin{corollary}[Theorem 7.6 in \cite{M}]\label{proj_kaehler_char} 
Let $(\bar{M},\bar{J},\bar{g})$ be a K{\"a}hler manifold 
and $\pi_{S}:S \to \bar{M}$ be a principal $S^{1}$-bundle with a connection $\eta$. 
Let $\{ \bar{U}_{\alpha} \}_{\alpha \in \Lambda}$ be an open covering of 
$\bar{M}$ 
with local trivializations  
$\pi_{S}^{-1}(\bar{U}_{\alpha}) \cong \bar{U}_{\alpha} \times S^{1}$ for 
each $\alpha \in \Lambda$ and $e^{f^{2}_{\alpha \beta} \sqrt{-1}}$ be the transition functions associated 
with $\{ \bar{U}_{\alpha} \}_{\alpha \in \Lambda}$. 
Let $\{ \gamma_{2}^{\alpha} \}_{\alpha \in \Lambda}$ be a collection of a local $1$-form on $\bar{U}_{\alpha}$ 
associated to the connection $1$-form $\eta$ and 
$\{\bar{B}^{\alpha} \}_{\alpha \in \Lambda}$ be 
a collection of a local $(1,2)$-tensor on $\bar{U}_{\alpha}$. 
We assume that \\
(1)' \, $\bar{g}(\bar{B}^{\alpha}_{\, \cdot \, } \, \cdot \,\, , \, \cdot \,)$ is totally symmetric 
and $\bar{B}^{\alpha}$ anti-commutes with $\bar{J}$ for all $\alpha \in \Lambda$, \\
(2)' \, 
$\bar{B}^{\beta} =e^{-2 f^{2}_{\alpha \beta} \bar{J}} \bar{B}^{\alpha}$
on $\bar{U}_{\alpha} \cap \bar{U}_{\beta}$, \\
(5)'\, the curvature tensor $R^{\bar{D}^{\bar{g}}}$ of the Levi-Civita connection $\bar{D}^{\bar{g}}$ 
for the K{\"a}hler metric $\bar{g}$ 
is of the form  
\begin{align*}
R^{\bar{D}^{\bar{g}}} 
=& -\frac{1}{4}[\bar{B}^{\alpha},\bar{B}^{\alpha}]
 - 2 \bar{g}_{\bar{J}}  \otimes \bar{J} 
+ \bar{g} \wedge \mathrm{Id} - \bar{g}_{\bar{J}} \wedge \bar{J},  
\end{align*}
(6-1)' \, the curvature of $\eta$ satisfies 
$d \eta=2 \pi_{S}^{\ast}(\bar{\omega})$ $(\bar{\omega}=\bar{g}_{J})$, \\
(6-3)' \, $d ^{\bar{D}} \bar{B}^{\alpha}  =- 2(\gamma_{2}^{\alpha} 
\wedge \bar{J} \bar{B}^{\alpha})$ all $\alpha \in \Lambda$. \\
Then $M:=S \times \mathbb{R}^{>0}$ admits a conical special K{\"a}hler special structure 
$(g,J,\omega,\nabla,\xi)$ such that 
$\bar{M}$ is a projective special K{\"a}hler manifold with the corresponding principal 
$\mathbb{C}^{\ast}$-bundle 
$\pi:=\pi_{S} \circ \pi_{1}:M \to \bar{M}$, where $\pi_{1}:S \times \mathbb{R}^{>0} \to S$ is the projection onto 
the first factor. 
\end{corollary}

\begin{proof}
Consider a metric 
$g:=r^{2}\pi_{S}^{\ast} \bar{g} - r^{2} \eta \otimes \eta - dr \otimes dr$ on $M$. 
By Theorems \ref{char:thm} and \ref{char:thm_two_dim} with $\bar{c}_{\alpha}=0$ for all $\alpha \in \Lambda$, 
$(M,J,\nabla,\xi)$ constructed as in Section \ref{dim_ge4:sec} is a conical special complex manifold.
The complex connection $D$ defined by (\ref{def_c_con})
is the Levi-Civita connection of $g$. Indeed, it is easy to 
check $Dg=0$, hence $D^{g}\xi=D\xi=\mathrm{Id}$. 
Finally, it is also easy to obtain 
\begin{align*}
\nabla \omega &= (D^{g}-\frac{1}{2}JA)\omega =-\frac{1}{2}JA \cdot \omega 
                       =\frac{1}{2}\omega (JA \, \cdot \, ,\, \cdot \,) 
                       + \frac{1}{2}\omega (\, \cdot \, , JA \, \cdot \,) \\
                     &=\frac{1}{2} g (JA \, \cdot \, ,J \, \cdot \,) 
                       -\frac{1}{2} g (\, \cdot \, , A \, \cdot \,)=0,  
\end{align*}
since $A$ is $g$-symmetric due to (1)' (see \eqref{test:def} and \eqref{AB:eq} for the relation between  the tensors $A$ and $\bar B^\alpha$) and $\nabla=D^{g}-\frac{1}{2}JA$.   
\end{proof}

\begin{remark}
{\rm 
The term
\[ 
 2 \bar{g}_{\bar{J}}  \otimes \bar{J} 
- \bar{g} \wedge \mathrm{Id} + \bar{g}_{\bar{J}} \wedge \bar{J}
\]
in (5)' corresponds to the curvature tensor of the Fubini-Study metric on $\mathbb{C}P^{n}$.
}
\end{remark}

Theorems \ref{char:thm} and \ref{char:thm_two_dim} generalize Theorem 7.6 of \cite{M}.  
See also \cite[Remark 7.7]{M}. Similarly to Corollary~\ref{app_1}, 
we give an intrinsic characterization of 
projective special K\"ahler manifolds.

\begin{corollary}\label{int_ch_psk}
Let $(\bar{M},\bar{J},\bar{g})$ be a K{\"a}hler manifold, and let $\bar{B}$ be a symmetric 
$(1,2)$-tensor on $\bar{M}$ which anti-commutes 
with $\bar{J}$ and satisfies that $\bar{g}(\bar{B}_{\, \cdot \, } \, \cdot \,\, , \, \cdot \,)$ is totally symmetric. 
Assume that the curvature $R^{\bar{D}^{\bar{g}}}$ of the Levi-Civita connection $\bar{D}^{\bar{g}}$ of 
the K{\"a}hler metric $\bar{g}$ satisfies  
\begin{align}\label{spk_curv}
R^{\bar{D}^{\bar{g}}}=
 -\frac{1}{4}[\bar{B},\bar{B}]
 - 2 \bar{g}_{\bar{J}}  \otimes \bar{J} 
+ \bar{g} \wedge \mathrm{Id} - \bar{g}_{\bar{J}} \wedge \bar{J} 
\end{align}
and there exists a one-form $\gamma$ on $\bar{M}$
such that 
\begin{align}
&d \gamma = 2 \bar{\omega} \,\, (\bar{\omega}=\bar{g}_{\bar{J}}), \label{psk_gam}\\
&d ^{\bar{D}^{\bar{g}}} \bar{B}  
=- 2\gamma
\wedge (\bar{J} \bar{B}). \label{psk_B}
\end{align}
Then $\bar{M}$ can be realized as a projective 
special K{\"a}hler manifold by the $\mathbb{C}^{\ast}$-quotient 
of a conical special K{\"a}hler manifold 
$(M=\bar{M} \times S^{1} \times \mathbb{R}^{>0},J,g,\omega, \nabla, \xi)$. 
\end{corollary}

\begin{proof}
Consider the trivial $S^{1}$-bundle $S:=\bar{M} \times S^{1}$ with the connection 
$\eta = \pi_{S}^{\ast}( \gamma \otimes \sqrt{-1})  
+ d \theta \otimes \sqrt{-1}. $
Then $(S,\eta)$ satisfies the assumptions of Corollary \ref{proj_kaehler_char} .
\end{proof}

See also Corollary 7.9 in \cite{M}. 

\begin{remark}\label{gr_ch_ms}
{\rm
If $[\bar{\omega}]=0$ and $\bar{M}$ is parallelizable, 
then the conditions for PSCB type can be 
an be expressed using matrix-valued one-forms defined on whole $\bar{M}$. 
For instance, Proposition 6.2 in \cite{MS1} 
provides an intrinsic characterization of a simply connected K{\"a}hler group
$\bar{M}$ with exact K{\"a}hler form $\bar{\omega}=d \kappa$, 
which admits a compatible projective special K{\"a}hler structure. 
Let $p=(p^{i}_{j})$, $q=(q^{i}_{j})$ be symmetric $(n \times n)$-matrix valued one-forms on $\bar{M}$ satisfying 
$p=-q \circ \bar{J}$ as stated in Proposition 6.2 of \cite{MS1}. 
Define 
\[ \bar{B}_{X}X_{i}=2 \left( \sum_{j=1}^{n} p^{j}_{i}(X)X_{j}+q^{j}_{i}(X) \bar{J}X_{j} \right), \] 
for $X \in T\bar{M}$, where $(X_{1},\dots,X_{n},\bar{J}X_{1},\dots, \bar{J}X_{n})$ is a globally left-invariant orthonormal frame on $\bar{M}$. 
The anti-commute for $\bar{B}$ corresponds to the condition $p=-q \circ \bar{J}$.  
The totally symmetry of $\bar{g}(\bar{B}_{\, \cdot \, } \, \cdot \,\, , \, \cdot \,)$ 
is equivalent to $p$, $q$ being symmetric matrices and (6.2) in \cite{MS1}. 
Furthermore, the condition \eqref{spk_curv} for $R^{\bar{D}^{\bar{g}}}$ in Corollary \ref{int_ch_psk} 
is equivalent to (6.3) and (6.4) in \cite{MS1}. 
Lastly, it can be verified that
the condition \eqref{psk_B}
corresponds to (6.5) and (6.6) in \cite{MS1} 
(note that $\gamma=2 \kappa$ for \eqref{psk_gam}). 
Thus Proposition 6.2 in \cite{MS1} is recovered via our Corollary \ref{int_ch_psk}. 
}
\end{remark}

\section{Low dimensional examples}\label{low_ex}

In this section we give examples of 
projective special complex manifolds $\bar{M}$ of dimension $4$ and $2$ using 
the results in Sections \ref{dim_ge4:sec} and \ref{dim_2:sec}. 
Some of these examples carry canonical metrics which turn them into well-known homogeneous 
projective special K\"ahler
manifolds, compare \cite{AC,MS, MS1,M}.

Let $(N,J,D)$ be a surface with a complex structure $J$ and 
a torsion-free complex connection $D$. We assume that the Ricci tensor $Ric^{D}$ is symmetric 
 and  $D Ric^{D}=0$. 
Since $Ric^{D}$ is symmetric, there exists locally a $D$-parallel volume form
$\Omega$ on $N$ and globally if $N$ has finite fundamental group.\footnote{In fact, by the finiteness of $\pi_1(N)$, the holonomy group of the induced flat connection 
in the line bundle $\wedge^2T^*N$ is compact (hence, a subgroup if $\{ \pm 1\}$). Thus there exists a parallel metric in the line bundle. Since $N$ is oriented there is a global (parallel) unit section.}
We choose $\Omega$ compatible with the orientation defined by $J$. 
Because $\mathrm{GL(1,\mathbb{C})} \cap \mathrm{SL}(2,\mathbb{R}) \cong \mathrm{U}(1) \cong \mathrm{SO}(2)$, the connection $D$ is the Levi-Civita connection of a K{\"a}hler metric 
\[ g_\Omega :=\Omega(\, \cdot \, , J \, \cdot \,).\]  
Any unit vector field $u$ with respect to $g_\Omega$ extends uniquely to a $g_\Omega$-orthonormal frame $(u_{1},u_{2}) := (u,Ju)$ compatible with 
the orientation. 
Let $(\theta^{i}_{j})$ be the connection form of  $D$ with respect 
to such a frame. So $\theta^{2}_{1}=-\theta^{1}_{2}$, $\theta^{1}_{1}=\theta^{2}_{2}=0$. 
We give some lemmas which will be used later.   

\begin{lemma}\label{Ricci_two_dim}
We have
\begin{align*}
&R^{D}_{u,Ju}u 
=-Ric^{D}(u,u)Ju, \,\, i.e. \,\,\, R^{D}=Ric^{D}_{J} \otimes J, \\
&Ric^{D}(u,u) =Ric^{D}(Ju,Ju), \\
&Ric^{D}(u,Ju) =Ric^{D}(Ju,u)=0. 
\end{align*}
\end{lemma}

\begin{lemma}\label{str._eq}
$R^{D}_{u,Ju}u=d \theta^{2}_{1}(u,Ju)Ju$, and hence  
 $d \theta^{2}_{1}=-Ric^{D}(u,u)\Omega$. 
\end{lemma}

\begin{lemma}\label{rank_ricci}
The rank of $Ric^{D}$ is constant, that is, $Ric^{D}$ is definite or vanishes. 
\end{lemma}

\begin{proof}
By Lemma \ref{Ricci_two_dim}, we have 
\begin{equation} \label{sign:eq}\det \left(
\begin{array}{cc}
Ric^{D}(u,u) & Ric^{D}(u,J u)\\
Ric^{D}(J u,u) & Ric^{D}(J u,Ju)
\end{array} 
\right)=(Ric^{D}(u,u))^{2}. 
\end{equation}
Since $Ric^{D}$ is $D$-parallel, the rank of $Ric^{D}$ is constant, that is, 
${\rm rank} \, Ric^{D}=2$ or $0$. In view of \eqref{sign:eq}, this means that   
$Ric^{D}$ is definite or vanishes. 
\end{proof}

\subsection{$\dim \bar{M}=4$ (products of surfaces)}
Here we present four examples obtained as an application of Corollary \ref{app_1}.
Let $(N,J,D)$ (resp. $(N^{\prime},J^{\prime},D^{\prime})$) be a surface equipped 
with a complex structure $J$ (resp. $J^{\prime}$)
and a complex connection $D$ (resp. $D^{\prime}$). 
We assume that the Ricci tensors $Ric^{D}$ and $Ric^{D^{\prime}}$ are symmetric and 
$D Ric^{D}=D^{\prime} Ric^{D^{\prime}}=0$.  
Consider the product  
\[ \bar{M}:=N \times N^{\prime}, \,\, \bar{J}:=J+J^{\prime},\,\, \bar{D}:=D+D^{\prime}. \] 
For simplicity, we denote the pullbacks of covariant tensors on 
$N$ and $N^{\prime}$ to $N \times N^{\prime}$
by the same symbols as the original tensors. 
Similarly, the lifts of vectors and vector fields from each factor to $N \times N^{\prime}$
are denoted by the same symbols.
There exist locally parallel volume forms $\Omega$ and $\Omega^{\prime}$ defined on open subsets $U \subset N$ and 
$U^{\prime} \subset N^{\prime}$ and $\bar D$ coincides on $U\times U'$ with the Levi-Civita connection of the product metric $g_\Omega + g_{\Omega'}$.
Take vector fields $u\in \Gamma(TU)$ and $u'\in \Gamma(TU')$ 
such that $g_\Omega (u,u)=g_{\Omega'}(u',u')=1$. 
Then $(u,Ju,u^{\prime},J^{\prime}u^{\prime})$ forms a orthonormal frame on 
$U \times U^{\prime} \subset \bar{M}=N \times N^{\prime}$ with respect to the product metric.

We begin by computing the c-projective Weyl curvature of $[\bar{D}]$ in order to examine condition \eqref{c1} in Corollary \ref{app_1}. The following lemmas are straightforward.

\begin{lemma}\label{e2}
$R^{\bar{D}}=R^{D}+R^{D^{\prime}}$, 
$Ric^{\bar{D}}=Ric^{D}+Ric^{D^{\prime}}$. 
\end{lemma}

\begin{lemma}\label{e3}
The Rho tensor of $\bar{D}$, see \eqref{P-tensor:eq}, is given by 
\begin{align*}
P^{\bar{D}}=\frac{1}{2+1}Ric^{\bar{D}}=\frac{1}{3} \left( Ric^{D}+Ric^{D^{\prime}} \right). 
\end{align*}
In particular, $P^{\bar{D}}$ is symmetric and $\bar{J}$-hermitian. 
\end{lemma}

Let $W^{[\bar{D}]}$ be c-projective Weyl curvature of $[\bar{D}]$. 
Since $P^{\bar{D}}$ is symmetric and $P^{\bar{D}}_{\bar{J}}$ is skew-symmetric, we see that 
\begin{align}\label{wcp1}
W^{[\bar{D}]}_{X,Y}Z=
&R^{\bar{D}}_{X,Y}Z - \frac{1}{3}Ric^{\bar{D}}(X, \bar{J}Y) \bar{J}Z \\
&+\frac{1}{6}( Ric^{\bar{D}}(X,Z)Y-Ric^{\bar{D}}(Y,Z)X) \nonumber \\
&-\frac{1}{6} ( Ric^{\bar{D}}(X,\bar{J}Z)\bar{J}Y-Ric^{\bar{D}}(Y,\bar{J}Z)\bar{J}X) \nonumber
\end{align}
for $X$, $Y$, $Z \in T \bar{M}$.
Since $[W^{[\bar{D}]},\bar J]=0$ (cf.\ Proposition \ref{comm_W_J:prop}), $W^{[\bar{D}]}$ is completely 
determined by the following components.

\begin{lemma}\label{e4}
These equations hold for any vector fields $u\in \Gamma(TU)$, $u'\in \Gamma(TU')${\rm :}
\begin{align*}
&W^{[\bar{D}]}_{u, Ju}u=-\frac{1}{3} Ric^{D}(u,u)Ju, \,\,
W^{[\bar{D}]}_{u^{\prime}, J^{\prime}u^{\prime}}u^{\prime}
=-\frac{1}{3} Ric^{D^{\prime}}(u^{\prime},u^{\prime})J^{\prime} u^{\prime}, \\
&W^{[\bar{D}]}_{u, Ju}u^{\prime}=\frac{1}{3} Ric^{D}(u,u)J^{\prime}u^{\prime}, \,\, 
W^{[\bar{D}]}_{u^{\prime}, J^{\prime}u^{\prime}}u
=\frac{1}{3} Ric^{D^{\prime}}(u^{\prime},u^{\prime})Ju, \\
&W^{[\bar{D}]}_{u, u^{\prime}} u = \frac{1}{6} Ric^{D}(u,u) u^{\prime}, \,\,
W^{[\bar{D}]}_{u, u^{\prime}} u^{\prime} =-\frac{1}{6}Ric^{D^{\prime}}(u^{\prime},u^{\prime})u.
\end{align*}
\end{lemma}

\begin{proof}
Using Lemmas \ref{Ricci_two_dim}, \ref{e2}, \ref{e3} and \eqref{wcp1}, 
we obtain 
\begin{align*}
W^{[\bar{D}]}_{u, Ju}u =
&R^{\bar{D}}_{u,Ju}u + \frac{1}{3}Ric^{\bar{D}}(u, u) \bar{J}u 
+\frac{1}{6}( Ric^{\bar{D}}(u,u)Ju-Ric^{\bar{D}}(Ju,u)u) \\
&-\frac{1}{6} ( -Ric^{\bar{D}}(u,\bar{J}u)u-Ric^{\bar{D}}(Ju,\bar{J}u)\bar{J}u) \\
=
&-Ric^{D}(u,u)Ju + \frac{1}{3}Ric^{D}(u, u) \bar{J}u 
+\frac{1}{6} Ric^{D}(u,u)Ju \\
&+\frac{1}{6} Ric^{D}(Ju,\bar{J}u)\bar{J}u \\
=
&-\frac{1}{3} Ric^{D}(u,u)Ju.
\end{align*}
Similar calculations give the other equations.   
\end{proof}

Next we fix an orthonormal local frame $(u,Ju,u^{\prime},J^{\prime}u^{\prime})$ and consider $\bar{B}$ of a very special form. 
Specifically we define a symmetric $(1,2)$-tensor $\bar{B}$ on $U \times U^{\prime} \subset \bar{M}$, 
which anti-commutes with $\bar{J}$ as follows:
\begin{align}
&\bar{B}_{u}u := a u^{\prime} +b J^{\prime} u^{\prime}, \label{def_B_four}\\
&\bar{B}_{Ju}u =\bar{B}_{u}Ju  := -\bar{J}\bar{B}_{u}u=bu^{\prime}-aJ^{\prime}u^{\prime}, \nonumber \\
&\bar{B}_{Ju}Ju:=-a u^{\prime} -b J^{\prime} u^{\prime} (=-\bar{B}_{u}u), \nonumber \\
&\bar{B}_{u^{\prime}}u^{\prime}=\bar{B}_{J^{\prime}u^{\prime}}u^{\prime}=\bar{B}_{u^{\prime}}J^{\prime}u^{\prime}
=\bar{B}_{J^{\prime}u^{\prime}}J^{\prime}u^{\prime}:=0, \nonumber \\
&\bar{B}_{u}u^{\prime} =\bar{B}_{u^{\prime}}u :=su + t J u, \nonumber \\
&\bar{B}_{Ju} u^{\prime} =\bar{B}_{u}J^{\prime}u^{\prime}:= -\bar{J}\bar{B}_{u}u^{\prime}
=t u-s J u, \nonumber \\
&\bar{B}_{Ju} J^{\prime}u^{\prime} :=-su - t Ju (= - \bar{B}_{u}u^{\prime}), \nonumber 
\end{align}
where $a$, $b$, $s$ and $t$ are functions on 
$U \times U^{\prime}$. We will derive equations for the conditions 
in Corollary \ref{app_1} to be satisfied.  
The following two lemmas can be proved by straightforward calculations. 

\begin{lemma}\label{e5}
We have the following equations{\rm :}
\begin{align*}
&[\bar{B},\bar{B}]_{u, Ju}u=2(bs-at)u+2(as+bt)Ju, \\
&[\bar{B},\bar{B}]_{u^{\prime}, J^{\prime}u^{\prime}}u^{\prime}=0, \\
&[\bar{B},\bar{B}]_{u, Ju}u^{\prime}=-2(bs-at)u^{\prime}+2(as+bt)J^{\prime}u^{\prime}, \\
&[\bar{B},\bar{B}]_{u^{\prime}, J^{\prime}u^{\prime}}u
=2(s^{2}+t^{2}) Ju,\\
&[\bar{B},\bar{B}]_{u, u^{\prime}} u = (as+bt)u^{\prime}+(bs-at)J^{\prime}u^{\prime}, \\
&[\bar{B},\bar{B}]_{u, u^{\prime}} u^{\prime} =-(s^{2}+t^{2})u,\\
&[\bar{B},\bar{B}]_{Ju, u^{\prime}} u =(sb-ta)u^\prime -(sa+tb)J^\prime u^\prime ,\\
&[\bar{B},\bar{B}]_{Ju, u^{\prime}} u^\prime =-(s^2+t^2)Ju.
\end{align*}
\end{lemma}

Note that the other components of $[\bar{B},\bar{B}]$ are determined by $[\bar{B},\bar{B}]_{\bar J X, \bar JY}=  [\bar{B},\bar{B}]_{X, Y}$ and $[[\bar{B},\bar{B}], \bar J ]=0$, equations which hold in general, since $\bar B$ is symmetric and anti-commutes with $\bar J$. 
We recall the symmetric $J$-Hermitian tensor $\bar{\cal B}$ defined by 
$\bar{\cal B}(X,Y)=\mathrm{Tr} (\bar{B}_{X}\bar{B}_{Y})$ for $X$, $Y \in T \bar{M}$.

\begin{lemma}\label{e6}
$\bar{\cal B}(u,u)=4(as+bt)$, $\bar{\cal B}(u^{\prime},u^{\prime})=2(s^{2}+t^{2})$ and 
$\bar{\cal B}(X,X')=0$, when $X\in TU, X'\in TU'$.
\end{lemma}

\begin{lemma}\label{e65}
We denote the right-hand side of \eqref{c1} by $F$. It holds that 
\begin{align*}
&F(u,Ju,u)= -\frac{1}{2} (bs-at)u +\frac{1}{6} (as+bt)Ju, \\
&F(u^{\prime},J^{\prime}u^{\prime},u^{\prime})=\frac{1}{3}(s^{2}+t^{2})J^{\prime}u^{\prime}, \\
&F(u,Ju,u^{\prime})=\frac{1}{2} (bs-at)u^{\prime} - \frac{1}{6} (as+bt)J^{\prime}u^{\prime}, \\
&F(u^{\prime},J^{\prime}u^{\prime},u)=-\frac{1}{3}(s^{2}+t^{2})Ju, \\
&F(u,u^{\prime},u)=-\frac{1}{12} (as+bt)u^{\prime} -\frac{1}{4} (bs-at)J^{\prime} u^{\prime}, \\ 
&F(u,u^{\prime},u^{\prime})=\frac{1}{6}(s^{2}+t^{2})u. 
\end{align*}
\end{lemma}

\begin{proof}
Using Lemmas \ref{e5} and \ref{e6}, we have 
\begin{align*}
F(u,Ju,u)=&-\frac{1}{4}[\bar{B},\bar{B}]_{u, Ju}u
-\frac{1}{4 \cdot 3}\bar{\cal B}_{\bar{J}}(u,Ju)Ju 
+\frac{1}{8 \cdot 3}(\bar{\cal B}(u,u)Ju-\bar{\cal B}(Ju,u)u) \\
&-\frac{1}{8 \cdot 3}(\bar{\cal B}_{\bar{J}}(u,u)J^{2}u-\bar{\cal B}_{\bar{J}}(Ju,u)Ju) 
\\
=& -\frac{1}{2} \left( (bs-at)u+(bt+as)Ju \right) +\frac{1}{12} \bar{\cal B}(u,u)Ju \\
 &+\frac{1}{24} \bar{\cal B}(u,u)Ju+\frac{1}{24} \bar{\cal B}(u,u)Ju \\
=& -\frac{1}{2} \left( (bs-at)u+(bt+as)Ju \right) +\frac{1}{6} \cdot 4(as+bt) Ju \\
=& -\frac{1}{2} (bs-at)u +\frac{1}{6} (as+bt)Ju. 
\end{align*}
Other equations are obtained in a similar way. 
\end{proof}

\begin{lemma}\label{e7}
The condition \eqref{c1} in Corollary \ref{app_1} holds if and only if 
\begin{align*}
Ric^{D}(u,u)=-\frac{1}{2}(as+bt), \,\,
bs-at=0, \,\,
Ric^{D^{\prime}}(u^{\prime},u^{\prime})=-(s^{2}+t^{2}) 
\end{align*}
on $U \times U^{\prime}$. 
\end{lemma}

\begin{proof}
These follow from Lemmas \ref{e4} and \ref{e65}. 
\end{proof}

For the tensor 
$\bar{a}=\frac{1}{8(2+1)}\bar{\cal B}-\frac{1}{2}P^{\bar{D}}$, 
we have the following lemma.

\begin{lemma}\label{e14}
It holds that  
\begin{align*}
&\bar{a}(u,u)=\frac{1}{6}(as+bt)-\frac{1}{6}Ric^{D}(u,u), \,\, \bar{a}(u,u^{\prime})=0, \\
&\bar{a}(u^{\prime},u^{\prime})=\frac{1}{12}(s^{2}+t^{2}) - \frac{1}{6}Ric^{D^{\prime}}(u^{\prime},u^{\prime}). 
\end{align*}
\end{lemma}

Using Lemmas \ref{e7} and \ref{e14}, we obtain 

\begin{lemma}\label{e145}
In particular, when the condition \eqref{c1} is satisfied, then the tensor $\bar{a}$ is given by 
\begin{align*}
&\bar{a}(u,u)=-\frac{1}{2}Ric^{D}(u,u), \,\, \bar{a}(u,u^{\prime})=0, \,\, 
\bar{a}(u^{\prime},u^{\prime})=- \frac{1}{4}Ric^{D^{\prime}}(u^{\prime},u^{\prime}). 
\end{align*}
This shows that $\bar{a}$ is $\bar{D}$-parallel if \eqref{c1} is fulfilled. 
Thus if \eqref{c1} is satisfied and $\bar{c}=0$, then the condition \eqref{c2} holds.  
\end{lemma}

Next we examine condition \eqref{c4} in the case where $\bar{c}=0$. 
For this purpose, 
we prepare the following lemmas. 
Let $(\theta_{i}^{j})$ (resp. $(\theta^{\prime}{}_{i}^{j})$) be the connection form for $D$ (resp. $D^{\prime}$) 
with respect to the local frame $(u_{1},u_{2})=(u,Ju)$ 
(resp. $(u^{\prime}_{1},u^{\prime}_{2})=(u^{\prime},J^{\prime}u^{\prime})$).

\begin{lemma}\label{e8}
The condition \eqref{c4} in Corollary \ref{app_1} holds for $(u, Ju, u)$
if and only if  
\[ db-2a\theta_{2}^{1} - J^{\ast}da+2b J^{\ast} \theta_{1}^{2} =-2a\gamma -2b J^{\ast} \gamma \]
on $U$. 
\end{lemma}
\begin{proof}
By the definition \eqref{def_B_four} of $\bar{B}$, we have 
\begin{align*}
(\bar{D}_{u} \bar{B})_{Ju}u
&=\bar{D}_{u} (\bar{B}_{Ju}u)-\bar{B}_{\bar{D}_{u}Ju}u-\bar{B}_{Ju}(\bar{D}_{u} u) \\
&=\bar{D}_{u}(b u^{\prime}-aJ^{\prime}u^{\prime})-\bar{B}_{\theta_{2}^{1}(u)u}u
 -\bar{B}_{Ju}(\theta_{1}^{2}(u)Ju)\\
&=(ub)u^{\prime}-(ua) J^{\prime}u^{\prime}-\theta_{2}^{1}(u)(au^{\prime}+bJ^{\prime}u^{\prime})
  -(\theta_{1}^{2}(u)(-au^{\prime}-bJ^{\prime}u^{\prime})\\
&=(ub)u^{\prime}-(ua) J^{\prime}u^{\prime}-2\theta_{2}^{1}(u)(au^{\prime}+b J^{\prime}u^{\prime})\\
&=((ub)-2\theta_{2}^{1}(u)a)u^{\prime}+(-(ua) -2\theta_{2}^{1}(u)b)J^{\prime}u^{\prime}
\end{align*}
and 
\begin{align*}
(\bar{D}_{Ju} \bar{B})_{u}u
&=\bar{D}_{Ju} (\bar{B}_{u}u)-2\bar{B}_{\bar{D}_{Ju}u}u\\
&=\bar{D}_{Ju} (au^{\prime}+bJ^{\prime}u^{\prime})-2\bar{B}_{\theta_{1}^{2}(Ju) Ju}u \\
&=((Ju)a)u^{\prime}+((Ju)b)J^{\prime}u^{\prime}-2\theta_{1}^{2}(Ju) (b u^{\prime}-aJ^{\prime}u^{\prime}) \\
&=((Ju)a-2\theta_{1}^{2}(Ju) b)u^{\prime}+((Ju)b+2\theta_{1}^{2}(Ju)a)J^{\prime}u^{\prime}. 
\end{align*}
We easily  obtain 
\begin{align*}
(-2\gamma \wedge \bar{J} \bar{B})(u,Ju,u)
&=-2\gamma(u)\bar{J} \bar{B}_{Ju}u+2\gamma(Ju)\bar{J} \bar{B}_{u}u \\
&=-2\gamma(u)\bar{J} (b u^{\prime}-aJ^{\prime}u^{\prime})
+2\gamma(Ju)\bar{J} (au^{\prime}+bJ^{\prime}u^{\prime})\\
&=-2\gamma(u)(b J^{\prime} u^{\prime}+au^{\prime})
+2\gamma(Ju)(a J^{\prime} u^{\prime}-bu^{\prime})\\
&=(-2\gamma(u)a-2\gamma(Ju)b)u^{\prime}+(-2\gamma(u)b+2\gamma(Ju)a)J^{\prime} u^{\prime}. 
\end{align*}
Therefore the equation $(d^{\bar{D}} \bar{B})(u,Ju,u)=(-2\gamma \wedge \bar{J} \bar{B})(u,Ju,u)$ holds
iff 
\begin{align*}
&(ub)-2\theta_{2}^{1}(u)a - (Ju)a+2\theta_{1}^{2}(Ju) b=-2\gamma(u)a-2\gamma(Ju)b 
\end{align*}
and
\begin{align*}
&-(ua) -2\theta_{2}^{1}(u)b-(Ju)b-2\theta_{1}^{2}(Ju)a = -2\gamma(u)b+2\gamma(Ju)a. 
\end{align*}
The two equations above are equivalent to each other. 
Hence $(d^{\bar{D}} \bar{B})(u,Ju,u)=(-2\gamma \wedge \bar{J} \bar{B})(u,Ju,u)$ holds iff 
\begin{align*}
&db-2a\theta_{2}^{1} -J^{\ast} da +2b J^{\ast} \theta_{1}^{2} =-2a\gamma -2b J^{\ast} \gamma  
\end{align*}
on $U$. 
\end{proof}

The remaining components of equation \eqref{c4} can be analyzed through straightforward and repetitive computations, analogous to those in Lemma \ref{e8}. Nevertheless, for the reader's convenience, we provide detailed calculations for each case. Lemma \ref{e125} states a necessary and sufficient condition for equation \eqref{c4}.

\begin{lemma}\label{e9}
The condition \eqref{c4} is satisfied for $(u, Ju, u^{\prime})$
if and only if 
\[ dt-2s \theta_{2}^{1}-J^{\ast} ds +2t J^{\ast} \theta_{1}^{2} =-2s \gamma-2t J^{\ast} \gamma \]
on $U$. 
\end{lemma}
\begin{proof}
We have 
\begin{align*}
(\bar{D}_{u} \bar{B})_{Ju}u^{\prime}
&=\bar{D}_{u} (\bar{B}_{Ju}u^{\prime})-\bar{B}_{\bar{D}_{u}Ju}u^{\prime}\\
&=(ut)u+t(D_{u}u)-(us)Ju-s(D_{u}Ju)-\theta_{2}^{1}(u)(su+tJu) \\
&=(ut)u+t\theta_{1}^{2}(Ju)-(us)Ju-s\theta_{2}^{1}(u)u-s\theta_{2}^{1}(u)u-t \theta_{2}^{1}(u)Ju \\
&=((ut)-2s \theta_{2}^{1}(u))u+(-(us)-2t \theta_{2}^{1}(u))Ju, \\
(\bar{D}_{Ju} \bar{B})_{u}u^{\prime}
&=\bar{D}_{Ju} (\bar{B}_{u}u^{\prime})-\bar{B}_{\bar{D}_{Ju} u}u^{\prime} \\
&=\bar{D}_{Ju} (su+tJu)-\bar{B}_{\theta_{1}^{2}(Ju) Ju}u^{\prime} \\
&=((Ju)s)u+s\theta_{1}^{2}(Ju)Ju+((Ju)t)Ju+t\theta_{2}^{1}(Ju)u
-t \theta_{1}^{2}(Ju) u+s\theta_{1}^{2}(Ju)Ju\\
&=((Ju)s-2t \theta_{1}^{2}(Ju))u+((Ju)t+2 s \theta_{1}^{2}(Ju))Ju
\end{align*}
and 
\begin{align*}
(-2\gamma \wedge \bar{J} \bar{B})(u,Ju,u^{\prime})
&=-2\gamma(u)\bar{J}\bar{B}_{Ju}u^{\prime}+2\gamma(Ju)\bar{J} \bar{B}_{u}u^{\prime}\\
&=-2\gamma(u)\bar{J}(tu-sJu)+2\gamma(Ju)\bar{J} (su+tJu) \\
&=(-2\gamma(u)s-2\gamma(Ju)t)u+(-2\gamma(u)t+2\gamma(Ju)s)Ju. 
\end{align*}
Therefore the equation 
$(d^{\bar{D}} \bar{B})(u,Ju,u^{\prime})=(-2\gamma \wedge \bar{J} \bar{B})(u,Ju,u^{\prime})$ holds
iff 
\begin{align*}
(ut)-2s \theta_{2}^{1}(u) - (Ju)s + 2t \theta_{1}^{2}(Ju)=-2\gamma(u)s-2\gamma(Ju)t
\end{align*}
and
\begin{align*}
-(us)-2t \theta_{2}^{1}(u)-(Ju)t-2s\theta_{1}^{2}(Ju)=-2\gamma(u)t+2\gamma(Ju)s. 
\end{align*}
The above two equations are equivalent to each other. 
\end{proof}

\begin{lemma}\label{e10}
Both sides in \eqref{c4} vanish for $(u^{\prime}, J^{\prime}u^{\prime}, u^{\prime})$. 
\end{lemma}

\begin{lemma}\label{e11}
The condition \eqref{c4} holds for $(u^{\prime}, J^{\prime}u^{\prime}, u)$
if and only if 
\begin{align*}
dt-s \theta^{\prime}{}_{2}^{1}
-J^{\prime \ast } ds +t J^{\prime \ast } \theta^{\prime}{}_{1}^{2} 
=-2s \gamma-2t J^{\prime \ast } \gamma. 
\end{align*}
on $U^{\prime}$
\end{lemma}

\begin{proof}
We have 
\begin{align*}
(\bar{D}_{u^{\prime}} \bar{B})_{J^{\prime}u^{\prime}}u
&=\bar{D}_{u^{\prime}} (\bar{B}_{J^{\prime}u^{\prime}}u)-\bar{B}_{\bar{D}_{u^{\prime}}J^{\prime}u^{\prime}}u \\
&=\bar{D}_{u^{\prime}} (tu-sJu)-\bar{B}_{\theta^{\prime}{}_{2}^{1}(u^{\prime})u^{\prime}}u \\
&=(u^{\prime}t-\theta^{\prime}{}_{2}^{1}(u^{\prime})s)u
+(-(u^{\prime}s)-\theta^{\prime}{}_{2}^{1}(u^{\prime})t)Ju, \\
(\bar{D}_{J^{\prime} u^{\prime}} \bar{B})_{u^{\prime}}u
&=\bar{D}_{J^{\prime} u^{\prime}} (\bar{B}_{u^{\prime}}u)-\bar{B}_{\bar{D}_{J^{\prime} u^{\prime}}u^{\prime}}u \\
&=\bar{D}_{J^{\prime} u^{\prime}} (su+tJu)
-\bar{B}_{\theta^{\prime}{}_{1}^{2}(J^{\prime} u^{\prime})J^{\prime}  u^{\prime}}u\\
&=((J^{\prime} u^{\prime})s-\theta^{\prime}{}_{1}^{2}(J^{\prime} u^{\prime})t)u
+((J^{\prime} u^{\prime})t+\theta^{\prime}{}_{1}^{2}(J^{\prime} u^{\prime})s)Ju
\end{align*}
and 
\begin{align*}
(-2\gamma \wedge \bar{J} \bar{B})(u^{\prime},J^{\prime}u^{\prime},u)
&=-2\gamma(u^{\prime})\bar{J}\bar{B}_{J^{\prime}u^{\prime}}u
+2\gamma(J^{\prime}u^{\prime})\bar{J} \bar{B}_{u^{\prime}}u\\
&=-2\gamma(u^{\prime})\bar{J}(tu-sJu)+2\gamma(J^{\prime}u^{\prime})\bar{J} (su+tJu) \\
&=(-2\gamma(u^{\prime})s-2\gamma(J^{\prime}u^{\prime})t)u
+(-2\gamma(u^{\prime})t+2\gamma(J^{\prime}u^{\prime})s)Ju. 
\end{align*}
Therefore the equation 
$(d^{\bar{D}} \bar{B})(u^{\prime},J^{\prime}u^{\prime},u)
=(-2\gamma \wedge \bar{J} \bar{B})(u^{\prime},J^{\prime}u^{\prime},u)$ holds
iff 
\begin{align*}
u^{\prime}t-\theta^{\prime}{}_{2}^{1}(u^{\prime})s
-(J^{\prime} u^{\prime})s+\theta^{\prime}{}_{1}^{2}(J^{\prime} u^{\prime})t
=-2\gamma(u^{\prime})s-2\gamma(J^{\prime}u^{\prime})t
\end{align*}
and
\begin{align*}
-(u^{\prime}s)-\theta^{\prime}{}_{2}^{1}(u^{\prime})t
-(J^{\prime} u^{\prime})t-\theta^{\prime}{}_{1}^{2}(J^{\prime} u^{\prime})s
=-2\gamma(u^{\prime})t+2\gamma(J^{\prime}u^{\prime})s
\end{align*}
The two equations above are equivalent to each other.
\end{proof}

\begin{lemma}\label{e12}
The condition \eqref{c4} holds for $(u, u^{\prime}, u)$
if and only if 
\begin{align*}
&ds+2t \theta_{2}^{1}=2t \gamma, \,\, dt+2s \theta_{1}^{2}=-2s\gamma 
\end{align*}
on $U$ and 
\begin{align*}
&-da-b \theta^{\prime}{}_{2}^{1}=-2b\gamma, \,\, 
-db-a \theta^{\prime}{}_{1}^{2}=2 a \gamma 
\end{align*}
on $U^{\prime}$.
\end{lemma}

\begin{proof}
We have 
\begin{align*}
(\bar{D}_{u} \bar{B})_{u^{\prime}}u
&=\bar{D}_{u} (\bar{B}_{u^{\prime}}u)-\bar{B}_{u^{\prime}}(\bar{D}_{u} u) \\
&=\bar{D}_{u} (su+tJu)-\bar{B}_{u^{\prime}} \theta_{1}^{2}(u)Ju \\
&=(us)u+s \theta_{1}^{2}(u) Ju+(ut)Ju+t \theta_{2}^{1}(u)u- t \theta_{1}^{2}(u)u+s\theta_{1}^{2}(u)Ju) \\
&=((us)+2t \theta_{2}^{1}(u))u+((ut)+2s\theta_{1}^{2}(u))Ju, \\ 
(\bar{D}_{u^{\prime}} \bar{B})_{u}u
&=\bar{D}_{u^{\prime}} (\bar{B}_{u}u) \\
&=\bar{D}_{u^{\prime}} (a u^{\prime}+bJ^{\prime}u^{\prime}) \\
&=((u^{\prime}a)+b \theta^{\prime}{}_{2}^{1}(u^{\prime}))u^{\prime}
+(u^{\prime}b+a \theta^{\prime}{}_{1}^{2}(u^{\prime}))J^{\prime}u^{\prime}  
\end{align*}
and
\begin{align*}
(-2\gamma \wedge \bar{J} \bar{B})(u,u^{\prime},u)
&=-2\gamma(u)\bar{J}\bar{B}_{u^{\prime}}u
+2\gamma(u^{\prime})\bar{J} \bar{B}_{u}u\\
&=-2\gamma(u)\bar{J}(su+tJu)+2\gamma(u^{\prime})\bar{J} (au^{\prime}+bJ^{\prime}u^{\prime}) \\
&=2t \gamma(u)u-2s\gamma(u)Ju
-2b\gamma(u^{\prime})u^{\prime}+2 a \gamma(u^{\prime}) J^{\prime}u^{\prime}. 
\end{align*}
Therefore the equation 
$(d^{\bar{D}} \bar{B})(u,u^{\prime},u)
=(-2\gamma \wedge \bar{J} \bar{B})(u,u^{\prime},u)$ holds
iff 
\begin{align*}
&us+2t \theta_{2}^{1}(u)=2t \gamma(u), \,\, ut+2s\theta_{1}^{2}(u)=-2s\gamma(u), \\
&-u^{\prime}a-b \theta^{\prime}{}_{2}^{1}(u^{\prime})=-2b\gamma(u^{\prime}), \,\, 
-u^{\prime}b-a \theta^{\prime}{}_{1}^{2}(u^{\prime})=2 a \gamma(u^{\prime}). \qedhere
\end{align*}
\end{proof}

\begin{lemma}\label{e13}
The condition \eqref{c4} holds for $(u, u^{\prime}, u^{\prime})$
if and only if 
\begin{align*}
&-ds+t \theta^{\prime}{}_{1}^{2}=-2t \gamma, \,\, 
-dt-s \theta^{\prime}{}_{1}^{2}=2 s \gamma 
\end{align*}
on $U^{\prime}$. 
\end{lemma}

\begin{proof}
We have 
\begin{align*}
(\bar{D}_{u} \bar{B})_{u^{\prime}}u^{\prime}
&=\bar{D}_{u} (\bar{B}_{u^{\prime}}u^{\prime})=0, \\
(\bar{D}_{u^{\prime}} \bar{B})_{u}u^{\prime}
&=\bar{D}_{u^{\prime}} \bar{B}_{u}u^{\prime}- \bar{B}_{u}(\bar{D}_{u^{\prime}}u^{\prime}) \\
&=\bar{D}_{u^{\prime}} (su+tJu)- \bar{B}_{u}(\theta^{\prime}{}_{1}^{2}(u^{\prime})J^{\prime}u^{\prime}) \\
&=(u^{\prime}s-\theta^{\prime}{}_{1}^{2}(u^{\prime})t)u
+(u^{\prime}t+\theta^{\prime}{}_{1}^{2}(u^{\prime})s)Ju
\end{align*}
and 
\begin{align*}
(-2\gamma \wedge \bar{J} \bar{B})(u, u^{\prime}, u^{\prime})
&=-2\gamma(u)\bar{J}\bar{B}_{u^{\prime}}u^{\prime}
+2\gamma(u^{\prime})\bar{J} \bar{B}_{u}u^{\prime} \\
&=-2t \gamma(u^{\prime})u+2 s \gamma(u^{\prime})Ju. 
\end{align*}
Therefore the equation 
$(d^{\bar{D}} \bar{B})(u, u^{\prime}, u^{\prime})
=(-2\gamma \wedge \bar{J} \bar{B})(u, u^{\prime}, u^{\prime})$ holds
iff 
\begin{align*}
&-u^{\prime}s+\theta^{\prime}{}_{1}^{2}(u^{\prime})t=-2t \gamma(u^{\prime}), \,\,
-u^{\prime}t-\theta^{\prime}{}_{1}^{2}(u^{\prime})s=2 s \gamma(u^{\prime}). \qedhere
\end{align*}
\end{proof}

By gathering Lemmas \ref{e8}-\ref{e13}, we summarize the results regarding condition \eqref{c4}. 

\begin{lemma}\label{e125}
When $\bar{c}=0$, 
the condition \eqref{c4} is satisfied  
if and only if 
\begin{align*}
&db-2a\theta_{2}^{1} - J^{\ast} da +2b J^{\ast} \theta_{1}^{2} =-2a\gamma -2b  J^{\ast} \gamma, \,\, 
ds+2t \theta_{2}^{1}=2t \gamma, \,\, dt+2s \theta_{1}^{2}=-2s\gamma 
\end{align*}
on $U$ and
\begin{align*}
&da+b \theta^{\prime}{}_{2}^{1}=2b\gamma, \,\, db+a \theta^{\prime}{}_{1}^{2}=-2 a \gamma, \,\, 
ds-t \theta^{\prime}{}_{1}^{2}=2t \gamma, \,\, dt+s \theta^{\prime}{}_{1}^{2}=-2 s \gamma 
\end{align*}
on $U^{\prime}$.
\end{lemma}

The following lemma addresses condition \eqref{c5}.

\begin{lemma}\label{e155}
Assume that the condition \eqref{c1} is satisfied. 
When $\bar{c}=0$, 
the condition \eqref{c5} holds if and only if 
\begin{align*}
&2s Ric^{D}(u,u)-a Ric^{D^{\prime}}(u^{\prime},u^{\prime})=0, \,\, 2tRic^{D}(u,u)-b 
Ric^{D^{\prime}}(u^{\prime},u^{\prime})=0.
\end{align*}
\end{lemma}
\begin{proof}
Set $L(X,Y,Z):=\bar{a}(X,\bar{B}_{Y}Z)-\bar{a}(Y,\bar{B}_{X}Z)$ for $X$, $Y$, $Z \in T\bar{M}$. 
We will show that  
\begin{align*}
L(u,u^{\prime},u) =s\bar{a}(u,u)-a\bar{a}(u^{\prime},u^{\prime}), \,\, 
L(u,u^{\prime},Ju) =t\bar{a}(u,u)-b\bar{a}(u^{\prime},u^{\prime})  
\end{align*}
for $u \in TU$ and $u^{\prime} \in TU^{\prime}$ and $L=0$ for other cases. 
Since $\bar{B}$ is symmetric and anti-commutes with $\bar{J}$, we obtain 
$L(u,Ju, \, \cdot \,)=0$ and $L(u^{\prime},J^{\prime}u^{\prime}, \, \cdot \,)=0$. 
Moreover it holds that 
\begin{align*}
L(u,u^{\prime},u)
=\bar{a}(u,\bar{B}_{u^{\prime}}u)-\bar{a}(u^{\prime},\bar{B}_{u}u)
=s\bar{a}(u,u)-a\bar{a}(u^{\prime},u^{\prime})
\end{align*}
and 
\begin{align*}
L(u,u^{\prime},Ju)
=\bar{a}(u,\bar{B}_{u^{\prime}}Ju)-\bar{a}(u^{\prime},\bar{B}_{u}Ju)
=t\bar{a}(u,u)-b\bar{a}(u^{\prime},u^{\prime})
\end{align*}
from $\bar{a}(u,u^{\prime})=0$. 
Next we can also check that  
\begin{align*}
&L(u,u^{\prime},u^{\prime})
=\bar{a}(u,\bar{B}_{u^{\prime}}u^{\prime})-\bar{a}(u^{\prime},\bar{B}_{u}u^{\prime})
=-\bar{a}(u^{\prime},su+tJu)=0, \\
&L(u,u^{\prime},J^{\prime}u^{\prime})
=\bar{a}(u,\bar{B}_{u^{\prime}}J^{\prime}u^{\prime})-\bar{a}(u^{\prime},\bar{B}_{u}J^{\prime}u^{\prime})
=-\bar{a}(u^{\prime},tu-sJu)=0. 
\end{align*}
Then the condition \eqref{c5} holds for $\bar{c}=0$ if and only if 
\begin{align*}
&s\bar{a}(u,u)-a\bar{a}(u^{\prime},u^{\prime})=0, \,\, t\bar{a}(u,u)-b\bar{a}(u^{\prime},u^{\prime})=0.
\end{align*}
Using Lemma \ref{e145}, we have the conclusion. 
\end{proof}

To locally realize $\bar{M}=N \times N^{\prime}$ as a projective special complex manifold, 
we consider two cases based on Ricci curvatures of the factors, 
and choose $a$, $b$, $s$, $t$ and $\gamma$ so that \eqref{c1}-\eqref{c5} in Corollary \ref{app_1} 
are satisfied in each case. \\

\noindent
{\bf (1) Flat cases.} $Ric^{D}=0$ and $Ric^{D^{\prime}}=0$:  
At first we choose $s=t=0$ and $a=b=0$. It is easy to see that the condition \eqref{c1} holds from 
Lemma \ref{e7}. 
We take the one-form 
$\gamma=0$. Then it can be verified that Lemma \ref{e125} holds for the condition (\ref{c4}).
Since $\bar{a}=0$ by Lemma \ref{e145}, it is easy to see that $d \gamma=2 \bar{a}_{\bar{J}}$ for \eqref{c3}. 
By Lemma \ref{e155}, the condition \eqref{c5} is satisfied for $\bar{c}=0$. 
The product $\mathbb{C} \times \mathbb{C}$ with $\bar{B} =0$, $\gamma=0$, $\bar{c}=0$ 
satisfies all the conditions in Corollary \ref{app_1}, which is a special case of Example \ref{asc_ex}. 
Next we take functions $s=t=0$ and $a$, $b$ such that $a^{2}+b^{2}$ is a non-zero constant. 
We see that the condition \eqref{c1} holds from 
Lemma \ref{e7}, and hence \eqref{c2} holds by Lemma \ref{e145}. 
We choose the one-form  
\begin{align}\label{flat_gamm}
\gamma=-\theta_{1}^{2}-\frac{1}{2} \theta^{\prime}{}_{1}^{2} 
-\frac{1}{2} \frac{a db -b da}{a^{2}+b^{2}}. 
\end{align}
Note that 
\begin{align}\label{e40}
\dfrac{a db -b da}{a^{2}+b^{2}}= 
\left\{   \begin{array}{l}
-\dfrac{da}{b} \,\,\, \mbox{on} \,\,\, (U \times U^{\prime}) \backslash b^{-1}(0) \vspace{2mm}\\
\dfrac{db}{a} \,\,\, \mbox{on} \,\,\, (U \times U^{\prime}) \backslash a^{-1}(0)
\end{array}
\right.
\end{align}
due to the constancy of $a^{2}+b^{2}$. 
Then it can be verified that the condition \eqref{c4} holds. 
Indeed, the condition \eqref{c4} reduces to 
\begin{align}\label{e30}
&db-2a\theta_{2}^{1} -  J^{\ast} da +2b  J^{\ast} \theta_{1}^{2} =-2a\gamma -2b  J^{\ast} \gamma, 
\end{align}
on $U$ and
\begin{align}\label{e31}
&da+b \theta^{\prime}{}_{2}^{1}=2b\gamma, \,\, db+a \theta^{\prime}{}_{1}^{2}=-2 a \gamma
\end{align}
on $U^{\prime}$ by $s=t=0$. It is easy to see that \eqref{flat_gamm} satisfies \eqref{e30}. 
If $b(x) \neq 0$ and $a(x) \neq 0$ at a point $x \in U \times U^{\prime}$, it holds that \eqref{e31} at $x$ by \eqref{e40}. 
If $b(x)=0$ (resp. $a(x)=0$) for $x \in U \times U^{\prime}$, then $a(x) \neq 0$ (resp. $b(x)\neq 0$) and $(da)_{x}=0$ 
(resp. $(db)_{x}=0$). Then it holds \eqref{e31}. 
Therefore $\gamma$ in \eqref{flat_gamm} satisfies \eqref{c4}. 
Since 
\[ d \gamma=-d \theta_{1}^{2}-\frac{1}{2} d \theta^{\prime}{}_{1}^{2}
=Ric^{D}(u,u) \Omega+\frac{1}{2} Ric^{D^{\prime}} (u^{\prime},u^{\prime}) \Omega^{\prime}=0 \]
by Lemma \ref{str._eq} and $\bar{a}=0$, we see that $d \gamma=2 \bar{a}_{\bar{J}}(=0)$ for \eqref{c3}.
From Lemma \ref{e155}, it follows that the condition \eqref{c5} is satisfied for $\bar{c}=0$. 
Hence we obtain 

\begin{proposition}\label{4_dim_ex_0}
The manifold $(\mathbb{C}^{2}, \bar{J}, [\bar{D}])$ with 
the standard complex structure $\bar{J}$ and the c-projective structure 
induced from standard flat connection $\bar{D}$ can be realized as a 
projective special complex manifold such that $\bar{B}\neq 0$, $\bar{\cal B}=0$ and $\bar{c}=0$ 
by considering an $S^{1}$-bundle with a flat connection $\eta$ given by \eqref{flat_gamm} via \eqref{eta_gamma:eq}. 
The resulting conical special complex manifold is non-trivial and of horizontal type 
with respect to $\eta$, see Definition~\ref{hortype:def}. 
\end{proposition}

In particular, there exists a $4$-dimensional c-projectively flat manifold which can be  
realized as a projective special complex manifold with $\bar{B} \neq 0$ and $\bar{\cal B}=0$. 
For a projective special K{\"a}hler manifold,  if $\bar{\cal B}=0$, then $\bar{B} =0$. 
The reason is that the square of an endomorphism which is symmetric
with respect to a definite metric is trace-free only if the endomorphism vanishes.
Therefore, the manifold in Proposition~\ref{4_dim_ex_0} 
does not admit a compatible projective special K{\"a}hler metric. 
\\

\noindent
{\bf (2) Non flat cases.} $Ric^{D^{\prime}} < 0$: 
We take functions $s$, $t$ such that $Ric^{D^{\prime}}(u^{\prime},u^{\prime}) =-(s^{2}+t^{2})$ 
in which $s^{2}+t^{2}$ is non-zero constant. 
We choose a constant $k$ as  
\[  k=\frac{2 Ric^{D}(u,u)}{Ric^{D^{\prime}}(u^{\prime},u^{\prime})} \]
and set $a:=ks$ and $b:=kt$.    
Then we have 
\[ Ric^{D}(u,u) =\frac{1}{2} k Ric^{D^{\prime}}(u^{\prime},u^{\prime}) =-\frac{1}{2}k(s^{2}+t^{2})=-\frac{1}{2}(as+bt)  \]
and $bs-at=0$. 
Then \eqref{c1} and \eqref{c2} are fulfilled. 
By choosing the one-form 
\begin{align}\label{eq423_gamma}
\gamma=-\theta_{1}^{2}-\frac{1}{2} \theta^{\prime}{}_{1}^{2} 
-\frac{1}{2} \frac{s dt -t ds}{s^{2}+t^{2}}
\left(= -\theta_{1}^{2}-\frac{1}{2} \theta^{\prime}{}_{1}^{2} 
-\frac{1}{2} \frac{a db -b da}{a^{2}+b^{2}} \,\,\,
\mbox{when} \,\,\, k\neq 0
\right), 
\end{align}
we can verify that \eqref{c4} holds due to Lemma \ref{e125}, by applying the similar argument 
to the flat case. 
Since 
\begin{align}\label{eq_422_gamma}
d \gamma=-d \theta_{1}^{2}-\frac{1}{2} d \theta^{\prime}{}_{1}^{2}
=Ric^{D}(u,u) \Omega+\frac{1}{2} Ric^{D^{\prime}}(u^{\prime},u^{\prime}) \Omega^{\prime}, 
\end{align}
we easily see that $d \gamma=2 \bar{a}_{\bar{J}}$ for \eqref{c3} from Lemma \ref{e145}. 
From Lemma \ref{e155}, it follows that the condition (6.4) is satisfied for $\bar{c}=0$. 
Therefore the choice of $k$ as above means that \eqref{c5} is fulfilled due to Lemma \ref{e155}. 

We denote the standard complex structure and standard connection on $\mathbb{C}P^{1}$ endowed with 
the Fubini-Study metric
by $J^{+}$ and $D^{+}$, respectively. Let $\mathbb{C}H^{1}$ represent the hyperbolic plane with negative constant Gaussian curvature and $J^{-}$, $D^{-}$ be the standard complex structure and 
standard connection on it. Here a specific value of a constant Gaussian curvature is not crucial
in geometry of projective special complex manifolds 
since homothetic changes of metrics $\langle \, \cdot \, ,  \, \cdot \,\rangle 
\mapsto c \langle \, \cdot \, ,  \, \cdot \,\rangle$ $(c>0)$ 
preserve $J^{\pm}$, $D^{\pm}$. 
Furthermore we denote the standard complex structure and standard connection on $\mathbb{C}$ by $J^{0}$ and $D^{0}$, respectively. We conclude that

\begin{proposition}\label{4_dim_ex_1}
The products of surfaces 
\begin{align*}
(\bar{M},\bar{J},[\bar{D}])= 
\left\{   
\begin{array}{l}
(\mathbb{C} \times \mathbb{C}H^{1}, J^{0}+J^{-},[D^{0}+D^{-}]) \quad (\mbox{if} \,\, k=0), \\ 
(\mathbb{C}P^{1} \times \mathbb{C}H^{1},J^{+}+J^{-},[D^{+}+D^{-}]) \quad (\mbox{if} \,\, k<0), \\ 
(\mathbb{C}H^{1} \times \mathbb{C}H^{1},J^{-}+J^{-},[D^{-}+D^{-}]) \quad (\mbox{if} \,\, k>0)
\end{array}
\right.
\end{align*}
can be locally realized as projective special complex manifolds 
with $\bar{\cal B} \neq 0$ and $\bar{c}=0$ by considering $S^{1}$-bundles 
with the (non-flat) connections $\eta$ given by \eqref{eq423_gamma} via \eqref{eta_gamma:eq}. 
The curvatures of these connections 
are involving the Ricci curvatures of respective factors, as expressed in \eqref{eq_422_gamma}. 
Each resulting conical special complex manifold 
is non-trivial and of horizontal type with respect to the connection $\eta$. 
\end{proposition}

\begin{remark}
{\rm 
The Cotton-York tensor 
$C^{\bar{D}}=d^{\bar{D}}P^{\bar{D}}$ of $\bar{D}$ in Propositions \ref{4_dim_ex_0} and \ref{4_dim_ex_1} 
vanishes. 
}
\end{remark}

When a projective special complex manifold admits a compatible projective special 
K\"ahler metric $\bar g$, the fundamental $(0,2)$-tensor $\bar{a}$ coincides with $\bar{g}$ 
as we discussed in Section \ref{SpKah}, in particular, $\bar{a}$ is positive definite. 
Consequently, the local realizations for $\mathbb{C} \times \mathbb{C}H^{1}$ 
and $\mathbb{C}P^{1} \times \mathbb{C}H^{1}$ given in Proposition \ref{4_dim_ex_1}  
do not occur in the projective special   
K\"ahler setting, as implied by \eqref{eq_422_gamma}. 
In other words, only the third case $\mathbb{C}H^{1} \times \mathbb{C}H^{1}$ in Proposition \ref{4_dim_ex_1} is projective special   
K\"ahler.  
Lemma \ref{e145} implies that the Gaussian curvatures of the factors 
are $-2$ and $-4$, which is consistent with Section 7 in \cite{MS1} and with one of the two $4$-dimensional simply connected projective special K{\"a}hler Lie groups classified in \cite{M}. The same homogeneous projective special K\"ahler manifolds are obtained by specializing the classification \cite{AC} to  complex dimension~$2$: $\mathbb{C}H^2$ and $(\mathrm{SL}(2,\mathbb{R})/\mathrm{SO}(2)) \times (\mathrm{SO}_0(2,1)/\mathrm{SO}(2))\cong \mathbb{C}H^{1} \times \mathbb{C}H^{1}$ associated with $A_4$ and $B_3$ under the correspondence explained in \cite{AC}. The realization as a Lie group follows from the 
Iwasawa decomposition.

\subsection{$\dim \bar{M}=2$}
Here we give three examples obtained as an application of Theorem \ref{char:thm_two_dim}. 
Let $(\bar{M},\bar{J},\bar{D})$ be a surface with a complex structure $\bar{J}$ and 
a torsion-free complex connection $\bar{D}$. We assume that the Ricci tensor $Ric^{\bar{D}}$ is symmetric 
and  $\bar{D} Ric^{\bar{D}}=0$. 
There exists locally parallel volume forms $\Omega$ on $U \subset \bar{M}$. 
Take a local frame $(u,Ju)$ such that $\Omega(u,\bar{J}u)=1$ on $U$ and define a symmetric tensor 
$\bar{B}$ on $U$ by
\begin{align}\label{def_s_t}
&\bar{B}_{u}u := s u+ t \bar{J}u, \\
&\bar{B}_{u}\bar{J}u = \bar{B}_{\bar{J}u}u := -s\bar{J}u+tu\, (=-\bar{J}\bar{B}_{u}u),  \nonumber \\
&\bar{B}_{\bar{J}u} \bar{J}u :=-su-t \bar{J}u\, (=-\bar{B}_{u}u) \nonumber
\end{align}
for some functions $s$, $t$ on $U$ to be determined in the proof of Proposition \ref{2D:prop}. 
Note that $\bar{B}_X$ anti-commutes with $\bar{J}$ for all $X$. 
It holds that  
\begin{align*}
[\bar{B},\bar{B}]_{u,\bar{J}u}u=\bar{B}_{u}\bar{B}_{\bar{J}u}u-\bar{B}_{\bar{J}u}\bar{B}_{u}u=2(s^{2}+t^{2})\bar{J}u.
\end{align*}

Let $(\theta_{i}^{j})$ be the connection form for $\bar{D}$ 
with respect to a local frame $(u_{1},u_{2})=(u,\bar{J}u)$.

\begin{lemma}\label{der_BB}
It holds that 
\begin{align*}
(\bar{D}_{u} \bar{B})_{\bar{J}u}u
=&\{ dt(u)  + 3s \theta^{2}_{1}(u) \}u + \{  - ds(u) + 3 t \theta^{2}_{1}(u) \} \bar{J}u
\end{align*}
and 
\begin{align*}
(\bar{D}_{\bar{J}u} \bar{B})_{u}u
=&\{ ds(\bar{J}u) - 3t\theta^{2}_{1} (\bar{J}u) \} u 
  +\{ dt(\bar{J}u)+ 3s\theta^{2}_{1} (\bar{J}u) \} \bar{J}u.
\end{align*}
\end{lemma}

\begin{proof}
We have 
\begin{align*}
(\bar{D}_{u} \bar{B})_{\bar{J}u}u
=&\bar{D}_{u} (\bar{B}_{\bar{J}u}u)-\bar{B}_{\bar{D}_{u}\bar{J}u}u-\bar{B}_{\bar{J}u}(\bar{D}_{u}u) \\
=&\bar{D}_{u}(-s\bar{J}u+tu)-\bar{B}_{\theta^{1}_{2}(u) u}u-\bar{B}_{\bar{J}u} (\theta^{2}_{1}(u) \bar{J}u) \\
=&-(us)\bar{J}u - s \theta^{1}_{2}(u) u+ (ut)u+t \theta^{2}_{1}(u) \bar{J}u\\
&-\theta^{1}_{2}(u) (su+t\bar{J}u)-\theta^{2}_{1}(u) (-su-t\bar{J}u)\\
=&\{ dt(u)  + 3s \theta^{2}_{1}(u) \}u + \{  - ds(u) + 3 t \theta^{2}_{1}(u) \} \bar{J}u
\end{align*}
and 
\begin{align*}
(\bar{D}_{\bar{J}u} \bar{B})_{u}u
=&\bar{D}_{\bar{J}u} (\bar{B}_{u}u)-2 \bar{B}_{\bar{D}_{\bar{J} u} u} u \\
=&\bar{D}_{\bar{Ju}} (su+t \bar{J}u)- 2 \bar{B}_{\theta^{2}_{1} (\bar{J}u) \bar{J}u} u \\
=&((\bar{J}u)s)u+s\theta^{2}_{1} (\bar{J}u) \bar{J}u + ((\bar{J}u) t) \bar{J}u +t \theta^{1}_{2}(\bar{J}u) u
-2\theta^{2}_{1} (\bar{J}u) (-s\bar{J}u+tu) \\
=&\{ ds(\bar{J}u) - 3t\theta^{2}_{1} (\bar{J}u) \} u 
  +\{ dt(\bar{J}u)+ 3s\theta^{2}_{1} (\bar{J}u) \} \bar{J}u. \qedhere
\end{align*}
\end{proof}

\begin{lemma}\label{b_wedge}
For any one-form $\gamma$ on $U$, we have 
\begin{align*}
(\gamma \wedge \bar{J} \bar{B})_{u,Ju}u
=\{ s \gamma(u) + t \gamma(\bar{J}u) \}u
   +\{  t \gamma(u) -s\gamma(\bar{J}u) \}\bar{J}u. 
\end{align*}
\end{lemma}

\begin{proof}
By a straightforward calculation, we have
\begin{align*}
(\gamma \wedge \bar{J} \bar{B})_{u,Ju}u
&=\gamma(u)\bar{J} \bar{B}_{\bar{J}u}u-\gamma(\bar{J}u)\bar{J} \bar{B}_{u}u \\
&=\gamma(u)\bar{J} (-s\bar{J}u+tu)-\gamma(\bar{J}u)\bar{J} (s u+ t \bar{J}u)  \\
&=\{ s \gamma(u) + t \gamma(\bar{J}u) \}u
   +\{  t \gamma(u) -s\gamma(\bar{J}u) \}\bar{J}u. \qedhere
\end{align*}
\end{proof}

By Lemmas \ref{der_BB} and \ref{b_wedge}, we conclude

\begin{lemma}\label{sol_gam}
If $s^{2}+t^{2}$ is non-zero constant, then 
the one-form 
\begin{align}\label{gamma_2_dim}
\gamma=-\frac{3}{2} \theta^{2}_{1}-\frac{1}{2} \frac{s dt-tds}{s^{2}+t^{2}} 
\end{align}
satisfies the condition (6-3)  in Definition \ref{intrinsic_2}. 
The one-form $\gamma$ given in \eqref{gamma_2_dim} is the unique one-form satisfying (6-3) 
for a given $\bar{B}$. 
\end{lemma}

\begin{proof}
By Lemmas \ref{der_BB} and \ref{b_wedge}, the condition (6-3) is equivalent to 
\begin{align*}
& dt + 3 s \theta^{2}_{1} - \bar{J}^{\ast} ds + 3t \bar{J}^{\ast} \theta^{2}_{1} 
=-2 s \gamma -2 t \bar{J}^{\ast} \gamma, \\
&-ds + 3 t \theta^{2}_{1}-\bar{J}^{\ast} dt -3s \bar{J}^{\ast} \theta^{2}_{1} 
= -2 t \gamma +2 s \bar{J}^{\ast} \gamma.
\end{align*}
These two equations are equivalent to each other, so the condition (6-3) is equivalent to
\begin{align}\label{6_3_sur}
& dt + 3 s \theta^{2}_{1} - \bar{J}^{\ast} ds + 3t \bar{J}^{\ast} \theta^{2}_{1} 
=-2 s \gamma -2 t \bar{J}^{\ast} \gamma. 
\end{align}
If $s^{2}+t^{2} \neq 0$, the one-form 
\begin{align*}
\gamma=-\frac{3}{2} \theta^{2}_{1}-\frac{1}{2} \frac{s dt-tds}{s^{2}+t^{2}} 
\end{align*}
satisfies (\ref{6_3_sur}), where we used that $sds = -tdt$ due to the constancy of $s^{2}+t^{2}$. 
To establish the uniqueness of $\gamma$, it suffices to show that $\phi=0$ if 
$\phi \wedge \bar{J}\bar{B}=0$ for a one-form $\phi$. 
By considering $(\phi \wedge \bar{J}\bar{B})(u,\bar{J}u,u)=0$, we obtain the equation 
\[ 
\left( \begin{array}{cc}
                    s   & t \\
                    -t  &  s\\
                    \end{array} 
             \right)
\left( \begin{array}{c}
                    \phi(u)  \\
                    \phi(\bar{J}u) \\
                    \end{array} 
             \right)
             =
\left( \begin{array}{c}
                    0 \\
                    0 \\
                    \end{array} 
             \right).
\]
Since $s^{2}+t^{2}$ is non-zero constant, we have 
$\phi(u)=\phi(\bar{J}u)=0$, thus, $\phi=0$.  
\end{proof}

\begin{proposition}\label{2D:prop}
Let $(\bar{M},\bar{J},\bar{D})$ be a surface with a complex structure $\bar{J}$ and 
a torsion-free complex connection $\bar{D}$. We assume that the Ricci tensor $Ric^{\bar{D}}$ is symmetric and  $\bar{D}$-parallel. 
Then $\bar{M}$ can be locally realized as a projective special complex manifold 
by considering an $S^{1}$-bundle with the connection $\eta$ whose curvature form is $0$, 
$-\frac{1}{2} Ric^{\bar{D}}_{\bar{J}}$ or $-\frac{3}{2} Ric^{\bar{D}}_{\bar{J}}$.    
The conical special complex manifold associated with the first two $S^1$-bundles is trivial and the third case is non-trivial. All manifolds are of horizontal type with respect to $\eta$. 
\end{proposition}

\begin{proof}
The following argument is essentially given in \cite{MS}. 
Therefore our proof proceeds by  constructing the 
required data and checking that the equations from Definition \ref{intrinsic_2} are satisfied. 
We can define 
a symmetric $(0,2)$-tensor $\bar{a}$ on $U$ by 
\begin{align}\label{def_metric}
\bar{a}(X,Y):=\delta\, \Omega(X,\bar{J}Y) \,\,\,\,\, 
(\delta=0, 1 \,\, \mbox{or} \,\, -1)
\end{align}
for $X,Y \in TU$. 
Note that $\bar{a}(u,u)=\bar{a}(\bar{J}u,\bar{J}u)=\delta$, $\bar{a}(u,\bar{J}u)=0$. 
We can see that $\bar{a}$ and  $\bar{B}$ satisfy the conditions (6-2) and (6-4) if $\bar c=0$. 
By Lemma \ref{Ricci_two_dim}, the condition (6-1) is equivalent to
\begin{align}\label{ric_sur_p}
-Ric^{\bar{D}}(u,u)\bar{J}u
=-\frac{1}{2}(s^{2}+t^{2})\bar{J}u+4\bar{a}(u,u)\bar{J}u. 
\end{align}

\noindent
Case (1): If $\delta=0$, then $\bar{a}=0$. Then we have $Ric^{\bar{D}}(u,u)=\frac{1}{2}(s^{2}+t^{2})$. 
Since $Ric^{\bar{D}}$ is $\bar{D}$-parallel, we see that $s^{2}+t^{2}$ is constant. \\
Case (1-i): If $s^{2}+t^{2}=0$, then $\bar{B}=0$ and $Ric^{\bar{D}}=0$. Take a trivial bundle 
$U \times S^{1}$ with the trivial connection and $\gamma=0$. The conditions (5) and (6-3) are satisfied. 
In particular $\bar{D}$ is flat. \\
Case (1-ii): We consider $s^{2}+t^{2} \neq 0$. 
Then the one-form 
\begin{align*}
\gamma=-\frac{3}{2} \theta^{2}_{1}-\frac{1}{2} \frac{s dt-tds}{s^{2}+t^{2}} 
\end{align*}
is the unique one-form satisfying (6-3) for $\bar B$ defined in \eqref{def_s_t}. By Lemma \ref{str._eq}
\begin{align*}
d\gamma=-\frac{3}{2} d \theta^{2}_{1}= \frac{3}{2} Ric^{\bar{D}}(u,u) \Omega. 
\end{align*}
Therefore the condition (5) $d\gamma=2(\bar{a}_{\bar{J}})=0$ holds 
if and only if  $Ric^{D}=0$. 
However this means that $s^{2}+t^{2}=0$, that is, this case does not occur. \\

\noindent
Case (2): If $\delta \neq 0$, then \eqref{ric_sur_p} is equivalent to 
\begin{equation}\label{3_10_prime:eq}
-Ric^{\bar{D}}=-\frac{1}{2}(s^{2}+t^{2}) \delta \,  \bar{a}+4\bar{a}. 
\end{equation}
Note that $s^{2}+t^{2}$ needs to be constant since $\bar{a}$ and $Ric^{\bar{D}}$ are 
$\bar{D}$-parallel. \\
\noindent
Case (2-i): If $s^{2}+t^{2} = 0$, then $\bar{B}=0$. 
In that case, we have $\bar{a}=-\frac{1}{4} Ric^{\bar{D}}$ for both $\delta=1$ and $-1$. 
Since $Ric^{\bar{D}}_{\bar{J}}$ is closed, there exists a $1$-form $\gamma$ such that 
$d \gamma=2(\bar{a}_{\bar{J}})$ locally, which establishes condition (5) in Definition \ref{intrinsic_2}. \\
Case (2-ii): We consider the case $s^{2}+t^{2} \neq 0$. 
By Lemma \ref{sol_gam}, 
\begin{align*}
\gamma=-\frac{3}{2} \theta^{2}_{1}-\frac{1}{2} \frac{s dt-tds}{s^{2}+t^{2}} 
\end{align*}
satisfies (6-3). By Lemma \ref{str._eq}, we have 
\begin{align*}
d\gamma=-\frac{3}{2} d \theta^{2}_{1}= \frac{3}{2} Ric^{\bar{D}}(u,u) \Omega
=-\frac{3}{2} Ric^{\bar{D}}(u,u)  \delta \, \bar{a}_{\bar{J}}. 
\end{align*}
Therefore $d\gamma=2(\bar{a}_{\bar{J}})$ if and only if 
$-\frac{3}{2}  \delta \, Ric^{\bar{D}}(u,u) =2$. 
By \eqref{3_10_prime:eq}, this implies $0<s^{2}+t^{2}=\frac{16}{3} \delta$. 
Hence it should be $\delta=1$. 
Therefore $d\gamma=2(\bar{a}_{\bar{J}})$ if and only if  $d\gamma_2=-\frac{3}{2} Ric^{\bar{D}}_{\bar J}$.
This finishes the proof. 
\end{proof}

\begin{remark}\label{delta_rem}
{\rm 
We see that 
(1) $\bar{a}>0$ and $Ric^{\bar{D}}<0$ when $\delta=1$, 
(2) $\bar{a}<0$ and $Ric^{\bar{D}}>0$ when $\delta=-1$ and 
(3) $\bar{a}=0$ and $Ric^{\bar{D}}=0$ when $\delta=0$. 
}
\end{remark}

\begin{remark}
{\rm
A construction of a conical special K{\"a}hler manifold over a surface of constant 
negative Gaussian curvature 
is given in \cite{MS}, in which there also occur two possibilities. 
See also \cite[Example 2.5]{MS1}. These can be also found by specializing the classification \cite{AC} to the case of complex dimension $1$, which  yields 
precisely the projective special K\"ahler manifolds $\mathrm{SU}(1,1)/\mathrm{S(U(1)\times U(1))}=\mathbb{C}H^1$ and 
$\mathrm{SL}(2,\mathbb{R})/\mathrm{SO}(2)$ associated with the Cartan types $A_2$ and $G_2$, respectively, in terms of the correspondence explained in \cite{AC}. Both have constant negative curvature but are distinguished by the value of their curvature and their origin in a linear or nonlinear Lagrangian cone, respectively.  
}
\end{remark}

\begin{remark}
{\rm
If $\bar{M}$ is simply connected and $\bar{D}$ is complete, then
$(\bar{M},\bar{J},\bar{D})$ is holomorphically and affinely isomorphic to $\mathbb{C}$, 
$\mathbb{C}H^{1}$ or $\mathbb{C}P^{1}$ 
with the standard complex structure and connection. 
By Lemma \ref{rank_ricci}, $Ric^{\bar{D}}$ is definite or vanishes.   
When $Ric^{\bar{D}}=0$, then $\bar{D}$ is the standard flat complex complete connection, so $M$ is isomorphic to 
$\mathbb{C}$.  
When $Ric^{\bar{D}}$ is positive (resp. negative) definite, then $Ric^{\bar{D}}$ 
(resp. $-Ric^{\bar{D}}$) is a Riemannian metric with constant Gaussian curvature $1$ (resp. $-1$). 
Since $\pm Ric^{\bar{D}}$ is homothetic to $\bar{a}$, then 
$\bar{D}=\bar{D}^{\bar{a}}=\bar{D}^{\pm Ric^{\bar{D}}}$. 
Because $\bar{D}$ is complete, then the Riemannian metrics $\pm Ric^{\bar{D}}$ is complete. 
Consequently $\bar{M}$ admits a complete (positive definite) metric of constant Gaussian curvature $-1$ or $1$
and this shows that $\bar{M}$ is isomorphic to $\mathbb{C}H^{1}$ or $\mathbb{C}P^{1}$.  
}
\end{remark}

\section{An application to the generalized rigid c-map}
In this section, 
we construct a non-trivial 
$(\mathbb{R}/\pi \mathbb{Z})$-family of hypercomplex structures 
arising from a generalization of the rigid c-map on the tangent bundle of the total space of a 
$\mathbb{C}^*$-bundle over a complex manifold endowed with 
a c-projective structure of a specific type.

\subsection
{An $S^1$-family of hypercomplex structures obtained from 
a circle bundle of PSCB type}
\setcounter{equation}{0}

Let $(M,J,\nabla)$ be a special complex manifold.
We can define the $\nabla$-horizontal lift $X^{h_{\nabla}}$
and the vertical lift $X^{v}$
of $X \in \Gamma(TM)$. See \cite{Bl} for example. 
The $C^\infty (TM)$-module $\Gamma (T(TM))$ 
is generated by vector fields of the form $X^{h_{\nabla}}+Y^{v}$, 
where $X$, $Y \in \Gamma(TM)$.
On $TM$, we define a triple of $(1,1)$-tensors 
$(I_{1},I_{2},I_{3})$ by 
\begin{align}
I_{1}(X^{h_{\nabla}}+Y^{v}) &=(JX)^{h_{\nabla}}-(JY)^{v}, \label{I_1}\\
I_{2}(X^{h_{\nabla}}+Y^{v}) &= Y^{h_{\nabla}}-X^{v}, \label{I_2}\\
I_{3}(X^{h_{\nabla}}+Y^{v}) &= (J Y)^{h_{\nabla}}+(JX)^{v} \label{I_3}
\end{align}
for $X^{h_{\nabla}}+Y^{v} \in T(TM)$. 
Note that $(I_{1},I_{2},I_{3})$ is a hypercomplex structure (\cite{CH}). 
There exists a unique torsion-free connection for which the hypercomplex structure is parallel (\cite{O}). 
Such a connection is called the Obata connection on a hypercomplex manifold 
and its explicit expression is given in \cite{AM}.  
It is known that the cotangent bundle of a special K{\"a}hler 
manifold carries a (pseudo-)hyperK{\"a}hler structure (\cite{CFG}). 
We recall a generalization of the rigid c-map as follows.

\begin{theorem}[Generalized rigid c-map \cite{CH}]\label{ricci_flat} 
The tangent bundle of any special complex manifold $(M,J,\nabla)$ 
carries a canonical hypercomplex structure defined by $($$\ref{I_1}$$)$-$($$\ref{I_3}$$)$, and 
the Obata connection of the hypercomplex manifold $(TM, (I_{1},I_{2},I_{3}))$ 
is Ricci flat.
\end{theorem}

\begin{remark}
{\rm
It is known that the Ricci curvature of the Obata connection on 
a hypercomplex manifold of dimension $4n$ ($n \geq 2$)
vanishes if and only if its restricted holonomy group is contained in $\mathrm{SL}(n,\mathbb{H})$. 
}
\end{remark}

Let $(M,J,\nabla,\xi)$ be a conical special manifold. Consider 
$e^{tJ}=(\cos t)\mathrm{Id}+(\sin t)J$ 
for $t \in \mathbb{R}$ and a connection $\nabla^{t}$ defined by  
\begin{align}\label{fam_sp_conn}
\nabla^{t} :=  e^{tJ} \nabla e^{-tJ}. 
\end{align}

\begin{lemma}\label{t_coni}
$(M,J,\nabla^{t},\xi)$ is a conical special complex manifold for all $t \in \mathbb{R}/\pi \mathbb{Z}$. 
\end{lemma}

\begin{proof}
By \cite[Corollary 2]{ACD}, we see that $\nabla^{t}$ is a special connection. 
It follows that $\nabla^{t} \xi=\mathrm{Id}$ from $\nabla^{t} =\nabla-(\sin t) e^{tJ} (\nabla J)$. 
\end{proof}

\begin{lemma}\label{indep_t}
Let $D^{t}$ be a torsion-free complex connection defined from $\nabla^{t}$ by (\ref{connection}).   
We have $D^{t}=D^{0}(=D)$ for all $t \in \mathbb{R}/\pi \mathbb{Z}$. 
\end{lemma}

\begin{proof}
Since $\nabla^{t}J=(\mathrm{Id}+2 (\sin t) e^{tJ}J)(\nabla J)$, we obtain
\begin{align*}
D^{t}&=\nabla^{t}-\frac{1}{2}J (\nabla^{t} J)
=\nabla-(\sin t) e^{tJ} (\nabla J)-\frac{1}{2}J (\mathrm{Id}+2 (\sin t) e^{tJ}J)   (\nabla J) \\
&=\nabla-\frac{1}{2}J (\nabla J)=D. \qedhere
\end{align*}
\end{proof}

\begin{proposition}\label{indep_D}
Let $\{ \nabla^{t} \}_{t \in \mathbb{R}/\pi \mathbb{Z}}$ be the family of special connections 
on $(M,J,\xi)$ with $\nabla^{0}=\nabla$ defined as \eqref{fam_sp_conn}. 
If $\pi:M \to \bar{M}$ is a principal $\mathbb{C}^{\ast}$-bundle for which $\bar{M}$ 
is a projective special complex manifold, 
then the canonical c-projective structure $\bar{\mathcal{P}}(\pi)$ on 
$\bar{M}$ is independent of the choice of $t \in \mathbb{R}/\pi \mathbb{Z}$. 
\end{proposition}

\begin{proof}
It follows immediately from Lemmas \ref{t_coni} and \ref{indep_t}. 
\end{proof}

Proposition \ref{indep_D} gives an application to the generalized rigid c-map.  

\begin{theorem}\label{app_cmap}
Let $(\bar M , \bar J, \mathcal{P})$ be a complex manifold of $\dim \bar{M}=2n \geq 4$ with 
a c-projective structure $\mathcal{P}$ and 
$\pi_{S}:S \to \bar{M}$ be a principal $S^{1}$-bundle with a connection 
$\eta$ of PSCB type. 
Then there exists 
a $(\mathbb{R}/\pi \mathbb{Z})$-family $\{ (I^{t}_{1},I^{t}_{2},I^{t}_{3}) \}_{t \in \mathbb{R}/\pi \mathbb{Z}}$ of hypercomplex structures 
on the $(4n+4)$-dimensional manifold $T(S \times \mathbb{R}^{>0})$
 derived from $\mathcal{P}$ and $\pi_{S}$, 
all of whose Ricci curvatures of the Obata connections vanish. 
In addition, if $[\bar{B}^{\alpha},[\bar{B}^{\alpha},\bar{B}^{\alpha}]] \neq 0$ on $\bar{U}_{\alpha} \subset M$ 
for some $\alpha \in \Lambda$, then
the induced quaternionic structure $Q^{t}=\langle I^{t}_{1},I^{t}_{2},I^{t}_{3} \rangle$ varies
for each $t \in \mathbb{R}/\pi \mathbb{Z}$. 
\end{theorem}

\begin{proof}
By Theorem \ref{char:thm}, we see that $S \times \mathbb{R}^{>0}$ is a conical special complex manifold. 
By virtue of Theorem \ref{ricci_flat} and Proposition \ref{indep_D}, 
the $(4n+4)$-dimensional manifold $T(S \times \mathbb{R}^{>0})$ admits 
a $(\mathbb{R}/\pi \mathbb{Z})$-family of hypercomplex structures 
$\{ (I^{t}_{1},I^{t}_{2},I^{t}_{3}) \}_{t \in \mathbb{R}/\pi \mathbb{Z}}$ 
whose Ricci curvatures of the Obata connections are flat. 
Take arbitrary $t, t^{\prime} \in \mathbb{R}/\pi \mathbb{Z}$ and we may assume that $t=0$. 
From Remark 8.9 in \cite{CH}, the quaternionic Weyl curvature $W^{Q^{t^{\prime}}}$ of 
$Q^{t^{\prime}}= \langle I^{t^{\prime}}_{1},I^{t^{\prime}}_{2},I^{t^{\prime}}_{3} \rangle$
satisfies
\[ 
W^{Q^{t^{\prime}}}_{U^{v},V^{v}}X^{h_{\nabla^{t^{\prime}}}}
=W^{Q^{0}}_{U^{v},V^{v}}X^{h_{\nabla^{t^{\prime}}}}
+\frac{1}{4} (\sin t^{\prime})e^{t^{\prime}J}([[A_{U},A_{V}],A_{u}]X)^{v}
\]
at $u \in T(S \times \mathbb{R}^{>0})$, where $X^{h_{\nabla^{t^{\prime}}}}$ is the horizontal lift 
with respect to $\nabla^{t^{\prime}}$ and $A=\nabla^{0} J$. 
Note that the quaternionic Weyl curvature is given in \cite{AM}. 
If $[\bar{B}^{\alpha},[\bar{B}^{\alpha},\bar{B}^{\alpha}]] \neq 0$ for some $\alpha$, then $[[A,A],A] \neq 0$.  
This shows that $Q^{t} \neq Q^{t^{\prime}}$ if $t \neq t^{\prime}$ in $\mathbb{R}/\pi \mathbb{Z}$. 
\end{proof}

The following corollary is a direct consequence of Corollary \ref{app_1} and Theorem \ref{app_cmap}.

\begin{corollary}\label{app3}
Let $(\bar{M},\bar{J},\bar{D})$ be a complex manifold with $\dim \bar{M} \geq 4$ and 
a torsion-free complex connection $\bar{D}$. We assume that all conditions in Corollary \ref{app_1} hold.  
If $[\bar{B},[\bar{B},\bar{B}]] \neq 0$, then we can obtain a non-trivial $(\mathbb{R}/\pi \mathbb{Z})$-family of hypercomplex structures on $T(\bar{M} \times S^{1} \times \mathbb{R}^{>0})$
derived from the c-projective structure $[ \bar{D} ]$,  
all of whose Ricci curvatures of the Obata connections vanish. 
\end{corollary}

We give an example of Corollary \ref{app3} in Section \ref{loc_ex_TU}.

\begin{remark}
{\rm 
Let $(\bar M , \bar{J}, [\bar{D}])$ 
be a complex manifold endowed with the flat c-projective 
structure induced by a torsion-free complex connection $\bar{D}$ and 
$\dim \bar{M} =2n \geq 4$. Additionally we assume that $C^{\bar{D}}=0$ if $2n=4$.  
As stated in Corollary \ref{cflat_ex}, 
if $[\frac{1}{2 \pi} (P^{\bar{D}}_{\bar{J}})^{a}] \in H^{2}(\bar{M},\mathbb{Z})$, 
then $\bar{M}$ can be realized as a projective 
special complex manifold via a $\mathbb{C}^{\ast}$-quotient 
of a trivial conical special complex manifold $(M,J,\nabla,\xi)$. 
Since the conical special complex manifold $(M,J,\nabla,\xi)$ is trivial, that is, $A=\nabla J=0$, 
we see that $R^{\tilde{\nabla}^{0}}=W^{Q}=0$ for $(TM,(I_{1},I_{2},I_{3}))$. 
Here we refer to equations derived in the proof of \cite[Theorem 6.5]{CH}. 
By applying the H/Q-correspondence to 
$(TM,(I_{1},I_{2},I_{3}))$ with the data $P={\pi_{TM}}^{\#} \pi^{\#}S$, 
$\Theta=-(\pi \circ \pi_{TM})^{\ast}(P^{\bar{D}}_{\bar{J}})^{a}$, $f=f_{1}=1$, 
we obtain a quaternionic manifold $\overline{TM}$. 
See Theorem 8.3 and Remark 8.12 in \cite{CH} and the diagram below. 
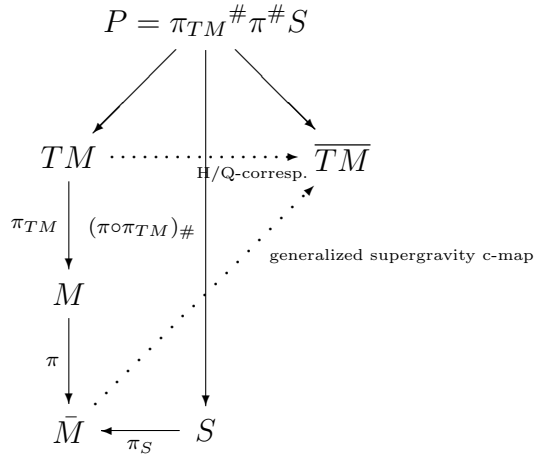
\begin{eqnarray*}
  \begin{diagram}[H]
  \node[2]{P={\pi_{TM}}^{\#} \pi^{\#}S}
  \arrow[1]{sw}
  \arrow[3]{s,l,2}{(\pi \circ \pi_{TM}){}_{\#}} 
  \arrow[1]{se}\\
  \node[1]{TM}
  \arrow[2]{e,b,3,..}{\tiny \mbox{H/Q-corresp.}}
  \arrow[1]{s,l}{\pi_{TM}}
  \node[2]{\overline{TM}}\\
  \node[1]{M}
  \arrow[1]{s,l}{\pi}\\
  \node[1]{\bar{M}}
  \arrow[2]{ne,r,3,..}{\tiny \mbox{generalized supergravity c-map}}
  \node[1]{S}
  \arrow[1]{w,b}{\pi_{S}}
\end{diagram}
\end{eqnarray*}
If we choose a different data $P=TM \times S^{1}$, $\Theta=0$, 
$f=f_{1}=1$, the resulting quaternionic structure on $\overline{TM}$ is flat. 
This demonstrates that the c-projectively flat complex manifold with 
$[\frac{1}{2 \pi} (P^{\bar{D}}_{\bar{J}})^{a}] \in H^{2}(\bar{M},\mathbb{Z})$
yields the flat quaternionic manifold 
through our construction, as a consequence of the generalized supergravity c-map. 
For more details on the H/Q-correspondence and the generalized supergravity c-map, see \cite{CH}.
}
\end{remark}

\subsection{A local example}\label{loc_ex_TU}

In this subsection we present a local example which admits a non-trivial family of hypercomplex structures, as described in Corollary \ref{app3}, by demonstrating that
$[A,[A,A]] \neq 0$ for this case. 
For this purpose, recall the local expression of 
a conical special complex manifold given in \cite{ACD}. See also Section \ref{realizations}. 
Let $(\mathbb{C}^{n+1}, J)$ be the standard complex vector space 
with the standard coordinate system $(z_{0},\dots,z_{n})$.
For a holomorphic function $g$ with homogeneity of degree one,  
we consider the holomorphic $1$-form 
\[ \alpha =g dz_{0} - \sqrt{-1} \sum_{i=1}^{n} z_{i} dz_{i}  \]
on $U:=\{ z=(z_{0},z_{1},\dots,z_{n}) \in \mathbb{C}^{n+1} \mid \mathrm{Im}\, g_{0}(z) \neq 0
\,\, \mbox{and} \,\, z_{0} \neq 0\}$, 
where $g_{i}=\frac{\partial g}{\partial z_{i}}$ $(i=0,1,\dots,n)$. 
The condition $\mathrm{Im}\, g_{0} \neq 0$ implies that the one-form $\alpha$ is regular, and hence 
the $1$-form $\alpha$ induces a conical special complex structure on $U$.  
In \cite{CH}, we have the matrix representation 
\begin{align*}
A=\nabla J
=\frac{1}{\mathrm{Im} \, g_{0}}
\left( \begin{array}{ccc}
                    A_{0 } & \dots & A_{n} \\
                    0_{2}  &  \cdots & 0_{2}  \\
                    \vdots &  \ddots &  \vdots\\
                    0_{2}  &  \cdots & 0_{2} 
                    \end{array} 
             \right)
\end{align*}
of 
$A$ with respect to the frame 
\[ \left( \frac{\partial}{\partial u_{0}},  
\frac{\partial}{\partial v_{0}}, \dots ,
\frac{\partial}{\partial u_{n}},   
\frac{\partial}{\partial v_{n}} \right), \]
where
\[
z_{i}=u_{i}+\sqrt{-1}v_{i}, \,\,\,
A_{i}=
\left( \begin{array}{cc}
                    -d \, \mathrm{Re}\, g_{i} & d \, \mathrm{Im}\, g_{i} \\
                    d \, \mathrm{Im}\, g_{i} & d \, \mathrm{Re}\, g_{i}  
                    \end{array} 
             \right)\,\,\, \mbox{and} \,\,\,\, 
      0_{2}=
\left( \begin{array}{cc}
                    0 & 0 \\
                    0 & 0   
                    \end{array} \right).                      
\]
Then it holds that 
\[ [[A,A],A]=
\frac{1}{(\mathrm{Im}\, g_{0})^{3}} 
\left( \begin{array}{cccc}
                    A_{0} \wedge A_{0}\wedge A_{0} &  A_{0} \wedge A_{1}  \wedge A_{1}
                    &  \dots & A_{0} \wedge A_{n}  \wedge A_{n}\\
                     0_{2}  &  0_{2}& \cdots & 0_{2}  \\
                     \vdots & \vdots &  \ddots &  \vdots\\
                    0_{2}  &  0_{2}& \cdots & 0_{2} 
                    \end{array} 
             \right). 
\]
Here we can compute 
\[ 
A_{0} \wedge A_{i}  \wedge A_{i}
=\left(
\begin{array}{cc}
-2 d \, \mathrm{Im} \, g_{0} \wedge d \, \mathrm{Im} \, g_{i} \wedge d \, \mathrm{Re} \, g_{i}
& -2 d \, \mathrm{Re} \, g_{0} \wedge d \, \mathrm{Im} \, g_{i} \wedge d \, \mathrm{Re} \, g_{i} \\
-2 d \, \mathrm{Re} \, g_{0} \wedge d \, \mathrm{Im} \, g_{i} \wedge d \, \mathrm{Re} \, g_{i}
& 2 d \, \mathrm{Im} \, g_{0} \wedge d \, \mathrm{Im} g_{i} \, \wedge d \, \mathrm{Re} g_{i} 
\end{array}
\right).
\]
for $i=1,\dots, n$. Now we set 
\[ w_{i}:=\frac{z_{i}}{z_{0}}, \, w(k_{1},\dots,k_{n}):=w_{1}^{k_{1}} \cdots w_{n}^{k_{n}}, \, 
\bar{w}(k_{1},\dots,k_{n}):=\bar{w}_{1}^{k_{1}} \cdots \bar{w}_{n}^{k_{n}} \]
for $i=1,\dots, n$ and $k_{1} \cdots k_{n} \in \mathbb{Z}$. 
We choose  
\[ g=\frac{-\sqrt{-1}  z_{1}^{l_{1}} \cdots z_{n}^{l_{n}}}{z_{0}^{l-1}}, \]
where $l \in \mathbb{N}$ with $l>1$ and $\sum_{i=1}^{n} l_{i}=l$. Since 
\begin{align*}
g_{0} = -\sqrt{-1}\, (-l+1)\,w(l_{1},\dots,l_{n}), 
g_{i}  = -\sqrt{-1}\, l_{i} \, w(l_{1},\dots, l_{i}-1,\dots,l_{n}) \,\, (i \geq 1), 
\end{align*}
we have 
\begin{align*}
\mathrm{Re} \, g_{0} &=\frac{1}{2}(g_{0}+\bar{g}_{0})
                          =\frac{\sqrt{-1}}{2}(-l+1)(-w(l_{1},\dots,l_{n})+\bar{w}(l_{1},\dots,l_{n})), \\
\mathrm{Im} \, g_{0} &=\frac{1}{2 \sqrt{-1}}(g_{0}-\bar{g}_{0})
                          =-\frac{1}{2}(-l+1)(w(l_{1},\dots,l_{n})+\bar{w}(l_{1},\dots,l_{n})),\\
\mathrm{Re} \, g_{i} &= \frac{1}{2}(g_{i}+\bar{g}_{i})
                          =\frac{\sqrt{-1}}{2}l_{i} (-w(l_{1},\dots, l_{i}-1,\dots,l_{n})
                            +\bar{w}(l_{1},\dots, l_{i}-1,\dots,l_{n})),\\
\mathrm{Im} \, g_{i} &= \frac{1}{2 \sqrt{-1}} (g_{i}-\bar{g}_{i})
                          =-\frac{1}{2} l_{i} (w(l_{1},\dots, l_{i}-1,\dots,l_{n})
                            +\bar{w}(l_{1},\dots, l_{i}-1,\dots,l_{n})). 
\end{align*}
Their exterior derivatives are given as: 
\begin{align*}
d \, \mathrm{Re} \, g_{0} =& \frac{\sqrt{-1}}{2}(-l+1)(-dw(l_{1},\dots,l_{n})+d\bar{w}(l_{1},\dots,l_{n}))
                              \\
                           =&\frac{\sqrt{-1}}{2}(-l+1) \times \\
                             &\left( \sum_{j=1}^{n} -l_{j} w(l_{1},\dots, l_{j}-1,\dots,l_{n}) dw_{j} 
                                   +l_{j}\bar{w}(l_{1},\dots, l_{j}-1,\dots,l_{n}) d \bar{w}_{j} \right), \\
d \, \mathrm{Im} \, g_{0} =&-\frac{1}{2}(-l+1)(dw(l_{1},\dots,l_{n})+d\bar{w}(l_{1},\dots,l_{n})),\\
                             =&-\frac{1}{2}(-l+1) \times \\
                              &\left( \sum_{j=1}^{n} l_{j} w(l_{1},\dots, l_{j}-1,\dots,l_{n}) dw_{j}
                                   +l_{j}\bar{w}(l_{1},\dots, l_{j}-1,\dots,l_{n}) d \bar{w}_{j} \right), \\
d \, \mathrm{Re} \, g_{i} =&\frac{\sqrt{-1}}{2}l_{i} (-d w(l_{1},\dots, l_{i}-1,\dots,l_{n})
                            +d \bar{w}(l_{1},\dots, l_{i}-1,\dots,l_{n})) \\
                            =&\frac{\sqrt{-1}}{2}l_{i}
                             \bigg( \sum_{j=1, j \neq i}^{n} -l_{j} w(l_{1},\dots, l_{i}-1,\dots, l_{j}-1,\dots,l_{n}) dw_{j} \\
                             &+l_{j} \bar{w}(l_{1},\dots, l_{i}-1,\dots, l_{j}-1,\dots,l_{n}) d \bar{w}_{j} \\
                             &-(l_{i}-1)w(l_{1},\dots, l_{i}-2,\dots,l_{n}) dw_{i} 
                             +(l_{i}-1) \bar{w}(l_{1},\dots, l_{i}-2,\dots,l_{n}) d \bar{w}_{i} \bigg), \\
d \, \mathrm{Im} \, g_{i} =&-\frac{1}{2} l_{i} (dw(l_{1},\dots, l_{i}-1,\dots,l_{n})
                            +d\bar{w}(l_{1},\dots, l_{i}-1,\dots,l_{n})) \\
                            =&-\frac{1}{2} l_{i} 
                             \bigg( \sum_{j=1, j \neq i}^{n} l_{j} w(l_{1},\dots, l_{i}-1,\dots, l_{j}-1,\dots,l_{n}) dw_{j} \\
                             &+l_{j} \bar{w}(l_{1},\dots, l_{i}-1,\dots, l_{j}-1,\dots,l_{n}) d \bar{w}_{j} \\
                             &+(l_{i}-1)w(l_{1},\dots, l_{i}-2,\dots,l_{n}) dw_{i} 
                             +(l_{i}-1) \bar{w}(l_{1},\dots, l_{i}-2,\dots,l_{n}) d \bar{w}_{i} \bigg). 
\end{align*}
For example, when $n=2$ and $g=-\sqrt{-1} \, z_{1}^{2} \, z_{2}/z_{0}^{2}$ $(l_{1}=2, l_{2}=1,l=3)$, 
we find that 
 $A_{0} \wedge A_{1}  \wedge A_{1} \neq 0$. This indicates that 
$[A,[A,A]] \neq 0$. To verify this, it suffices to check that the $(1,1)$-entry 
\begin{align*}
-2 d \, \mathrm{Im} \, g_{0} \wedge d \, \mathrm{Im} \, g_{1} \wedge d \, \mathrm{Re} \, g_{1}
\end{align*}
of $A_{0} \wedge A_{1}  \wedge A_{1} $ does not vanish, for instance.  
Indeed, by using  
\begin{align*}
d \, \mathrm{Im} \, g_{0} =&-\frac{1}{2}(-3+1)( 2w(1,1)dw_{1}+w(2,0)dw_{2} 
+2\bar{w}(1,1)d \bar{w}_{1}) +w(2,0) d \bar{w}_{2}) \\
=&2w_{1}w_{2} dw_{1}+(w_{1})^{2} dw_{2}+2\bar{w}_{1}\bar{w}_{2} d\bar{w}_{1}+(\bar{w}_{1})^{2} d\bar{w}_{2}, \\
d \, \mathrm{Re} \, g_{1} =&\frac{\sqrt{-1}}{2} 2 
(-w(1,0)dw_{2}+\bar{w}(1,0)d\bar{w}_{2}  
-(2-1) w(2,-1) dw_{1}+(2-1) w(2,-1)d \bar{w}_{1}) \\
=&\sqrt{-1} \left( -w_{1} dw_{2}+\bar{w}_{1} d \bar{w}_{2}-\frac{(w_{1})^{2}}{w_{2}} dw_{1} 
+ \frac{(\bar{w}_{1})^{2}}{\bar{w}_{2}} d \bar{w}_{1} \right), \\
d \, \mathrm{Im} \, g_{1} =& -\frac{1}{2} 2 (w(1,0)dw_{2}+\bar{w}(1,0)d\bar{w}_{2}  
+(2-1) w(2,-1) dw_{1}+(2-1) w(2,-1)d \bar{w}_{1} )\\
=&- \left( w_{1} dw_{2}+\bar{w}_{1} d \bar{w}_{2}+\frac{(w_{1})^{2}}{w_{2}} dw_{1} 
+ \frac{(\bar{w}_{1})^{2}}{\bar{w}_{2}} d \bar{w}_{1} \right), 
\end{align*}
we can see that 
\begin{align*}
&\sqrt{-1} (d \, \mathrm{Im} \, g_{1} \wedge d \, \mathrm{Re} \, g_{1}) \\
&=2 \left(
|w_{1}|^{2} dw_{2} \wedge d\bar{w}_{2}+
\frac{ |w_{1}|^{2} w_{1} } {w_{2}} dw_{1} \wedge d\bar{w}_{2}
+\frac{ |w_{1}|^{4} } {|w_{2}|^{2}} dw_{1} \wedge d\bar{w}_{1}
+\frac{ |w_{1}|^{2} \bar{w}_{1} } { \bar{w}_{2} } dw_{2} \wedge d\bar{w}_{1}
\right)
\end{align*}
so 
\begin{align*}
\frac{1}{2}\sqrt{-1} 
(d \, \mathrm{Im} \, g_{0} \wedge & d \, \mathrm{Im} \, g_{1} \wedge d \, \mathrm{Re} \, g_{1}) \\
=&
\left( 2w_{1}w_{2} |w_{1}|^{2}-\frac{|w_{1}|^{2} (w_{1})^{3}}{w_{2}} \right) dw_{1} \wedge dw_{2} \wedge d \bar{w}_{2} \\
&-\left( 2\bar{w}_{1} \bar{w}_{2} |w_{1}|^{2}-\frac{|w_{1}|^{2} (\bar{w}_{1})^{3}}{ \bar{w}_{2}} \right) 
d \bar{w}_{1} \wedge d \bar{w}_{2} \wedge d w_{2} \\
&+\frac{| w_{1} |^{4} \bar{w}_{1}}{ | w_{2} |^{2}} d \bar{w}_{1} \wedge d w_{1} \wedge d \bar{w}_{2}  
-\frac{| w_{1} |^{4} w_{1}}{ | w_{2} |^{2}} d w_{1} \wedge d \bar{w}_{1} \wedge d w_{2}.  
\end{align*}
Hence $d \, \mathrm{Im} \, g_{0} \wedge d \, \mathrm{Im} \, g_{1} \wedge d \, \mathrm{Re} \, g_{1}$ 
does not vanish. So we can summarize the results as follows.

\begin{example}\label{12_dim_ex}
Let $U \subset \mathbb{C}^{3}$ be the conical special complex manifold given 
by the holomorphic $1$-form
\[ \alpha =g dz_{0} - \sqrt{-1} \sum_{i=1}^{2} z_{i} dz_{i}, \,\,  g=-\frac{\sqrt{-1} \, z_{1}^{2} \, z_{2}}{z_{0}^{2}}. \]
We can associate $(\bar{U},[\bar{D}])$ to a $12$-dimensional manifold 
$(TU, 
\{ (I^{t}_{1},I^{t}_{2},I^{t}_{3}) \}_{t \in \mathbb{R}/\pi \mathbb{Z}})$ 
with a non-trivial 
$(\mathbb{R}/\pi \mathbb{Z})$-family of hypercomplex structures via the generalized rigid c-map.  
\end{example}

\subsection{Characterization of case of flat quaternionic structure on the tangent bundle 
of a conical special complex manifold}
The curvature of the Obata connection $\tilde{\nabla}^{0}$ on $(TM,(I_{1},I_{2},I_{3}))$ for a special complex manifold 
$M$ was obtained in the proof of \cite[Theorem 6.5]{CH}. It is Ricci flat and the quaternionic Weyl curvature $W^{Q}$ of $Q=\langle I_{1},I_{2},I_{3} \rangle$ is in general not zero but satisfies
\[ W^{Q}_{X^{h_{\nabla}},Y^{h_{\nabla}}  } Z^{h_{\nabla}} \\
=R^{\tilde{\nabla}^{0}} _{X^{h_{\nabla}},Y^{h_{\nabla}}  } Z^{h_{\nabla}}
= -\frac{1}{4}([A_{X}, A_{Y}] Z )^{h_{\nabla}} \]
for tangent vectors $X$, $Y$, $Z$ on $M$. 
For instance, 
the quaternionic structure on $TU$ of 
Example \ref{12_dim_ex} 
is not flat, indeed $[A,[A,A]] \neq 0$.

It is natural to ask for relations between properties of the quaternionic structure 
$Q$ on $TM$ and properties of the canonical c-projective 
structure $\bar{\mathcal{P}}(\pi)$ on $\bar{M}$, since these structures are connected by the fibrations 
\[ (TM, Q=\langle I_{1},I_{2},I_{3} \rangle) \to (M,J,\nabla,\xi) \to (\bar{M},\bar{J}, \bar{\mathcal{P}}(\pi)). \]
In this section we study the case when the quaternionic structure 
of $(TM,(I_{1},I_{2},I_{3}))$ is flat. We show in the next theorem that it 
implies the vanishing of the Weyl tensor of the c-projective structure on $\bar M$.

\begin{theorem}\label{flatness_q_c}
Let $(M,J,\nabla,\xi)$ be a conical special complex manifold fibering over the projective special complex manifold 
$(\bar{M},\bar{J},\bar{\mathcal{P}}(\pi))$ with the corresponding principal $\mathbb{C}^{\ast}$-bundle $\pi : M \to \bar{M}$ and the canonical c-projective structure 
$\bar{\mathcal{P}}(\pi)$. 
If $\dim \bar{M} \geq 4$ and 
$W^{Q}=0$ for the quaternionic structure $Q=\langle I_{1},I_{2},I_{3} \rangle$ on $TM$, 
then we have $W^{ \bar{\mathcal{P}}(\pi)}=0$ and $\bar{\cal B}=0$.  
\end{theorem}

\begin{proof}
Take an open covering
$\{ \bar{U}_{\alpha} \}_{\alpha \in \Lambda}$ of 
$\bar{M}$ with trivializations $\pi^{-1}(\bar{U}_{\alpha})\cong \bar{U}_{\alpha}\times \mathbb{C}^*$. 
Let $X$, $Y$, $Z$ be vector fields on $\bar{M}$. 
We denote the horizontal lift of $X\in \Gamma (T\bar M)$ by $\tilde{X}\in \Gamma (TM)$ with respect to 
a principal connection on $M$ of type $(1,0)$ and note that the c-projective structure 
$\bar{\mathcal{P}}(\pi)$ 
is independent of the choice of the connection form as shown in Theorem \ref{can_c_pro}.  
On any $\bar{U}_{\alpha}$,  we have 
\begin{align*}
  &W^{Q}_{\tilde{X}^{h_{\nabla}},\tilde{Y}^{h_{\nabla}}  } \tilde{Z}^{h_{\nabla}} 
=-\frac{1}{4}([A_{\tilde{X}}, A_{\tilde{Y}}] \tilde{Z} )^{h_{\nabla}} 
=-\frac{1}{4}([B^{\alpha}_{\tilde{X}}, B^{\alpha}_{\tilde{Y}}] \tilde{Z} )^{h_{\nabla}} \\
=&\bigg[ \left( -\frac{1}{4} [\bar{B}^{\alpha}_{X}, \bar{B}^{\alpha}_{Y}]Z \right)^{\widetilde{}} 
-\frac{1}{4} \left( \bar{c}^{\alpha}(X,\bar{B}^{\alpha}_{Y}Z) 
-\bar{c}^{\alpha} (Y,\bar{B}^{\alpha}_{X}Z) \right) \xi \\
&-\frac{1}{4} \left( \bar{c}^{\alpha}_{\bar{J}}(X,\bar{B}^{\alpha}_{Y}Z) 
-\bar{c}^{\alpha}_{\bar{J}} (Y,\bar{B}^{\alpha}_{X}Z) \right) Z_{M} \bigg]^{h_{\nabla}}  \\
=& \bigg[ \bigg( W^{{\bar{\mathcal{P}}(\pi)}}_{X,Y}Z
+\frac{1}{4(n+1)} \bar{\cal B}_{\bar{J}}(X,Y) \bar{J}Z 
-\frac{1}{8(n+1)} (\bar{\cal B}(X,Z)Y-\bar{\cal B}(Y,Z)X) \\
&+\frac{1}{8(n+1)} (\bar{\cal B}_{\bar{J}} (X,Z)\bar{J}Y-\bar{\cal B}_{\bar{J}} (Y,Z)\bar{J}X) 
\bigg)^{\widetilde{}}  \\
&-\frac{1}{4} \left( \bar{c}^{\alpha}(X,\bar{B}^{\alpha}_{Y}Z) 
-\bar{c}^{\alpha} (Y,\bar{B}^{\alpha}_{X}Z) \right) \xi 
-\frac{1}{4} \left( \bar{c}^{\alpha}_{\bar{J}}(X,\bar{B}^{\alpha}_{Y}Z) 
-\bar{c}^{\alpha}_{\bar{J}} (Y,\bar{B}^{\alpha}_{X}Z) \right) Z_{M} \bigg]^{h_{\nabla}}, 
\end{align*}
where we have used \eqref{BBbar:eq} and \eqref{Weyl_bar}.
Since the quaternionic structure $Q$ on $TM$ is flat, in particular, the horizontal component (with respect to the principal connection) of the term inside of 
$[ \,\,\,\,\, ]^{h_{\nabla}}$ 
in the above equation vanishes: 
\begin{align*}
W^{\bar{\mathcal{P}} (\pi) }_{X,Y}Z
+\frac{1}{4(n+1)} \bar{\cal B}_{\bar{J}}(X,Y) \bar{J}Z 
&-\frac{1}{8(n+1)} (\bar{\cal B}(X,Z)Y-\bar{\cal B}(Y,Z)X) \\
&+\frac{1}{8(n+1)} (\bar{\cal B}_{\bar{J}} (X,Z)\bar{J}Y-\bar{\cal B}_{\bar{J}} (Y,Z)\bar{J}X)=0. 
\end{align*}
Since the c-projective Weyl curvature $W^{\bar{\mathcal{P}}(\pi)}$ satisfies 
$\mathrm{Tr}\, W^{ \bar{\mathcal{P}}(\pi)}_{(\,\, \cdot \,\,),Y}Z=0$, we have 
\begin{align*}
0=\mathrm{Tr}\, W^{ \bar{\mathcal{P}}(\pi)}_{(\,\, \cdot \,\,),Y}Z+\frac{1}{4}\bar{\cal B}(Y,Z)=\frac{1}{4}\bar{\cal B}(Y,Z).
\end{align*}
Then we obtain $W^{ \bar{\mathcal{P}}(\pi)}=0$ and  $\bar{\cal B}=0$. This completes the proof.
\end{proof}

\begin{remark}
{\rm 
The flatness $W^{Q}=0$ of the quaternionic structure on $TM$ is equivalent to 
$R^{\tilde{\nabla}^{0}}=0$. }
\end{remark}

If  $M$ is a special K{\"a}hler manifold, the following corollary holds. 

\begin{corollary}
Let $(M,g,J,\omega,\nabla,\xi)$ be a conical special K{\"a}hler manifold. 
If $W^{Q}=0$ for the quaternionic structure $Q=\langle I_{1},I_{2},I_{3} \rangle$ on $TM$, then 
the projective special K{\"a}hler manifold $(\bar{M},\bar{J},\bar{g})$ is of negative constant holomorphic sectional curvature $-4$ and the conical special K{\"a}hler manifold $M$ is trivial.  
\end{corollary}

\begin{proof}
If $\dim \bar{M} \geq 4$, we have $\bar{\cal B}=0$ by Theorem \ref{flatness_q_c}. 
The vanishing of $\bar{\cal B}$ implies that 
equality holds in \eqref{ineq_321}, thus indicating that 
$\bar{g}$ has constant holomorphic sectional 
curvature $-4$. 
Moreover $\bar{\cal B}=0$ means that $\bar{B}^{\alpha}=0$ for all $\alpha \in \Lambda$. 
Since a conical special K{\"a}hler manifold is of horizontal type, which means $\bar{c}^{\alpha}=0$, 
then we also have $B^{\alpha}=0$ for all $\alpha \in \Lambda$. 
This shows that $A=0$, hence $M$ is trivial. 
If $\dim \bar{M}=2$, we obtain 
$[\bar{B}^{\alpha}, \bar{B}^{\alpha}]=0$ 
in a similar calculation as above. From \eqref{curv}, we have 
\[ 
R^{\bar{D}^{\bar{g}}}=
 - 2 \bar{g}_{\bar{J}}  \otimes \bar{J} 
+ \bar{g} \wedge \mathrm{Id} - \bar{g}_{\bar{J}} \wedge \bar{J}, \] 
which means that $\bar{g}$ has the constant Gaussian curvature $-4$. 
Since $[\bar{B}^{\alpha},\bar{B}^{\alpha}]=0$, taking its trace shows that $\bar{\cal B}=0$. 
Therefore $M$ is trivial. 
\end{proof}

\begin{remark}
{\rm 
It is known that a K{\"a}hler metric  
whose Levi-Civita connection is\linebreak[4] c-projectively flat  
has a constant holomorphic sectional curvature (\cite{CEMN}). 
}
\end{remark}

\noindent
{\bf Acknowledgments.} 
Research by the first author is partially funded by the Deutsche Forschungsgemeinschaft (DFG, German Research 
Foundation) under Germany's
Excellence Strategy -- EXC 2121 Quantum Universe -- 390833306 
and under SFB-Gesch\"afts\-zeichen 1624 -- Projektnummer 506632645.  
The second author's research is partially supported by 
JSPS KAKENHI Grant Number 18K03272.


\noindent
Vicente Cort{\' e}s \\
Department of Mathematics \\
and Center for Mathematical Physics \\
University of Hamburg \\
Bundesstra\ss e 55, \\
D-20146 Hamburg, Germany. \\ 
email:vicente.cortes@uni-hamburg.de \\

\noindent
Kazuyuki Hasegawa \\
Faculty of teacher education \\
Institute of human and social sciences \\
Kanazawa university \\
Kakuma-machi, Kanazawa, \\
Ishikawa, 920-1192, Japan. \\
e-mail:kazuhase@staff.kanazawa-u.ac.jp

\end{document}